\documentclass[11pt, letter]{article}
         
\usepackage{amsfonts,amssymb, latexsym, amsmath, amsthm,mathrsfs, verbatim,geometry}
\usepackage{graphics}
\usepackage{slashed}
\usepackage{url,color}
\usepackage{enumerate}
\usepackage{esint}
\usepackage{pifont}
\usepackage{upgreek}
\usepackage{times}
\usepackage{calligra}
\usepackage{graphicx}
\usepackage{caption}
\usepackage{tikz}
\usepackage{accents}
\usetikzlibrary{calc}
\allowdisplaybreaks

\newtheorem{theorem}{Theorem}[section]

\newtheorem{proposition}[theorem]{Proposition}

\newtheorem{cor}[theorem]{Corollary}
\newtheorem{lemma}[theorem]{Lemma}
\newtheorem{remark}[theorem]{Remark}
\newtheorem*{proposition*}{Proposition}
 
\renewcommand{\theequation}{\arabic{section}.\arabic{equation}}

\def\bcb{\begin{color}{blue}}
\def\bcr{\begin{color}{red}} 
\def\bcv{\begin{color}{violet}} 
\def\ec{\end{color}}

\title{Global existence of the nonisentropic compressible Euler equations with vacuum boundary surrounding a variable entropy state}                

\author{Calum Rickard\footnote{Department of Mathematics, University of Southern California, Los Angeles, USA}, \ Mahir Had\v zi\'c\footnote{Department of Mathematics, University College London, 25 Gordon Street, London, WC1E 6XA, UK}, \
Juhi Jang\footnote{Department of Mathematics, University of Southern California, Los Angeles, USA and Korea Institute for Advanced Study, Seoul, Republic of Korea}}

\date{}
\begin{document}

\maketitle  

\abstract{
Global existence for the nonisentropic compressible Euler equations with vacuum boundary for all adiabatic constants $\gamma > 1$ is shown through perturbations around a rich class of background nonisentropic affine motions. The notable feature of the nonisentropic motion lies in the presence of non-constant entropies, and it brings a new mathematical challenge to the stability analysis of nonisentropic affine motions. In particular, the estimation of the curl terms requires a careful use of algebraic, nonlinear structure of the pressure. With suitable regularity of the underlying affine entropy, we are able to adapt the weighted energy method developed for the isentropic Euler \cite{1610.01666} to the nonisentropic problem. For large $\gamma$ values, inspired by \cite{shkoller2017global}, we use time-dependent weights that allow some of the top-order norms to potentially grow as the time variable tends to infinity. We also exploit coercivity estimates here via the fundamental theorem of calculus in time variable for norms which are not top-order.
}   

\section{Introduction}
We consider compressible Euler equations for ideal gases in three space dimensions
\begin{align} 
\rho(\partial_t \mathbf{u} + \mathbf{u} \cdot \nabla \mathbf{u}) + \nabla p  & = 0, \label{E:MOM} \\
\partial_t \rho + \text{div}(\rho\mathbf{u})   & = 0, \label{E:M} \\ 
\partial_t \epsilon + \mathbf{u} \cdot \nabla \epsilon + (\gamma-1) \epsilon \, \text{div}(\mathbf{u}) & = 0, \label{E:EE}
\end{align}                 
where $\mathbf{u}$ is the fluid velocity vector field, $\rho$ is the density, $\epsilon$ is the internal energy, $p$ is the pressure and $\gamma > 1$ is the adiabatic constant. Coupled with the equation of state for an ideal gas
\begin{equation} 
p(\rho,\epsilon)=(\gamma-1)\rho\epsilon, \label{E:EEOS}
\end{equation}
equations (\ref{E:MOM})-(\ref{E:EEOS}) describe the compressible flow of an inviscid, non-conducting and adiabatic gas.

It is often convenient to use another unknown - the entropy $S$ - instead of the internal energy. Equation (\ref{E:EE}) is then equivalently replaced by  
\begin{align}
\partial_t S + \mathbf{u} \cdot \nabla S & = 0, \label{E:S}
\end{align}
and the equation of state reads 
\begin{equation}  
p(\rho,S)=\rho^{\gamma} e^{S}. \label{E:SEOS} 
\end{equation}
Note that the entropy is just transported by the flow, and therefore the entropy formulation will be in particular useful in Lagrangian coordinates. See Section \ref{S:LAGR}. 

A special case where the entropy $S$ remains constant represents the isentropic process and in that case, the equation of state relates the pressure of the gas to the density only: $p=\rho^{\gamma}$ and the unknown variables for the system are the density and velocity. In this article, we are interested in the dynamics of the nonisentropic gas whose entropy changes in both time and space but is finite. 
         
We study the vacuum free boundary problem: that is, we consider the Euler equations (\ref{E:MOM})-(\ref{E:EE}) and (\ref{E:S}) in the following time dependent open bounded domain
$$ 
\Omega(t) \subset \mathbb{R}^3 \text{ with boundary } \partial \Omega(t) \text{ where } t \in [0,T] \text{ for some } T>0.
$$
The boundary conditions are then the physical vacuum boundary condition coupled with kinematic boundary condition
\begin{alignat}{2} 
p&=0 & \qquad \text{on} \enskip \partial \Omega(t), \label{E:VACUUM} \\
-\infty < \frac{\partial \epsilon}{\partial \mathbf{n}} &<0 & \qquad \text{on} \enskip \partial \Omega(t), \label{E:ENORMALDERIVATIVE} \\ 
 \mathcal{V} (\partial\Omega(t))&= \mathbf{u}\cdot \mathbf{n}(t) & \qquad \text{on} \enskip \partial \Omega(t),\label{E:VELOCITYBDRYE}       
 \end{alignat}   
with $\mathbf{n}$ the outward unit normal vector to $\partial\Omega(t)$, $\frac{\partial}{\partial \mathbf{n}}$ the outward normal derivative, and $\mathcal{V} (\partial\Omega(t))$ the normal velocity of $\partial\Omega(t)$. The condition (\ref{E:ENORMALDERIVATIVE}) is most convenient for us to express in terms of the internal energy $\epsilon$. When the entropy is bounded from below and from above, the physical vacuum condition \eqref{E:ENORMALDERIVATIVE} can be written as 
\begin{equation}\label{E:PHYSICALVACUUM2}
-\infty < \frac{\partial \rho^{\gamma-1}}{\partial \mathbf{n}}<0.
\end{equation}
%\begin{color}{red}
Physically, this condition implies a nontrivial acceleration of the gas at the boundary in the normal direction and mathematically, it  implies an inherent lack of regularity of the enthalpy at the vacuum boundary. Moreover it is clear from~\eqref{E:PHYSICALVACUUM2} that the solution is not even $C^1$ at the boundary. This is not just a mathematical curiosity, but it is in fact fundamental to give up the requirement of smoothness in order to track the evolution of compactly supported gases; we are led to study the free boundary problem and it turns out that~\eqref{E:PHYSICALVACUUM2} is a key condition to guarantee well-posedness of the free boundary. The physical vacuum condition is not merely a technical condition but it is realized for a wide range of physical systems for gaseous fluids. In particular a key impetus for the physical vacuum condition~\eqref{E:PHYSICALVACUUM2} comes from astrophysics. A famous class of equilibria of the gravitational Euler-Poisson system, known as the Lane-Emden stars, satisfy~\eqref{E:PHYSICALVACUUM2} and any rigorous theory that attempts to study the nonlinear dynamics in the vicinity of Lane-Emden stars must contend with~\eqref{E:PHYSICALVACUUM2}~\cite{Ch,GoWe,jang2014}. It also arises naturally in the context of Euler equations with damping for the gas flow through a porous medium \cite{L2,liu2000compressible,zeng2019almost}. For more detail on the physical vacuum, we refer to \cite{Coutand2012,JaMa2009,JM1,JM2012,doi:10.1002/cpa.21517,LXZ}.  
%\end{color}

Finally, we recall we are working on the time interval $[0,T]$ and we consider the initial conditions 
\begin{equation} 
(\rho(0,\cdot),\mathbf{u}(0,\cdot),
S(0,\cdot),\Omega(0))=(\rho_0,\mathbf{u}_0,S_0,\Omega_0). \label{E:IC}
\end{equation} 
Collectively, we will study the vacuum free boundary nonisentropic 
Euler system (\ref{E:MOM})-(\ref{E:IC}).

Before we move on, we briefly discuss some known results for the Euler equations. Due to vast literature, we will only mention the works relevant to the present article.  We refer to \cite{Coutand2012, 1610.01666, doi:10.1002/cpa.21517, luk2018shock} and reference therein for more thorough review. We begin with the Cauchy problem in the whole space. It is well-known that the Euler equations are hyperbolic and the existence of $C^1$  local-in-time  positive density solutions follows from the theory of symmetric hyperbolic systems \cite{kato1975cauchy,majda1984compressible}. Serre \cite{Se1997} and Grassin \cite{grassin1998global} proved global existence for a special class of initial data in the whole space by the perturbation of expansive wave solutions to the vectorial Burgers equation with linearly growing velocities at infinity - a related idea was used in the work of Rozanova~\cite{Ro}.
On the other hand, Sideris \cite{sideris1985} showed that singularities must form if the density is a strictly positive constant outside of a bounded set. A detailed description of shock formation starting with smooth initial data for irrotational relativistic fluids around is given by Christodoulou~\cite{christodoulou2007formation} for special-relativistic fluids and Christodoulou-Miao~\cite{MiCr} in the nonrelativistic case. For a more general framework covering a wider class of equations leading to shock formation see the works of Speck and Luk-Speck~\cite{luk2018shock,Sp}. Makino-Ukai- Kawashima \cite{MUK1986} proved that singularities form starting from compactly supported smooth solutions. We remark that these singularity and shock formation results do not apply to the physical vacuum free boundary problem.

In the vacuum free boundary framework, a lot of important progress has been made in the past decade. Local well-posedness for compressible Euler equations with physical vacuum has been established by Coutand-Shkoller \cite{Coutand2012} and Jang-Masmoudi \cite{doi:10.1002/cpa.21517}.  

First examples of global-in-time solutions surrounded by vacuum and satisfying the physical vacuum condition were given by Sideris~\cite{MR3634025}. These are the so-called affine motions found by a separation-of-variables ansatz for the Lagrangian flow map $\zeta(t,y)$
\begin{equation}\label{E:AFFINEANSATZ}
\zeta(t,y) = A(t)y,
\end{equation}
where $t\mapsto A(t)$ is an unknown $3\times3$ matrix. Such an ansatz severely reduces the dynamic degrees of freedom resulting in an ordinary differential equation (ODE) for the matrix $A(t)$. We will therefore refer to such solutions as ODE-type solutions. It should be noted, that the idea of 
considering the ansatz~\eqref{E:AFFINEANSATZ} was considered before in the context of nonisentropic flows and goes back to the works of Ovsyannikov \cite{ovsyannikov1956new} and Dyson \cite{10.2307/24902147}, wherein the ODE satisfied by $A(t)$ was already discovered. The affine motions constructed by Sideris importantly satisfy the physical vacuum condition~\eqref{E:PHYSICALVACUUM2}, which is a critical assumption in the general well-posedness framework developed in~\cite{Coutand2012,doi:10.1002/cpa.21517}.

%\begin{color}{red}
We remark that the ansatz (\ref{E:AFFINEANSATZ}) has been widely exploited in various physical applications beyond the solutions considered in this paper. For instance, it is explicitly seen to hold for special physical solutions to the Navier-Stokes equations;  %\cite{majda1986vorticity}
 specifically, as described in \cite{majda1986vorticity}, a jet which compresses in one coordinate direction and stretches in another, a fluid which drains purely along one direction from planes parallel to the alternate coordinate plane, and the pure rotation of a rigid body, all obey (\ref{E:AFFINEANSATZ}). Another physical application of the affine ansatz (\ref{E:AFFINEANSATZ}) can be found in the dynamics of gaseous stars governed by the Euler-Poisson system; in particular, in \cite{GoWe,Makino92}, the star collapse and expansion were demonstrated by using (\ref{E:AFFINEANSATZ}) in the radially symmetric setting. %\end{color}

In the {\em isentropic} case, the nonlinear stability of the Sideris solutions
was shown  by Had\v zi\'c-Jang \cite{1610.01666} for $\gamma \in (1,\frac53]$ and then extended to the full range $\gamma>1$ by Shkoller-Sideris~\cite{shkoller2017global}. In a recent work \cite{PHJ2019}, Parmeshwar-Had\v zi\'c-Jang showed the global existence of expanding solutions with small densities without relying on the background affine solutions, again in the class of isentropic flows. Also very recently in the isentropic setting with damping \cite{zeng2019almost}, Zeng has established the existence of almost global solutions by perturbing background approximate solutions known as Barenblatt solutions with sub-linear expansion which contrasts with the linear expansion considered otherwise.

By contrast to the isentropic case, it has remained an open question to construct open sets of initial data in the physically important {\em nonisentropic} case that lead to 
global existence in the presence of free vacuum boundaries. This is the main goal of this article. Indeed, the expansion of gas into vacuum is an important physical phenomenon \cite{greenspan1962expansion}. Our result demonstrates the linear expansion of nonisentropic gas into vacuum is a stable mechanism that avoids shock formation.

\subsection{Nonisentropic Affine Motion}\label{S:AFF}
As mentioned above, Sideris~\cite{MR3634025} constructed a family of affine motions satisfying the physical vacuum boundary condition. Such solutions are blobs of gas initially occupying $B_1(\mathbf{0})$ - the unit ball in $\mathbb{R}^3$. Their evolving support is given as the image of $B_1(\mathbf{0})$ under the matrix $A(t)$ i.e. $\Omega(t)=A(t)B_1(\mathbf{0})$, where $t\mapsto A(t)$ is an a priori unknown matrix. This generically gives us a gas supported on an evolving ellipsoid.
At the level of Lagrangian coordinates, this translates into separating variables and writing the flow map in the form~\eqref{E:AFFINEANSATZ}, see~\cite{10.2307/24902147,ovsyannikov1956new,MR3634025}. After plugging this back into the Lagrangian formulation of the problem, an algebraic manipulation leads to the following fundamental system of ODEs satisfied by $A(t)$
\begin{align}
A''(t)&= \delta (\det A(t))^{1-\gamma} A(t)^{-\top}, \label{E:AODE}\\
(A(0),A'(0),\delta) &\in \text{GL}^+(3) \times \mathbb M^{3\times3} \times \mathbb{R}_+, \label{E:FUNDAMENTALSYSTEMFIRST}
\end{align}
for fixed $\delta > 0$. In the above $\mathbb{M}^{3\times3}$ denotes the set of $3 \times 3$ matrices over $\mathbb{R}$ and $\text{GL}^+(3)=\{A \in \mathbb{M}^{3\times3} : \det A > 0\}$. With $A \in C(\mathbb{R},\text{GL}^+(3)) \cap C^\infty(\mathbb{R},\mathbb{M}^3)$ solving this system of ODEs, the associated solution of the Euler equations is given by
\begin{align}
\mathbf{u}_A(t,x)&=A'(t)A(t)^{-1}x, \label{E:AFFINEVELOCITY} \\
\rho_A(t,x)&=\accentset{\circ}{\rho}(|A(t)^{-1}x|)/(\det A(t)), \label{E:AFFINEDENSITY}\\
\epsilon_A(t,x)&=\accentset{\circ}{\epsilon} (|A(t)^{-1}x|)/(\det A(t))^{\gamma-1}.\label{E:AFFINEINTERNALENERGY}
\end{align}
We should think of the density $\rho_A$ and the internal energy $\epsilon_A$ as the basic profiles $\accentset{\circ}{\rho}=\accentset{\circ}{\rho}(r)$ and $\accentset{\circ}{\epsilon}=\accentset{\circ}{\epsilon}(r)$, modulated by the matrix $A(t)$ through~\eqref{E:AFFINEDENSITY}--\eqref{E:AFFINEINTERNALENERGY} respectively.

Following Sideris \cite{MR3634025}, to see why we have the fundamental system of ODEs (\ref{E:AODE}) first write $x=A(t)y$ with $y \in B_1(\mathbf{0})$ which follows from $\Omega(t)=A(t)B_1(\mathbf{0})$. Then using the chain rule
\begin{equation}\label{E:CHAINRULEUA}
\frac{d}{d t} \mathbf{u}_A(t,x(t,y)) = \partial_t \mathbf{u}_A + \partial_t x \cdot \nabla \mathbf{u}_A = \partial_t \mathbf{u}_A + A'(t) y \cdot \nabla \mathbf{u}_A = \partial_t \mathbf{u}_A + \mathbf{u}_A \cdot \nabla \mathbf{u}_A.
\end{equation}
On the other hand with (\ref{E:AFFINEVELOCITY}),
\begin{equation}
\frac{d}{d t} \mathbf{u}_A(t,x(t,y)) = \frac{d}{d t} (A'(t) y) = A''(t) y.
\end{equation}
Also using the chain rule, with $r=|y|$ as above,
\begin{equation}
\nabla_{x}  = A^{-\top} \nabla_{y} = A^{-\top} y \, \partial_r.
\end{equation}
Then substituting these formulas as well as (\ref{E:AFFINEDENSITY})-(\ref{E:AFFINEINTERNALENERGY}) and the equation of state (\ref{E:EEOS}) into the momentum equation (\ref{E:MOM}) we find
\begin{equation}
(\det A(t))^{-1} \accentset{\circ}{\rho}(r) A''(t) y + (\det A(t))^{-\gamma} A(t)^{-\top} y \, \accentset{\circ}{p}'(r) = 0,
\end{equation}
where $\accentset{\circ}{p}=(\gamma-1)\accentset{\circ}{\rho}\accentset{\circ}{\epsilon}$. Separate time and space variables using the common constant $\delta > 0$ and eliminate $y$ to obtain
\begin{align}
A''(t)&= \delta (\det A(t))^{1-\gamma} A(t)^{-\top}, \\
\accentset{\circ}{p}'(r)&=-\delta r \accentset{\circ}{\rho}(r).
\end{align}
Hence we have derived the fundamental system ODEs for $A(t)$ (\ref{E:AODE}) and also the fundamental ODE in space, see (\ref{E:SIDERISPRESSUREODE}) below.

Next we require the density profile $\accentset{\circ}{\rho}$ has the following properties (Lemma 1 \cite{MR3634025})
\begin{subequations}
\begin{equation}\label{E:RHOZEROAFFREGULARITYA}
\accentset{\circ}{\rho} \in C^0[0,1] \cap C^1[0,1),
\end{equation}
\begin{equation}\label{E:RHOZEROAFFREGULARITYB}
\accentset{\circ}{\rho}(r) > 0 \text{ for } r \in [0,1),
\end{equation}
\begin{equation}\label{E:RHOZEROAFFREGULARITYC}
\accentset{\circ}{\rho}'(0)=\accentset{\circ}{\rho}(1)=0,
\end{equation}
\begin{equation}\label{E:RHOZEROAFFLIMIT}
0<\lim_{r \rightarrow 1^-} (1-r)^{-\sigma} \accentset{\circ}{\rho}(r) < \infty \text{ for some }  \sigma> 0.
\end{equation}
\end{subequations}
Now the corresponding internal energy profile $\accentset{\circ}{\epsilon}$ is found by solving the fundamental ODE in space
\begin{equation}\label{E:SIDERISPRESSUREODE}
(\accentset{\circ}{p})'(r)=-\delta r \accentset{\circ}{\rho}(r),
\end{equation}
where we recall $\accentset{\circ}{p}=(\gamma-1)\accentset{\circ}{\rho}\accentset{\circ}{\epsilon}$ and $\delta > 0$.

From (\ref{E:S}) the entropy of the Sideris affine solution satisfies, using the same argument as (\ref{E:CHAINRULEUA}),
\begin{equation}
\frac{d}{d t} S_A(t,x(t,y)) = \partial_t S_A + \mathbf{u}_A \cdot \nabla S_A = 0.
\end{equation}
Hence $S_A(t,x)$ depends on $y$ only. Moreover using (\ref{E:AFFINEDENSITY})-(\ref{E:AFFINEINTERNALENERGY}) and the equivalent equations of state (\ref{E:EEOS})-(\ref{E:SEOS}),
\begin{equation}
(\gamma-1) \frac{\accentset{\circ}{\epsilon}(r) \accentset{\circ}{\rho}(r)}{(\det A(t))^\gamma} = \frac{\accentset{\circ}{\rho}(r)^\gamma}{(\det A(t))^\gamma} e^{S_A(y)}.
\end{equation}
Therefore canceling $(\det A(t))^\gamma$ we see that $S_A(y)$ must be a function of $r=|y|=|A(t)^{-1}x|$, that is, the entropy of the Sideris affine solutions is given by
\begin{equation}\label{E:AFFENTROPY}
S_A(t,x)=\accentset{\circ}{S}(|A(t)^{-1}x|).
\end{equation} 
Here, $\accentset{\circ}{S}=\accentset{\circ}{S}(r)$ is defined through the relationship $\accentset{\circ}{p}=(\accentset{\circ}{\rho})^{\gamma} e^{\accentset{\circ}{S}}$. Through (\ref{E:SIDERISPRESSUREODE}) we can obtain an explicit formula for $\accentset{\circ}{S}$ in terms of $\accentset{\circ}{\rho}$ which together with (\ref{E:RHOZEROAFFREGULARITYA})-(\ref{E:RHOZEROAFFLIMIT}) defines the nonisentropic entropy profile associated with the Sideris affine solutions 
\begin{equation}\label{E:AFFENTROPY}  
\accentset{\circ}{S}(r) =  \ln \left( \frac{ \delta \int_{r}^1 \ell \accentset{\circ}{\rho} (\ell) \, d \ell }{(\accentset{\circ}{\rho}(r))^\gamma} \right).
\end{equation}
This formula highlights an important feature of the nonisentropic affine setting: the solutions are defined as a class of functions through the choice of $\accentset{\circ}{\rho}$ which becomes an additional parameter in the solution scheme. In particular, the space which ``parametrizes" the nonisentropic affine motions is infinite dimensional, by contrast to the isentropic case:
\begin{remark}
In the isentropic case, $\accentset{\circ}{\rho}$ is fixed as follows~\cite{MR3634025}
\begin{equation}
\accentset{\circ}{\rho}(r)=\left[\frac{\delta(\gamma-1)}{2\gamma}(1-r^2)\right]^{\frac{1}{\gamma-1}}.
\end{equation} 
On the other hand, in the nonisentropic case, the choice of $\accentset{\circ}{\rho}$ is essentially arbitrary as long as the constraints~(\ref{E:RHOZEROAFFREGULARITYA})-(\ref{E:RHOZEROAFFLIMIT}) are met.
\end{remark}

\begin{remark}
One can in fact provide a different derivation of the affine motions using only the Eulerian formulation of the problem. This approach exploits heavily the rich symmetries of the problem. The details are given in Appendix~\ref{A:SCALING}.
\end{remark}

Motivated by physical considerations \cite{bekenstein1981universal}, it is important to isolate the affine motions with uniformly bounded entropies up to the vacuum boundary. Using the formula~(\ref{E:AFFENTROPY}), we obtain the following simple, but important characterization of affine entropy behavior.

%%%%%%%%%%%%%%%%%%%%%%%%%%%

\begin{lemma}\label{L:ENTROPYLIMIT}
For $r \in [0,1), \text{ }0 < (e^{\accentset{\circ}{S}})(r) < \infty$. Furthermore
$$(e^{\accentset{\circ}{S}})(1)=\lim_{r \rightarrow 1^-} (e^{\accentset{\circ}{S}})(r)=\begin{cases} 
\infty & \sigma(\gamma-1) > 1 \\
(\delta L^{1-\gamma})/ \sigma \gamma & \sigma(\gamma-1) = 1 \\
0 & \sigma(\gamma-1) < 1 \\
 \end{cases},
 $$
 where $\sigma> 0$ from~\eqref{E:RHOZEROAFFLIMIT} is the particular value such that
$$0 < \lim_{r \rightarrow 1^-} (1-r)^{-\sigma} \accentset{\circ}{\rho} (r):=L < \infty.$$
\end{lemma} 

%%%%%%%%%%%%%%%%%%%%%%%%%%%%%
\begin{proof}
First by (\ref{E:RHOZEROAFFREGULARITYA})-(\ref{E:RHOZEROAFFREGULARITYB}), $0 < \accentset{\circ}{\rho} (r) < \infty$ for $r \in [0,1)$ and hence   
$$0 < (e^{\accentset{\circ}{S}})(r)= \frac{ \delta  \int_r^1 \ell  \accentset{\circ}{\rho} (\ell) \, d\ell }{\accentset{\circ}{\rho}(r)^\gamma} < \infty \text{ for } r \in [0,1).$$
Now
\begin{align*}
\lim_{r \rightarrow 1^-} (e^{\accentset{\circ}{S}})(r)&=\delta \lim_{r \rightarrow 1^-}  \frac{ \int_r^1 \ell  \accentset{\circ}{\rho} (\ell) \, d\ell }{\accentset{\circ}{\rho}(r)^\gamma} = \delta \lim_{r \rightarrow 1^-}  \frac{(1-r)^{\sigma \gamma}}{\accentset{\circ}{\rho}(r)^\gamma}\frac{ \int_r^1 \ell  \accentset{\circ}{\rho} (\ell) \, d\ell }{(1-r)^{\sigma \gamma}}  \\
&= \frac{\delta }{L^\gamma} \lim_{r \rightarrow 1^-} \frac{ \int_r^1 \ell  \accentset{\circ}{\rho} (\ell) \, d\ell }{(1-r)^{\sigma \gamma}} = \frac{\delta}{L^\gamma} \lim_{r \rightarrow 1^-} \frac{-r \accentset{\circ}{\rho} (r)}{-\sigma \gamma (1-r)^{\sigma \gamma -1}} \text{ (L'Hospital's Rule)} \\
&=\frac{\delta}{\sigma \gamma L^\gamma} \lim_{r \rightarrow 1^-} \frac{\accentset{\circ}{\rho}(r)}{(1-r)^\sigma} \frac{r}{(1-r)^{\sigma(\gamma-1)-1}} = \frac{\delta L^{1-\gamma}}{\sigma \gamma} \lim_{r \rightarrow 1^-} \frac{r}{ (1-r)^{\sigma(\gamma-1)-1}},
\end{align*}
wherefrom the claim follows.
\end{proof}

%%%%%%%%%%%%%%%%%%%%%%%%%%%%%%%%%%

Lemma \ref{L:ENTROPYLIMIT} shows that the only value of $\sigma$ allowing a uniformly bounded  entropy is $\sigma= \frac{1}{\gamma-1}$. In this paper, we restrict our attention to the class of nonisentropic affine solutions with finite entropies: namely, we demand $\sigma=\frac{1}{\gamma-1}>0$ in (\ref{E:RHOZEROAFFLIMIT}). Then with $\sigma=\frac{1}{\gamma-1}$ and in view of the condition (\ref{E:RHOZEROAFFREGULARITYA})-(\ref{E:RHOZEROAFFREGULARITYC}), we shall from now on consider the profiles $\accentset{\circ}{\rho}$  of the form
\begin{equation}\label{E:RHOZEROAFFDEMAND}
\accentset{\circ}{\rho}(r)=(1-r)^{\frac{1}{\gamma-1}}\phi(r),
\end{equation}
where    
$\phi \in C^{k}[0,1], \ \phi>0$ satisfying $\phi'(0)=\frac{1}{\gamma-1}\phi(0)$ with $k\in \mathbb N$ to be specified. Here we demand the condition  $\phi'(0)=\frac{1}{\gamma-1}\phi(0)$ to ensure the regularity of $\accentset{\circ}{\rho}$ at the center as in \eqref{E:RHOZEROAFFREGULARITYC}.

In the following, we show that the affine entropy enjoys the same regularity as $\phi$.      

%%%%%%%%%%%%%%%%%%%%%%%%%%%%%%%% 
 
\begin{lemma}\label{L:ENTROPYREGULARITY}
Let $\phi \in C^{k}[0,1], \ \phi>0$ be given for some $k \in \mathbb{Z}_{\geq 0}$. With (\ref{E:RHOZEROAFFDEMAND}), we have that $e^{\accentset{\circ}{S}} \in C^{k}[0,1]$.
\end{lemma} 

%%%%%%%%%%%%%%%%%%%%%%%%%%%%%%%%
\begin{proof}
The case $k=0$ follows from Lemma \ref{L:ENTROPYLIMIT}. For $k \geq 1$, first notice that by (\ref{E:AFFENTROPY}) and (\ref{E:RHOZEROAFFDEMAND})
\begin{equation}\label{E:AFFENTROPYPOSTRHOZEROAFFDEMAND}
(e^{\accentset{\circ}{S}})(r)=\frac{  \int_r^1 \delta \ell \phi(\ell)(1-\ell)^{\frac{1}{\gamma-1}} \, d\ell }{(1-r)^{\frac{\gamma}{\gamma-1}}(\phi(r))^\gamma}.
\end{equation}
To prove higher-order regularity for $\accentset{\circ}{S}$ we first consider general functions of the form
\begin{equation}
z(r):=\frac{ \int_r^1 a(\ell)(1-\ell)^{X} \, d\ell }{(1-r)^{Y}b(r)},
\end{equation}
where $a, b \in C^{k}[0,1]$, $b \neq 0$ and $X, Y \in \mathbb{R}_{>0}$ with $Y \leq X+1$. Now use the product rule, the fundamental theorem of calculus to differentiate $z$, and then integrate-by-parts to obtain
\begin{align}
z'(r)&=-a(r)(1-r)^{X-Y}[b(r)]^{-1} \notag \\
&+ \int_r^1 a(\ell)(1-\ell)^{X} \, d \ell \left[Y(1-r)^{-Y-1}(b(r))^{-1}-[b(r)]^{-2}b'(r)(1-r)^{-Y}\right] \notag \\
&=\left( \frac{Y}{X+1}-1 \right) (1-r)^{X-Y} a(r)[b(r)]^{-1} + \left(\frac{Y}{X+1}\right) \frac{ \int_r^1 a'(\ell)(1-\ell)^{X+1} \, d \ell}{(1-r)^{Y+1}b(r)} \notag \\ 
&-\left(\frac{1}{X+1}\right) a(r)[b(r)]^{-2}b'(r)(1-r)^{X+1-Y} - \left(\frac{1}{X+1}\right) \frac{ b'(r) \int_r^1 a'(\ell)(1-\ell)^{X+1} \, d \ell}{(1-r)^{Y}[b(r)]^{2}} \notag \\
&=:(i)+(ii)+(iii)+(iv). \label{E:ZDERIVATIVE}
\end{align}
The key step in the above calculation was to use integration by parts to obtain a desirable form for $(i)$ by combining terms to avoid potentially unbounded negative powers of $(1-r)$. On this note, the term $(i)$ is either $0$ (when $Y=X+1$) or $C^k$ (when $Y \leq X$). Also term $(iii)$ is $C^{k-1}$. For term $(ii)$ notice that
\begin{align*}
|(ii)|=\left\lvert\   \left(\frac{Y}{X+1}\right) \frac{ \int_r^1 a'(\ell)(1-\ell)^{X+1} \, d \ell}{(1-r)^{Y+1}b(r)}  \right\rvert\ &\leq \frac{(\sup_{0\leq r \leq 1} |a'(r)|) \int_r^1 (1-\ell)^{X+1}\, d\ell}{(1-r)^{Y+1} |b(r)|} \\
&=\frac{(\sup_{0\leq r \leq 1} |a'(r)|)\frac{1}{X+2}(1-r)^{X+2}}{(1-r)^{Y+1}|b(r)|}.
\end{align*}
Hence $(ii)$ is bounded since $Y+1 \leq X+1+1 = X+2$. Similarly, term $(iv)$ is bounded. Furthermore, we note that terms $(ii)$ and $(iv)$ are of the same general form as $z(r)$ except with $C^{k-1}$ functions replacing  the $C^k$ functions $a, b$ inside the expression for $z$. Therefore we can repeat the same procedure as for the first derivative above $k-1$ times to obtain that $z \in C^k[0,1]$. Since $e^{\accentset{\circ}{S}}$ can be realized as a special instance of the function $z$ with $a(r)=\delta r\phi(r)$, $b(r)=\phi(r)^\gamma$, $X=\frac1{\gamma-1}$, and $Y=X+1$, see~(\ref{E:AFFENTROPYPOSTRHOZEROAFFDEMAND}), we conclude $e^{\accentset{\circ}{S}} \in C^k[0,1]$, as claimed.
\end{proof}

%%%%%%%%%%%%%%%%%%%%%%%%%%%%%%%% 

We collect these results into an important consequence that will be used throughout the paper.

%%%%%%%%%%%%%%%%%%%%%%%%%%%%%%%% 

\begin{cor}\label{C:ENTROPYREGULARITYCOROLLARY}
Let $\phi \in C^{k}[0,1], \ \phi>0$ be given for some $k \in \mathbb{Z}_{\geq 0}$. With (\ref{E:RHOZEROAFFDEMAND}), then
 there exist constants $0 < c \leq C$ such that for all $r \in [0,1]$
\begin{equation}\label{E:EXPSLOWERBOUNDUPPERBOUND}
0 < c \leq (e^{\accentset{\circ}{S}})(r) \leq C,
\end{equation}
\begin{equation}\label{E:EXPSDERVSUPPERBOUND}
\left\vert\left(\frac{d}{dr}\right)^j(e^{\accentset{\circ}{S}})(r)\right\vert \leq C, \ \ 1 \leq j \leq k.
\end{equation}
\end{cor}

%%%%%%%%%%%%%%%%%%%%%%%%%%%%%%%% 

\begin{proof}
The positive lower bound in (\ref{E:EXPSLOWERBOUNDUPPERBOUND}) follows from Lemma \ref{L:ENTROPYLIMIT} and the upper bounds in  (\ref{E:EXPSLOWERBOUNDUPPERBOUND})-(\ref{E:EXPSDERVSUPPERBOUND}) follow from Lemma \ref{L:ENTROPYREGULARITY} since $C^k$ functions are bounded. 
\end{proof}

%%%%%%%%%%%%%%%%%%%%%%%%%%%%%%%% 

To conclude our characterization of nonisentropic affine motion, we finally provide precise asymptotics-in-time for $A(t)$.

%%%%%%%%%%%%%%%%%%%%%%%%%%%%%%%% 

\begin{lemma}\label{L:AASYMPTOTICS}
Consider the initial value problem~\eqref{E:AODE}--\eqref{E:FUNDAMENTALSYSTEMFIRST} with $\delta>0$.
%\begin{align}\label{E:FUNDAMENTALSYSTEM}
%& \quad \quad A''(t) = \delta \det{A(t)}^{1-\gamma} A(t)^{-\top}, \notag \\ 
%&\text{with }  (A(0),A'(0)) \in \text{{\em GL}}^+(3)\times \mathbb M^{3\times3} \text { and } \delta > 0.
%\end{align}
For $\gamma\in(1,\tfrac{5}{3}]$, the unique solution $A(t)$ to the fundamental system~\eqref{E:AODE}--\eqref{E:FUNDAMENTALSYSTEMFIRST} has the property 
\begin{equation}\label{E:DETATCUBED1}
\det A(t) \sim 1 + t^3, \quad t \geq 0 
\end{equation}
Furthermore in this case, there exist matrices $A_0,A_1,M(t)$ such that
\begin{align}\label{E:AASYMP} 
A(t) = A_0 + t A_1 + M(t), \quad t \geq 0.
\end{align}
where $A_0,A_1$ are time-independent and $M(t)$ satisfies the bounds
\begin{align}\label{E:MASYMP} 
\|M(t)\| = o_{t \rightarrow \infty}(1+t), \ \ \|\partial_t M(t)\| \lesssim (1+t)^{3-3\gamma}.
\end{align}
For $\gamma>\tfrac53$, given matrices $A_0, A_1$ with $A_1$ positive definite, there exists a unique solution $A(t)$ to the fundamental system~\eqref{E:AODE}--\eqref{E:FUNDAMENTALSYSTEMFIRST} such that (\ref{E:DETATCUBED1}), (\ref{E:AASYMP}) and (\ref{E:MASYMP}) hold.
\end{lemma}

%%%%%%%%%%%%%%%%%%%%%%%%%%%%%%%% 

\begin{proof}
For all $\gamma >1$, we use Theorem 3 and Lemma 6 from \cite{MR3634025} to obtain the results. For $\gamma \in (1,\tfrac 53]$, we additionally use Lemma A.1 from \cite{1610.01666}. For $\gamma > \tfrac53$, we additionally use Lemma 1 from \cite{shkoller2017global}.
\end{proof}
In this paper we restrict our attention to the class of nonisentropic affine solutions expanding linearly in each coordinate direction: namely we require
\begin{equation}\label{E:DETATCUBED2}
\det A(t) \sim 1 + t^3, \quad t \geq 0.
\end{equation}
By Lemma \ref{L:AASYMPTOTICS}, for $\gamma\in(1,\tfrac53]$ this is not a restriction at all in fact since $A(t)$ will immediately satisfy (\ref{E:DETATCUBED2}). For $\gamma > \tfrac53$, Lemma \ref{L:AASYMPTOTICS} shows there exists a rich class of $A(t)$ satisfying (\ref{E:DETATCUBED2}). 
 
We denote the set of affine motions under consideration by $\mathscr{S}$. To recap,  the set $\mathscr{S}$ is parametrized by the quadruple
\begin{equation}    
(A(0),A'(0),\delta,\phi) \in \text{GL}^+ (3) \times \mathbb{M}^{3 \times 3} \times \mathbb{R}_+ \times \mathcal Z_k,%C^k[0,1],
\end{equation}
where
\begin{equation}\label{Z_k}
\mathcal Z_k := \left\{ \phi \in C^k[0,1]: \phi>0, \  \phi'(0)= \frac{1}{\gamma-1}\phi(0) \right \}
\end{equation}
and we take $k \in \mathbb{N}$ sufficiently large (to be specified later in Theorems \ref{T:LWPGAMMALEQ5OVER3} and \ref{T:MAINTHEOREMGAMMALEQ5OVER3}). 
We recall $(A(0),A'(0),\delta)$ are parameters for the fundamental system~\eqref{E:AODE}--(\ref{E:FUNDAMENTALSYSTEMFIRST}) and $\phi$ appears in the formula for $\accentset{\circ}{\rho}$ (\ref{E:RHOZEROAFFDEMAND}). The choice of $\phi \in C^{k}[0,1]$  highlights a new freedom in the specification of affine motions with respect to the isentropic setting.

With our set of nonisentropic affine motions $\mathscr{S}$ in hand, the goal of this paper is to establish the global-in-time stability of the nonisentropic Euler system (\ref{E:MOM})-(\ref{E:IC}) for all $\gamma > 1$ by perturbing around the expanding affine motions. 
%$\mathscr{S}$. 
%such that $\det A(t) \sim 1 + t^3, \ t \geq 0$.

\section{Formulation and Main Global Existence Result}\label{S:FORM}  

\subsection{Lagrangian Coordinates}\label{S:LAGR}
In order to analyze the stability problem for affine motions, we will use the Lagrangian formulation that brings the problem onto the fixed domain. We first define the flow map $\zeta$ as follows     
\begin{align}
\partial_t \zeta (t,y) &= \mathbf{u}(t,\zeta(t,y)), \\
\zeta(0,y)&=\zeta_0(y),
\end{align}
where $\zeta_0$ is a sufficiently smooth diffeomorphism to be specified. We introduce the notation
\begin{align} 
\mathscr{A}_\zeta := [D \zeta]^{-1} \quad &\text{(Inverse of the Jacobian matrix)}\label{E:SCRAZETA} \\ 
\mathscr{J}_\zeta := \det[D \zeta] \quad &\text{(Jacobian determinant)}\label{E:SCRJZETA} \\
f:=\rho \circ \zeta \quad &\text{(Lagrangian density)} \\
g:=S \circ \zeta \quad &\text{(Langrangian entropy)} \\
a_\zeta:= \mathscr{J}_\zeta  \mathscr{A}_\zeta \quad &\text{(Cofactor matrix)}.
\end{align}
In this framework material derivatives reduce to pure time derivatives and in particular, the entropy equation (\ref{E:S}) is simply reformulated as
\begin{equation}\label{E:LS}
\partial_t g(t,y) = 0 \text{ which implies } g(t,y)=S_0(\zeta_0(y)). 
\end{equation} 
In other words, the entropy remains constant along fluid particle worldline given by flow maps. Furthermore, it is well-known \cite{Coutand2012,doi:10.1002/cpa.21517} that the conservation of mass equation (\ref{E:M}) gives
\begin{equation}\label{E:LD}
f(t,y)=(\mathscr{J}_\zeta(t,y))^{-1} \rho_0(\zeta_0(y)) \mathscr{J}_\zeta(0,y).
\end{equation}
Finally using the nonisentropic equation of state $p=\rho^{\gamma} e^{S}$ the momentum equation (\ref{E:MOM}) is reformulated as
\begin{equation}\label{E:LMOM}
f \partial_{tt} \zeta_i + [\mathscr{A}_\zeta]_i^k (f^\gamma e^g)_{,k}=0.
\end{equation}
Here we use coordinates $i=1,2,3$ with the Einstein summation convention and the notation $F,_k$ to denote the $k^{th}$ partial derivative of $F$.

Next introduce the following notations
\begin{align} 
\bar{S}(y)&:=S_0(\zeta_0(y)), \label{E:BARSNOTATION} \\
w(y)&:=[ \rho_0(\zeta_0(y)) \mathscr{J}_\zeta(0,y)]^{\gamma-1}  . \label{E:WNOTATION}
\end{align}
Then using $[\mathscr{A}_\zeta]_i^k =  \mathscr{J}_\zeta^{-1} [a_\zeta]_i^k$, we obtain  
\begin{equation}\label{E:LIE}
w^\alpha \partial_{tt} \zeta_i + [a_\zeta]_i^k (w^{1+\alpha} \mathscr{J}_\zeta^{\left(-\frac{1}{\alpha}-1\right)}  e^{\bar{S}})_{,k}=0,
\end{equation}
where we set
$$\alpha:=\frac{1}{\gamma-1}.$$
Using the Piola identity
\begin{equation}
([a_\zeta]_i^k)_{,k} = 0,
\end{equation}
we rewrite (\ref{E:LIE}) as
\begin{equation}\label{E:LAGRANGIANPREAFF}
w^\alpha \partial_{tt} \zeta_i  + (w^{1+\alpha} e^{\bar{S}} [\mathscr{A}_\zeta]_i^k \mathscr{J}_\zeta^{-\frac{1}{\alpha}} )_{,k}=0.
\end{equation}
Affine motions described in Section \ref{S:AFF} can be realized as special solutions of (\ref{E:LAGRANGIANPREAFF}) of the form $\zeta(t,y)=A(t)y$. In this case $\mathscr{A}_\zeta^{\top}=[D \zeta]^{-\top}=A(t)^{-\top}$ and $\mathscr{J}_\zeta = \det A$. Hence the ansatz transforms (\ref{E:LAGRANGIANPREAFF}) into
\begin{equation}\label{E:POSTAFFANSATZ}
w^\alpha A_{tt}y + (\det A)^{1-\gamma} A^{-\top} \nabla(w^{1+\alpha} e^{\bar{S}}) = 0.
\end{equation}
We have that $w^{1+\alpha} e^{\bar{S}}$ is independent of $t$ and hence (\ref{E:POSTAFFANSATZ}) will hold if we require
\begin{align}
A_{tt}&=\delta (\det A)^{1-\gamma} A^{-\top}, \label{E:AFFREQ1} \\
w^{\alpha}\delta y &=  -\nabla (w^{1+\alpha}e^{\bar{S}}), \label{E:AFFREQ2}
\end{align}
for $\delta > 0$. 
At this stage, we demand
\begin{align}
[w(y)]^\alpha&=(1-|y|)^\alpha \phi(|y|), \label{E:WDEMAND} \\
\bar{S}(y)&=\bar{S}(|y|), \label{E:BARSDEMAND}
\end{align}
where we require by recalling \eqref{Z_k}
\begin{equation} 
%\phi \in C^{k}[0,1], \ \phi>0, \ \phi'(0)= \frac{1}{\gamma-1}\phi(0)&
\phi \in \mathcal Z_k \text{ for }k \in \mathbb{N} \text{ taken sufficiently large  to be specified  later by Theorems \ref{T:LWPGAMMALEQ5OVER3} and \ref{T:MAINTHEOREMGAMMALEQ5OVER3}}. 
\label{E:PHIDEMAND} 
\end{equation}
Note with Corollary \ref{C:ENTROPYREGULARITYCOROLLARY} and (\ref{E:BARSDEMAND})-(\ref{E:PHIDEMAND}), $e^{\bar{S}} \in C^k[0,1]$ with $e^{\bar{S}} \geq c > 0$.
 
We observe that (\ref{E:AFFREQ1})-(\ref{E:PHIDEMAND}) are nothing but the affine solutions described in Section \ref{S:AFF}, and produce the set of affine motions $\mathscr{S}$ under consideration. Fix an element of $\mathscr{S}$.
\begin{remark}
Through (\ref{E:BARSNOTATION})-(\ref{E:WNOTATION}), the initial data $\rho_0$, $S_0$ for our problem are chosen such that (\ref{E:WDEMAND})-(\ref{E:PHIDEMAND}), which include the regularity demands on $\phi$, are satisfied.
\end{remark}
With an affine motion from $\mathscr{S}$ fixed, we define the modified flow map $\eta:=A^{-1}\zeta$. Then $\mathscr{A}_\zeta^{\top}=A^{-\top}\mathscr{A}_\eta^\top$, $\mathscr{J}_\zeta = (\det A) \mathscr{J}_\eta$ where $\mathscr{A}_\eta^\top$, $\mathscr{J}_\eta$ are the $\eta$ equivalents of (\ref{E:SCRAZETA}), (\ref{E:SCRJZETA}) respectively. Now from (\ref{E:LAGRANGIANPREAFF}) we have 
\begin{align}
&w^\alpha(\partial_{tt}\eta_{i} + 2 [A^{-1}]_{i \ell} \partial_t A_{\ell j} \partial_t \eta_{j} + [A^{-1}]_{i \ell} \partial_{tt} A_{\ell j} \eta_j) \notag \\
&\qquad \qquad \qquad +(\det A)^{1-\gamma} [A^{-1}]_{i\ell} [A^{-1}]_{j \ell} (w^{1+\alpha} e^{\bar{S}} [\mathscr{A}_\eta]_j^k \mathscr{J}_\eta^{-\frac{1}{\alpha}} )_{,k}=0.
\end{align} 
Via $(\ref{E:AFFREQ1})$ we rewrite the above equation as      
\begin{equation}\label{E:ETAPRETAU}
w^\alpha((\det A)^{\gamma - \frac{1}{3}} \partial_{tt} \eta_i + 2 (\det A)^{\gamma - \frac{1}{3}} [A^{-1}]_{i \ell} \partial_t A_{\ell j} \partial_t \eta_{j}) + \delta w^\alpha \Lambda_{i \ell} \eta_{\ell} + (w^{1+\alpha} e^{\bar{S}} \Lambda_{ij}[\mathscr{A}_\eta]_j^k \mathscr{J}_\eta^{-\frac{1}{\alpha}} )_{,k}=0,
\end{equation} 
where the notation $\Lambda:=(\det A)^{\frac{2}{3}} A^{-1} A^{-\top}$ has been introduced.

Next make a change of time variable by setting 
$$\frac{d \tau}{dt}=(\det A)^{-\frac{1}{3}}.$$ Then we can formulate (\ref{E:ETAPRETAU}) as 
\begin{align}
&w^\alpha((\det A)^{\gamma-1} \partial_{\tau \tau} \eta_{i} - \tfrac{1}{3}(\det A)^{\gamma - 2} (\det A)_{\tau} \partial_\tau \eta_{i}  + 2 (\det A)^{\gamma-1} [A^{-1}]_{i \ell} \partial_\tau A_{\ell j} \partial_\tau \eta_{j}) + \delta w^\alpha \Lambda_{i \ell} \eta_{\ell} \notag \\
& \quad \quad \quad \quad + (w^{1+\alpha} e^{\bar{S}} \Lambda_{ij}[\mathscr{A}_\eta]_j^k \mathscr{J}_\eta^{-\frac{1}{\alpha}} )_{,k}=0.  
\end{align}  
Writing $A=\mu O$ where $\mu:=(\det A)^{\frac{1}{3}}$ and $O \in \text{SL}(3)$, we have $A^{-1}A_{\tau} = \mu^{-1}\mu_{\tau} I + O^{-1}O_{\tau}$. The $\eta$ equation is then   
\begin{equation}\label{E:POSTMUINTROEQN}
w^\alpha(\mu^{3 \gamma-3} \partial_{\tau \tau}\eta_{i} +  \mu^{3 \gamma - 4} \mu_{\tau} \partial_\tau \eta_{i}  + 2 \mu^{3 \gamma -3} \Gamma^*_{ij} \partial_\tau \eta_{j}) + \delta w^\alpha \Lambda_{i \ell} \eta_{\ell} + (w^{1+\alpha} e^{\bar{S}} \Lambda_{ij}[\mathscr{A}_\eta]_j^k \mathscr{J}_\eta^{-\frac{1}{\alpha}} )_{,k}=0.
\end{equation}  
where we have defined $\Gamma^*: = O^{-1}O_{\tau}$. Note $\eta(y) \equiv y$ corresponds to affine motion. Introducing the perturbation
\begin{equation}
\uptheta(\tau,y):=\eta(\tau,y)-y,
\end{equation}
and using (\ref{E:AFFREQ2}), equation (\ref{E:POSTMUINTROEQN}) can be written in terms of $\uptheta$
\begin{equation}\label{E:THETAGAMMALEQ5OVER3}
w^\alpha \mu^{3 \gamma -3} \left(\partial_{\tau \tau} \uptheta_i + \frac{\mu_\tau}{\mu} \partial_\tau \uptheta_i  + 2 \Gamma^*_{ij} \partial_{\tau} \uptheta_j\right) + \delta w^\alpha \Lambda_{i \ell} \uptheta_\ell + \left(w^{1+\alpha} e^{\bar{S}} \Lambda_{ij} \left( [\mathscr{A}_\eta]_j^k \mathscr{J}_\eta^{-\frac{1}{\alpha}}-\delta_j^k\right)\right) ,_k = 0,
\end{equation}
with the initial conditions 
\begin{equation}\label{E:THETAICGAMMALEQ5OVER3}
\uptheta(0,y)=\uptheta_0(y), \quad \uptheta_\tau(0,y)=\mathbf{V}(0,y)=\mathbf{V}_{0}(y), \quad (y \in \Omega=B_1(\mathbf{0})).
\end{equation}
Above we have introduced the notation $\mathbf{V}:=\partial_\tau \uptheta$ which will be used interchangeably. 
\begin{remark}[Initial entropy $S_0$ and initial density $\rho_0$] 
The initial entropy $S_0$ and the initial density $\rho_0$ are connected to the background affine motion via
\begin{align*}
S_0(x)&=S_A((\eta_0 \circ \zeta_A(0))^{-1}(x)), \\
\rho_0(x)&=\rho_A((\eta_0 \circ \zeta_A(0))^{-1}(x)) \det[D(\zeta_0^{-1}(x))]^{-1},
\end{align*}
where the composed maps are defined by $\eta_0(y):=A^{-1}(0)\zeta_0(y)$ and $\zeta_A(0)(y):=A(0)y$.
\end{remark}
 
\subsection{Notation}\label{S:NOTATION}
For ease of notation first set
\begin{equation}
\mathscr{A}:=\mathscr{A}_\eta; \quad \mathscr{J}:=\mathscr{J}_\eta.
\end{equation}
Using $\mathscr{A} [D\eta] = \text{{\bf Id}}$, we have the differentiation formulae for $\mathscr{A}$ and $\mathscr{J}$
\begin{equation}\label{E:AJDIFFERENTIATIONFORMULAE}
\partial \mathscr{A}^k_i = - \mathscr{A}^k_\ell \partial \eta^\ell,_s \mathscr{A}^s_i \  ; \quad \partial \mathscr{J} = \mathscr{J} \mathscr{A}^s_\ell\partial\eta^\ell,_s
\end{equation}
for $\partial=\partial_\tau$ or $\partial=\partial_i$, $i=1,2,3$.

Let $\mathbf{F}:\Omega \rightarrow \mathbb{R}^3$ and $f:\Omega \rightarrow \mathbb{R}$ be an arbitrary vector field and function respectively. First define the gradient and divergence along the flow map $\eta$ respectively
\begin{equation}
[\nabla_\eta \mathbf{F}]^i_r:=\mathscr{A}^s_r \mathbf{F}^i_{,s}; \quad \text{div}_\eta \mathbf{F}:=\mathscr{A}^s_\ell \mathbf{F}^\ell_{,s}.
\end{equation}
For curl estimates, introduce the anti-symmetric curl and cross product matrices respectively
\begin{equation}\label{E:CURLCROSSPRODMATRICES}
[\text{Curl}_{\Lambda\mathscr{A}}\mathbf{F}]^i_j :=\Lambda_{jm}\mathscr{A}^s_m\mathbf{F}^i,_s- \Lambda_{im}\mathscr{A}^s_m\mathbf{F}^j,_s; \quad [\Lambda\mathscr{A}\nabla f \times \mathbf{F}]^i_j :=\Lambda_{jm}\mathscr{A}^s_m f_{,s} \mathbf{F}^i- \Lambda_{im}\mathscr{A}^s_m f_{,s} \mathbf{F}^j.
\end{equation}
For any $k \in \mathbb{Z}_{\geq 0}$ and any continuous, non-negative function $\varphi:\Omega\to\mathbb R^+$, we consider the weighted $L^2$ norm
\begin{equation}
\| f \|^2_{k,\varphi} := \int_{\Omega} \varphi w^k |f(y)|^2 \, dy.
\end{equation}
We generalize this definition to vector fields $\mathbf{F}$ and 2-tensors $\mathbf{T}:\Omega \rightarrow \mathbb{M}^{3 \times 3}$
\begin{equation}
\|\mathbf{F}\|_{k,\varphi}^2 : =\sum_{i=1}^3 \|\mathbf{F}_i\|_{k,\varphi}^2,  \ \ \|\mathbf{T}\|_{k,\varphi}^2 : =\sum_{i,j=1}^3 \|\mathbf{T}_{ij}\|_{k,\varphi}^2.
\end{equation}
The weight function $\varphi$ will often include the smooth cut-off function $\psi:B_1(\mathbf{0}) \rightarrow [0,1]$ which satisfies
\begin{equation}
\psi=0 \text{ on } B_{\frac{1}{4}}(\mathbf{0}) \text{ and } \psi = 1  \text{ on } B_1(\mathbf{0}) \setminus B_{\frac{3}{4}}(\mathbf{0}).
\end{equation}
The following derivative operators will be used near the boundary 
\begin{equation}\label{E:DERIVATIVEDEFN}
\slashed\partial_{ji}:=y_j\partial_i -y_i\partial_j, \  i,j=1,2,3;  \quad X_r:=r\partial_r \text{ where } \partial_r:=\frac{y}{r}\cdot\nabla \text{ and } r=|y|,
\end{equation}
which represent angular derivatives tangent to the boundary and the radial derivative normal to the boundary.

\subsection{High-order Quantities}\label{S:HOQGAMMALEQ5OVER3}
Our time weights will differ depending on whether $\gamma \in (1,\frac53]$ or $\gamma > \frac53$. This is because we take a slightly different approach for $\gamma > \frac53$ by an adaptation of \cite{shkoller2017global}, applied to our nonisentropic setting.

On this note, introduce the following $\gamma$ dependent exponents
\begin{equation}
d(\gamma):=\begin{cases}
3\gamma - 3 & \text{ if } \  1<\gamma\leq \frac53 \\
2 & \text{ if } \  \gamma>\frac53
\end{cases}; \quad b(\gamma):=d(\gamma)+3-3\gamma=\begin{cases}
0 & \text{ if } \  1<\gamma\leq \frac53 \\
5-3 \gamma & \text{ if } \  \gamma>\frac53
\end{cases}.
\end{equation}     	 
Let $N \in \mathbb{N}$. To measure the size of the deviation $\uptheta$, we define the high-order weighted Sobolev norm as follows 
\begin{align}
\mathcal{S}^N(\uptheta,\mathbf{V})(\tau) &:= \sup_{0\leq \tau' \leq \tau} \Big\{  \sum_{a+|\beta| \leq N} \Big( \mu^{d(\gamma)}\| X_r^a \slashed\partial^\beta \mathbf{V}\|^2_{a+\alpha,\psi e^{\bar{S}}} + \| X_r^a \slashed\partial^\beta \uptheta \|_{a+\alpha,\psi e^{\bar{S}}}^2 \Big) \notag \\ 
& \qquad + \sum_{a+|\beta| \leq N-1} \Big( \|\nabla_\eta X_r^a \slashed\partial^\beta \uptheta\|^2_{a+\alpha+1,\psi e^{\bar{S}}}+\|\text{div}_\eta X_r^a \slashed\partial^\beta \uptheta \|^2_{a+\alpha+1,\psi e^{\bar{S}}} \Big) \notag \\
&\qquad  + \sum_{a+|\beta| = N} \Big( \mu^{b(\gamma)} \|\nabla_\eta X_r^a \slashed\partial^\beta \uptheta\|^2_{a+\alpha+1,\psi e^{\bar{S}}}+\mu^{b(\gamma)} \|\text{div}_\eta X_r^a \slashed\partial^\beta \uptheta \|^2_{a+\alpha+1,\psi e^{\bar{S}}} \Big)     \Big\}  \notag \\ 
&+\sup_{0\leq \tau' \leq \tau} \Big\{  \sum_{|\nu| \leq N} \Big( \mu^{d(\gamma)} \| \partial^\nu \mathbf{V}\|^2_{\alpha,(1-\psi) e^{\bar{S}}} + \| \partial^\nu \uptheta \|_{\alpha,(1-\psi) e^{\bar{S}}}^2 \Big) \notag \\
& \qquad+  \sum_{|\nu| \leq N-1} \Big(\|\nabla_\eta \partial^\nu\uptheta\|^2_{\alpha+1,(1-\psi) e^{\bar{S}}}+\|\text{div}_\eta \partial^\nu \uptheta \|^2_{\alpha+1,(1-\psi) e^{\bar{S}}} \Big) \notag \\
& \qquad  +\sum_{|\nu|=N} \Big( \mu^{b(\gamma)} \|\nabla_\eta \partial^\nu\uptheta\|^2_{\alpha+1,(1-\psi) e^{\bar{S}}}+ \mu^{b(\gamma)} \|\text{div}_\eta \partial^\nu \uptheta \|^2_{\alpha+1,(1-\psi) e^{\bar{S}}} \Big) \Big\}.  \label{E:SNNORMGAMMALEQ5OVER3} 
\end{align}
Modified curl terms arise during energy estimates which are not a priori controlled by the norm $\mathcal{S}^N(\tau)$. These are measured via the following high-order quantity 
\begin{align}\label{E:BNNORMGAMMALEQ5OVER3}
&\mathcal{B}^N[\mathbf{V}](\tau) \notag \\
& :=\sup_{0 \leq \tau' \leq \tau} \Big\{ \sum_{a+|\beta| \leq N-1} \| \text{Curl}_{\Lambda\mathscr{A}} X_r^a \slashed\partial^\beta \mathbf{V} \|^2_{a + \alpha + 1,\psi e^{\bar{S}}} + \sum_{a+|\beta| = N} \mu^{b(\gamma)} \| \text{Curl}_{\Lambda\mathscr{A}} X_r^a \slashed\partial^\beta \mathbf{V} \|^2_{a + \alpha + 1,\psi e^{\bar{S}}} \Big\} \notag \\ 
& +\sup_{0 \leq \tau' \leq \tau} \Big\{ \sum_{|\nu| \leq N-1} \| \text{Curl}_{\Lambda \mathscr{A}} \partial^\nu \mathbf{V}\|^2_{\alpha+1,(1-\psi)e^{\bar{S}}}+ \sum_{|\nu| = N } \mu^{b(\gamma)} \| \text{Curl}_{\Lambda \mathscr{A}} \partial^\nu \mathbf{V}\|^2_{\alpha+1,(1-\psi)e^{\bar{S}}} \Big\},
\end{align}
with $\mathcal{B}^N[\uptheta]$ defined in the same way: $\uptheta$ replaces $\mathbf{V}$ in (\ref{E:BNNORMGAMMALEQ5OVER3}).

\subsection{Main Theorem}\label{S:MAINTHEOREM}
Before giving our main theorem, first define the important $\mu$ related quantities
\begin{equation}\label{E:MU1MU0DEFNGAMMALEQ5OVER3}
\mu_1:=\lim_{\tau \rightarrow \infty} \frac{\mu_\tau(\tau)}{\mu(\tau)}, \quad \mu_0:=\frac{d(\gamma)}{2}\mu_1,
\end{equation}          
where we recall the quantity $d(\gamma)=\begin{cases}
3\gamma - 3 & \text{ if } \  1<\gamma\leq \frac53 \\ 
2 & \text{ if } \  \gamma>\frac53
\end{cases}.$ \\

\textbf{Local Well-Posedness.} Next, we give the local well-posedness of our system. 
\begin{theorem}\label{T:LWPGAMMALEQ5OVER3} 
Suppose $\gamma > 1$. Fix $N\geq 2\lceil \alpha \rceil +12$. Let $k \geq N+1$ in (\ref{E:PHIDEMAND}). Then suppose we have initial data for (\ref{E:THETAGAMMALEQ5OVER3}) $(\uptheta_0,{\bf V}_0)$ that satisfies $\mathcal{S}^N(\uptheta_0, {\bf V}_0) + \mathcal{B}^N({\bf V}_0)<\infty$. In this case there exists a unique solution $(\uptheta(\tau), {\bf V}(\tau)):\Omega \rightarrow \mathbb R^3\times \mathbb R^3$ to (\ref{E:THETAGAMMALEQ5OVER3})-(\ref{E:THETAICGAMMALEQ5OVER3}) for all $\tau\in [0,T]$, for some $T>0$. The solution has the property $\mathcal{S}^N (\uptheta, {\bf V})(\tau) + \mathcal{B}^N[{\bf V}](\tau) \leq 2(\mathcal{S}^N(\uptheta_0, {\bf V}_0) + \mathcal{B}^N({\bf V}_0))$ for each $\tau\in[0,T]$. Furthermore, the map $[0,T]\ni\tau\mapsto\mathcal{S}^N(\tau)\in\mathbb R_+$ is continuous. 
\end{theorem} 
\begin{proof}
The proof follows by adapting the argument in \cite{doi:10.1002/cpa.21517}. Notably, \cite{doi:10.1002/cpa.21517} proves local well-posedness for the isentropic Euler equations using spatial weights to handle the physical vacuum condition and we are considering the nonisentropic physical vacuum setting that also includes time weights. However using the regularity of $\bar{S}$ and $w$, which are independent of the solutions, as well as the solution independent time weights, we can obtain bounds on $\mathcal{S}^N$ and $\mathcal{B}^N$ through the estimates in Section \ref{S:CURLGAMMALEQ5OVER3} and Section \ref{S:ENERGYGAMMALEQ5OVER3}. These crucial esimates let us use $\mathcal{S}^N$ and $\mathcal{B}^N$ in the techniques of \cite{doi:10.1002/cpa.21517}.  
\end{proof}

\textbf{A priori assumptions.} Finally before our main theorem, make the following a priori assumptions on our local solutions from Theorem \ref{T:LWPGAMMALEQ5OVER3} 
\begin{equation}\label{E:APRIORI}
S^N(\tau) < \frac{1}{3}, \quad \| \mathscr{A}-\textbf{Id} \|_{W^{1,\infty}(\Omega)} < \frac{1}{3}, \quad \| D\uptheta \|_{W^{1,\infty}(\Omega)} < \frac{1}{3}, \quad \| \mathscr{J}-\textbf{Id} \|_{W^{1,\infty}(\Omega)} < \frac{1}{3},
\end{equation}
for all $\tau \in [0,T]$.

We are now ready to give our main theorem.   
\begin{theorem}\label{T:MAINTHEOREMGAMMALEQ5OVER3}\label{T:MAINTHEOREMGAMMALEQ5OVER3}
Suppose $\gamma >1$. Fix $N\geq 2\lceil \alpha \rceil +12$. Let $k \geq N+1$. Consider a fixed quadruple
\begin{equation}
(A(0),A'(0),\delta,\phi)\in \text{GL}^+ (3) \times \mathbb{M}^{3 \times 3} \times \mathbb{R}_+ \times \mathcal Z_k, %C^{k}[0,1],
\end{equation} 
parametrizing a nonisentropic affine motion from the set $\mathscr{S}$ so that $\det A(t) \sim 1 + t^3, \quad t \geq 0$. Then there is an $\varepsilon_0>0$ such that for every $\varepsilon \in [0,\varepsilon_0]$ and pair of initial data for (\ref{E:THETAGAMMALEQ5OVER3}) $(\uptheta_0,{\bf V}_0)$ satisfying $\mathcal{S}^N(\uptheta_0, {\bf V}_0) + \mathcal{B}^N({\bf V}_0)\leq \varepsilon$, there exists a global-in-time solution, $(\uptheta,{\bf V})$, to the initial value problem (\ref{E:THETAGAMMALEQ5OVER3})-(\ref{E:THETAICGAMMALEQ5OVER3}).  
\end{theorem}   
\begin{proof}  
With our a priori assumptions, our high-order quantities for our local solution from Theorem \ref{T:LWPGAMMALEQ5OVER3} will be shown to satisfy the curl and energy inequalities (\ref{E:CURL1MAINTHEOREMPROOFGAMMALEQ5OVER3})-(\ref{E:ENERGYMAINTHEOREMPROOFGAMMALEQ5OVER3}) stated below in our curl estimates and energy estimate Propositions.  Then via a similar continuity argument to that presented in \cite{1610.01666} we can firstly show that our a priori assumptions are in fact improved thus justifying making them originally. For instance, by the fundamental theorem of calculus and our modified weighted Sobolev-Hardy inequality Lemma \ref{L:EMBEDDING}
\begin{equation}
\|D\uptheta\|_{W^{1,\infty}} = \|\int_0^\tau D{\bf V}\|_{W^{1,\infty}} \le \int_0^\tau e^{-\mu_0\tau'} \mathcal{S}^N(\tau')\,d\tau' \lesssim \varepsilon <\frac16, \ \ \tau\in[0,\mathcal{T})
\end{equation}
for $\varepsilon>0$ small enough and $\mathcal{T} \geq T$ is from the continuity argument. Note the continuity argument will give that $\mathcal{T}=\infty$. Similar arguments apply to the remaining a priori assumptions. Secondly from this continuity argument we can deduce we have a global-in-time solution to the initial value problem (\ref{E:THETAGAMMALEQ5OVER3})-(\ref{E:THETAICGAMMALEQ5OVER3}).
\end{proof}                  
Therefore the rest of the treatment will be devoted to proving the following curl and energy estimate Propositions respectively which will establish the curl and energy inequalities needed for the proof of Theorem \ref{T:MAINTHEOREMGAMMALEQ5OVER3}.
\begin{proposition*}[Curl Estimates]
Suppose $\gamma > 1$. Let $(\uptheta, {\bf V}):\Omega \rightarrow \mathbb R^3\times \mathbb R^3$ be a unique local solution to (\ref{E:THETAGAMMALEQ5OVER3})-(\ref{E:THETAICGAMMALEQ5OVER3}) on $[0,T]$ with $T>0$ fixed and assume $(\uptheta, {\bf V})$ satisfies the a priori assumptions (\ref{E:APRIORI}). Fix $N\geq 2\lceil \alpha \rceil +12$. Let $k \geq N+1$ in (\ref{E:PHIDEMAND}). Then for all $\tau \in [0,T]$, the following inequalities hold for some $0<\kappa\ll 1$
\begin{align}
&\mathcal{B}^N[{\bf V}](\tau) \lesssim 
\begin{cases}
 e^{-2\mu_0\tau}\left(\mathcal{S}^N(0)+\mathcal{B}^N[\mathbf{V}](0)\right)+ e^{-2\mu_0\tau}\mathcal{S}^N(\tau)   & \text{if } \ 1<\gamma<\frac53 \\ 
e^{-2\mu_0\tau}\left(\mathcal{S}^N(0)+\mathcal{B}^N[\mathbf{V}](0)\right)+ (1+\tau^2)e^{-2\mu_0\tau}\mathcal{S}^N(\tau) & \text{if } \ \gamma \geq \frac53
\end{cases},
\label{E:CURL1MAINTHEOREMPROOFGAMMALEQ5OVER3}\\ 
&
\mathcal{B}^N[\uptheta](\tau) \lesssim \mathcal{S}^N(0)+\mathcal{B}^N[\mathbf{V}](0) + \kappa \mathcal{S}^N(\tau) + \int_0^\tau e^{-\mu_0\tau'}  \mathcal{S}^N(\tau')\,d\tau'.
\label{E:CURL2MAINTHEOREMPROOFGAMMALEQ5OVER3}
\end{align}
\end{proposition*}
\begin{proposition*}[Energy Estimate]
Suppose $\gamma > 1$. Let $(\uptheta, {\bf V}):\Omega \rightarrow \mathbb R^3\times \mathbb R^3$ be a unique local solution to (\ref{E:THETAGAMMALEQ5OVER3})-(\ref{E:THETAICGAMMALEQ5OVER3}) on $[0,T]$ with $T>0$ fixed and assume $(\uptheta, {\bf V})$ satisfies the a priori assumptions (\ref{E:APRIORI}). Fix $N\geq 2\lceil \alpha \rceil +12$. Let $k \geq N+1$ in (\ref{E:PHIDEMAND}). Then for all $\tau \in [0,T]$, we have the following inequality for some $0<\kappa\ll 1$
\begin{align}
& \mathcal S^N(\tau) \lesssim  \mathcal{S}^N(0)  +\mathcal{B}^N[\uptheta](\tau) + \int_0^\tau  (\mathcal{S}^N(\tau'))^{\frac12} (\mathcal{B}^N[{\bf V}](\tau'))^{\frac12}\,d\tau'  \notag \\
& \qquad \qquad \qquad \qquad \qquad \qquad  + \kappa\mathcal{S}^N(\tau) + \int_0^\tau e^{-\mu_0\tau'} \mathcal{S}^N(\tau') d\tau'. \label{E:ENERGYMAINTHEOREMPROOFGAMMALEQ5OVER3}
\end{align}
\end{proposition*}
We henceforth assume we are working with a unique local solution $(\uptheta, {\bf V}):\Omega \rightarrow \mathbb R^3\times \mathbb R^3$ to (\ref{E:THETAGAMMALEQ5OVER3})-(\ref{E:THETAICGAMMALEQ5OVER3}) such that $\mathcal{S}^N(\uptheta, {\bf V}) +\mathcal{B}^N[{\bf V}] <\infty$ on $[0,T]$ with $T>0$ fixed: Theorem \ref{T:LWPGAMMALEQ5OVER3} ensures the existence of such a solution, and furthermore we assume this local solution satisfies the a priori assumptions (\ref{E:APRIORI}).

To prove the above key results, we apply weighted energy estimates developed in \cite{1610.01666} which enable us to handle exponentially growing-in-time coefficients and the vacuum boundary. The exponentially growing time weights capture the stabilizing effect of the expanding background affine motion and were used crucially in \cite{1610.01666} to close the estimates. As seen in our high order quantities, only spatial derivative operators are used so as to keep intact the exponential structure of time weights which function as stabilizers allowing us to close estimates. Further, the increase in spatial weight $w$ in accordance with an increase in radial derivatives seen in $\mathcal{S}^N$ and $\mathcal{B}^N$ will be essential in closing energy estimates by avoiding potentially dangerous negative powers of $w$ near the boundary.

Among others, in the nonisentropic setting, we have to contend with the presence of $e^{\bar S}$ when proceeding with our estimates. The regularity of $e^{\bar S}$ given by Corollary \ref{C:ENTROPYREGULARITYCOROLLARY} will be important for us, used throughout our analysis. However this will only be useful at the stage of isolated derivatives of $e^{\bar S}$, and it will be seen both in the curl estimates and energy estimates that we will need to be careful when taking derivatives of expressions involving $e^{\bar S}$.

In the curl estimates we almost immediately have to deal with the presence of $e^{\bar S}$ as the usual derivation of the curl does not reveal a desired structure. We will derive the curl equation carefully by exploiting the algebraic nonlinear structure of the pressure in \eqref{E:SEOS} such that it is decoupled from the main equation and that it still exhibits a good structure without loss of derivatives. The fact that $e^{\bar S}$ is a radial function from our affine development will be crucial in obtaining favorable terms with respect to our high order quantities and weight structure. In the energy estimates, new commutator formulae are obtained for differentiating the pressure term which are carefully acquired to match the weight structure of our high order norm.

Furthermore to establish the results for all $\gamma > 1$, the proofs of the curl and energy inequalities will differ for $\gamma \in (1,\frac53]$ and $\gamma>\frac53$. This is because we need to eliminate the anti-damping effect encountered in \cite{1610.01666} for $\gamma>\frac53$. Hence we will need to consider different high-order quantities for $\gamma>\frac53$, which we first saw in Section \ref{S:HOQGAMMALEQ5OVER3}. Moreover, in this setting directly from the new equation structure we consider, many of our terms will contain time weights with negative powers. Thus to control such terms without this decay, we apply a coercivity estimate technique which is not used in the $\gamma \in (1,\frac53]$ case. This technique employs the fundamental theorem of calculus to express $\uptheta$ in terms of $\mathbf{V}$ and initial data with the coercivity estimates given in Lemma \ref{L:COERCIVITY}.
            
Due to this difference in methodologies for different $\gamma$, in Sections \ref{S:CURLGAMMALEQ5OVER3}-\ref{S:ENERGYGAMMALEQ5OVER3} we give the analysis for $\gamma \in (1,\frac 53]$. In Section \ref{S:GAMMAGREATER5OVER3} we give the analysis for $\gamma > \frac 53$. Thus henceforth we fix $\gamma \in (1,\frac 53]$ in Sections \ref{S:CURLGAMMALEQ5OVER3}-\ref{S:ENERGYGAMMALEQ5OVER3} and $\gamma > \frac53$ in Section \ref{S:GAMMAGREATER5OVER3}, with the corresponding implications for $\mathcal{S}^N$ and $\mathcal{B}^N$ given in Section \ref{S:HOQGAMMALEQ5OVER3}.
    
\section{Curl Estimates}\label{S:CURLGAMMALEQ5OVER3}
To control the modified curl in the nonisentropic setting, we crucially have to contend with the presence of $e^{\bar{S}}$ in our formulation. There is no natural decoupling to exploit because of this term: instead we need to artificially decouple the system through multiplication by an appropriate power of $e^{\bar{S}}$. To then analyze the modified curl we are led to introduce novel cross product terms: first recall the cross product matrix introduced in (\ref{E:CURLCROSSPRODMATRICES})
\begin{equation}
[\Lambda\mathscr{A}\nabla f \times \mathbf{F}]^i_j :=\Lambda_{jm}\mathscr{A}^s_m f_{,s} \mathbf{F}^i- \Lambda_{im}\mathscr{A}^s_m f_{,s} \mathbf{F}^j.
\end{equation} 
The related cross product commutator will also be needed
\begin{equation} 
[\partial_\tau,\Lambda \mathscr{A} \nabla f \times ] \mathbf{F}^i_j := \partial_\tau (\Lambda_{jm}\mathscr{A}^s_m)f,_s \mathbf{F}^i - \partial_\tau(\Lambda_{im} \mathscr{A}^s_m)f,_s\mathbf{F}^j.
\end{equation} 
These two cross product quantities defined above are unique to the nonisentropic setting. Finally before we derive desirable forms for our curl matrices we introduce a term that will let us commute the time derivative outside of the curl operator
\begin{equation}\label{E:CURLCOMMUTATOR}
[\partial_\tau, \text{Curl}_{\Lambda\mathscr{A}}] \mathbf{F}^i_j := \partial_\tau \left(\Lambda_{jm}\mathscr{A}^s_m\right) \mathbf{F},_s^i - \partial_\tau \left(\Lambda_{im}\mathscr{A}^s_m\right) \mathbf{F},_s^j. 
\end{equation}

We first derive the equations for $\text{Curl}_{\Lambda\mathscr{A}}{\bf V}$ and $\text{Curl}_{\Lambda\mathscr{A}}{\bf \uptheta}$.

\begin{lemma} 
Suppose $\gamma \in (1,\frac{5}{3}]$. Let $(\uptheta, {\bf V}):\Omega \rightarrow \mathbb R^3\times \mathbb R^3$ be a unique local solution to (\ref{E:THETAGAMMALEQ5OVER3})-(\ref{E:THETAICGAMMALEQ5OVER3}) on $[0,T]$ with $T>0$ fixed. Then for all $\tau \in [0,T]$ the curl matrices $\text{\em Curl}_{\Lambda\mathscr{A}}{\bf V}$ and $\text{\em Curl}_{\Lambda\mathscr{A}}{\bf \uptheta}$ can be written in the following desirable forms
\begin{align}\label{E:FINALCURLVEQNGAMMALEQ5OVER3}
\text{\em Curl}_{\Lambda\mathscr{A}}{\bf V}  &= \tfrac{1}{\gamma}\Lambda\mathscr{A} \nabla (\bar{S}) \times \mathbf{V} \notag \\
& +  \frac{\mu(0) \text{\em Curl}_{\Lambda \mathscr{A}} ({\bf V}(0))}{\mu} - \frac{\mu(0) \Lambda\mathscr{A} \nabla (\bar{S}) \times \mathbf{V}(0)}{\gamma \mu} \notag \\
& + \frac{1}{\mu}\int_0^\tau \mu [\partial_\tau, \text{\em Curl}_{\Lambda\mathscr{A}}] {\bf V} d\tau' -\frac{1}{\gamma \mu}\int_0^\tau \mu [\partial_\tau,\Lambda \mathscr{A} \nabla ( \bar{S}) \times ] \mathbf{V} d\tau ' \notag \\
& - \frac{2}{\mu} \int_0^\tau \mu \, \text{\em Curl}_{\Lambda\mathscr{A}}(\Gamma^\ast{\bf V}) d\tau' + \frac{2}{\gamma \mu} \int_0^\tau \mu \, \Lambda\mathscr{A} \nabla (\bar{S}) \times (\Gamma^\ast \mathbf{V}) d\tau ' \notag \\
&+\frac{\delta}{\gamma \mu}\int_0^\tau \mu^{4-3 \gamma}  \Lambda \nabla(\bar{S}) \times  \Lambda \uptheta d \tau '-\frac{\delta}{\gamma \mu}\int_0^\tau \mu^{4-3 \gamma} \Lambda \mathscr{A}[D \uptheta]\nabla (\bar{S}) \times \Lambda \eta d \tau ',
\end{align}  
and          
\begin{align}\label{E:FINALCURLTHETAEQNGAMMALEQ5OVER3}
&\text{\em Curl}_{\Lambda\mathscr{A}}{\bf \uptheta}  = \tfrac{1}{\gamma}\Lambda\mathscr{A} \nabla (\bar{S}) \times \uptheta \notag \\
&+\text{\em Curl}_{\Lambda \mathscr{A}}([\uptheta(0)]) - \tfrac{1}{\gamma}\Lambda\mathscr{A} \nabla (\bar{S}) \times \uptheta(0) \notag \\
&+\mu(0) \text{\em Curl}_{\Lambda \mathscr{A}} ({\bf V}(0)) \int_0^\tau \frac{1}{\mu(\tau')} d \tau' - \frac{\mu(0) \Lambda\mathscr{A} \nabla (\bar{S}) \times \mathbf{V}(0)}{\gamma} \int_0^\tau \frac{1}{\mu(\tau')} d \tau' \notag \\
&+ \int_0^\tau  [\partial_\tau, \text{\em Curl}_{\Lambda\mathscr{A}}] {\uptheta} d\tau' -\frac{1}{\gamma}\int_0^\tau  [\partial_\tau,\Lambda \mathscr{A} \nabla ( \bar{S}) \times ] \uptheta d\tau ' \notag \\
&+\int_0^\tau \frac{1}{\mu(\tau')} \int_0^{\tau'} \mu(\tau'') [\partial_\tau, \text{\em Curl}_{\Lambda\mathscr{A}}] {\bf V} d\tau'' d \tau' -\frac{1}{\gamma} \int_0^{\tau} \frac{1}{\mu(\tau')} \int_0^{\tau'} \mu(\tau '') [\partial_\tau,\Lambda \mathscr{A} \nabla ( \bar{S}) \times ] \mathbf{V} d\tau '' d \tau ' \notag \\
& - \int_0^\tau \frac{2}{\mu(\tau')} \int_0^{\tau '} \mu(\tau'') \, \text{\em Curl}_{\Lambda\mathscr{A}}(\Gamma^\ast{\bf V}) d\tau'' d \tau' + \frac{1}{\gamma}\int_0^{\tau} \frac{2}{\mu(\tau')} \int_0^{\tau '} \mu(\tau'') \, \Lambda\mathscr{A} \nabla (\bar{S}) \times (\Gamma^\ast \mathbf{V}) d\tau '' d \tau ' \notag \\
& +\frac{\delta}{\gamma} \int_0^\tau \frac{2}{\mu(\tau')} \int_0^{\tau'} \mu(\tau'')^{4-3 \gamma}  \Lambda \nabla(\bar{S}) \times  \Lambda \uptheta d \tau '' d \tau ' \notag \\
&-\frac{\delta}{\gamma} \int_0^\tau \frac{1}{\mu(\tau')}\int_0^{\tau'} \mu(\tau '')^{4-3 \gamma} \Lambda \mathscr{A}[D \uptheta]\nabla (\bar{S}) \times \Lambda \eta d \tau '' d \tau '.
\end{align}
\end{lemma}
\begin{proof}
First divide (\ref{E:POSTMUINTROEQN}) by $w^{\alpha}$ and use $\mathbf{V}=\uptheta_\tau=\eta_\tau$ to obtain 
\begin{equation}\label{E:CURLPREMULTIPLYGAMMALEQ5OVER3}
\mu^{3 \gamma-3} \left (\mathbf{V}_ \tau +  \frac{\mu_{\tau}}{\mu} \mathbf{V}  + 2 \Gamma^* \mathbf{V} \right) + \delta \Lambda \eta + w^{-\alpha} \Lambda \mathscr{A}^{\top} \left(\nabla(w^{1+\alpha} e^{\bar{S}}\mathscr{J}^{-\frac{1}{\alpha}}) +\alpha w^{1+\alpha} e^{\bar{S}}\nabla(\mathscr{J}^{-\frac{1}{\alpha}})\right)=0,
\end{equation}
where we have split the gradient term and moved away from coordinates at this stage.
Next, multiply (\ref{E:CURLPREMULTIPLYGAMMALEQ5OVER3}) by $e^{-\frac{\bar{S}}{\gamma}}$ to obtain a $\Lambda \nabla_{\eta} f$ term
\begin{equation}\label{E:CURLPOSTMULTIPLYGAMMALEQ5OVER3}
\mu^{3 \gamma-3} e^{-\frac{\bar{S}}{\gamma}} \left(\mathbf{V}_ \tau +  \frac{\mu_{\tau}}{\mu} \mathbf{V}  + 2 \Gamma^* \mathbf{V} \right) + \delta \Lambda e^{-\frac{\bar{S}}{\gamma}} \eta + (\alpha+1) \Lambda \nabla_\eta \left( w e^{\left(\frac{1}{\alpha+1}\right)\bar{S}}\mathscr{J}^{-\frac{1}{\alpha}} \right)=0.
\end{equation}
Since $\text{Curl}_{\Lambda \mathscr{A}} (\Lambda \nabla_\eta f )=0$ apply $\text{Curl}_{\Lambda \mathscr{A}}$ to (\ref{E:CURLPOSTMULTIPLYGAMMALEQ5OVER3})
\begin{align}
&\mu^{3 \gamma-3} \left(\text{Curl}_{\Lambda \mathscr{A}} \left(e^{-\frac{\bar{S}}{\gamma}} \mathbf{V}_ \tau\right) +  \frac{\mu_{\tau}}{\mu} \text{Curl}_{\Lambda \mathscr{A}} \left(e^{-\frac{\bar{S}}{\gamma}} \mathbf{V} \right)  + 2 \text{Curl}_{\Lambda \mathscr{A}} \left(e^{-\frac{\bar{S}}{\gamma}} \Gamma^* \mathbf{V} \right) \right) \notag \\
&\qquad \qquad \qquad \qquad + \delta \text{Curl}_{\Lambda \mathscr{A}} \left( \Lambda e^{-\frac{\bar{S}}{\gamma}} \eta\right) =0.
\end{align}
Then
\begin{equation}
\text{Curl}_{\Lambda \mathscr{A}} \left(e^{-\frac{\bar{S}}{\gamma}} \mathbf{V}_ \tau\right) +  \frac{\mu_{\tau}}{\mu} \text{Curl}_{\Lambda \mathscr{A}} \left(e^{-\frac{\bar{S}}{\gamma}} \mathbf{V} \right)  + 2 \text{Curl}_{\Lambda \mathscr{A}} \left(e^{-\frac{\bar{S}}{\gamma}} \Gamma^* \mathbf{V} \right) + \mu^{3-3 \gamma} \delta \, \text{Curl}_{\Lambda \mathscr{A}} \left( \Lambda e^{-\frac{\bar{S}}{\gamma}} \eta\right) =0.
\end{equation}
Reformulate this as
\begin{align}\label{E:CURLPREINTGAMMALEQ5OVER3}
&\partial_\tau \left(\mu \, \text{Curl}_{\Lambda\mathscr{A}}\left(e^{-\frac{\bar{S}}{\gamma}}{\bf V}\right)   \right) = \mu [\partial_\tau, \text{Curl}_{\Lambda\mathscr{A}}] \left(e^{-\frac{\bar{S}}{\gamma}}{\bf V}\right)  - 2 \mu \, \text{Curl}_{\Lambda\mathscr{A}}\left(e^{-\frac{\bar{S}}{\gamma}}\Gamma^\ast{\bf V}\right) \notag \\
&\qquad \qquad \qquad \qquad -\mu^{4-3 \gamma} \delta \, \text{Curl}_{\Lambda \mathscr{A}} \left( \Lambda e^{-\frac{\bar{S}}{\gamma}} \eta\right),
\end{align}
where we have used
\begin{equation}
\text{Curl}_{\Lambda \mathscr{A}} \left(e^{-\frac{\bar{S}}{\gamma}} \mathbf{V}_ \tau\right)=\partial_\tau \left( \, \text{Curl}_{\Lambda\mathscr{A}}\left(e^{-\frac{\bar{S}}{\gamma}}{\bf V}\right)   \right)-[\partial_\tau, \text{Curl}_{\Lambda\mathscr{A}}] \left(e^{-\frac{\bar{S}}{\gamma}}{\bf V}\right),
\end{equation}
Integrate (\ref{E:CURLPREINTGAMMALEQ5OVER3}) from $0$ to $\tau$ 
\begin{align}\label{E:CURLEXPMULTVGAMMALEQ5OVER3}
\text{Curl}_{\Lambda\mathscr{A}}\left(e^{-\frac{\bar{S}}{\gamma}}{\bf V}\right)  &= \frac{\mu(0) \text{Curl}_{\Lambda \mathscr{A}} \left(e^{-\frac{\bar{S}}{\gamma}}[{\bf V}(0)]\right)}{\mu}+\frac{1}{\mu}\int_0^\tau \mu [\partial_\tau, \text{Curl}_{\Lambda\mathscr{A}}] \left(e^{-\frac{\bar{S}}{\gamma}}{\bf V}\right) d\tau' \notag \\
&\qquad - \frac{2}{\mu} \int_0^\tau \mu \, \text{Curl}_{\Lambda\mathscr{A}}\left(e^{-\frac{\bar{S}}{\gamma}}\Gamma^\ast{\bf V}\right) d\tau'-\frac{\delta}{\mu}\int_0^\tau \mu^{4-3 \gamma} \text{Curl}_{\Lambda \mathscr{A}} \left( \Lambda e^{-\frac{\bar{S}}{\gamma}} \eta\right) d \tau'.
\end{align}
Via (\ref{E:CURLCOMMUTATOR}) we obtain
\begin{align}
\partial_\tau \text{Curl}_{\Lambda\mathscr{A}}\left(e^{-\frac{\bar{S}}{\gamma}}{\bf \uptheta}[\tau]\right)  &= \frac{\mu(0) \text{Curl}_{\Lambda \mathscr{A}} \left(e^{-\frac{\bar{S}}{\gamma}}[{\bf V}(0)]\right)}{\mu}+[\partial_\tau,\text{Curl}_{\Lambda \mathscr{A}}]\left(e^{-\frac{\bar{S}}{\gamma}} \uptheta\right) \notag \\
&+\frac{1}{\mu}\int_0^\tau \mu [\partial_\tau, \text{Curl}_{\Lambda\mathscr{A}}] \left(e^{-\frac{\bar{S}}{\gamma}}{\bf V}\right) d\tau' - \frac{2}{\mu} \int_0^\tau \mu \, \text{Curl}_{\Lambda\mathscr{A}}\left(e^{-\frac{\bar{S}}{\gamma}}\Gamma^\ast{\bf V}\right) d\tau' \notag \\
&-\frac{\delta}{\mu}\int_0^\tau \mu^{4-3 \gamma} \text{Curl}_{\Lambda \mathscr{A}} \left( \Lambda e^{-\frac{\bar{S}}{\gamma}} \eta\right) d \tau'.
\end{align}
Again integrating from $0$ to $\tau$ 
\begin{align}
&\text{Curl}_{\Lambda\mathscr{A}}\left(e^{-\frac{\bar{S}}{\gamma}}{\bf \uptheta}[\tau]\right) =\text{Curl}_{\Lambda \mathscr{A}}\left( e^{-\frac{\bar{S}}{\gamma}}{\bf \uptheta}[0] \right)+\mu(0) \text{Curl}_{\Lambda \mathscr{A}} \left(e^{-\frac{\bar{S}}{\gamma}}[{\bf V}(0)]\right) \int_0^\tau \frac{1}{\mu(\tau ')} d \tau ' \notag \\
&\quad \, +\int_0^\tau [\partial_\tau,\text{Curl}_{\Lambda \mathscr{A}}]\left(e^{-\frac{\bar{S}}{\gamma}} \uptheta[\tau '] \right) d \tau ' +\int_0^\tau\frac{1}{\mu(\tau ')}  \int_0^{\tau '} \mu [\partial_\tau, \text{Curl}_{\Lambda\mathscr{A}}] \left(e^{-\frac{\bar{S}}{\gamma}}{\bf V}[\tau ''] \right) d \tau '' d\tau' \notag \\
& \quad \, -\int_0^\tau \frac{2}{\mu(\tau ')} \int_0^{\tau '} \mu(\tau '') \, \text{Curl}_{\Lambda\mathscr{A}}\left(e^{-\frac{\bar{S}}{\gamma}}\Gamma^\ast{\bf V}[\tau '']\right) d\tau '' d \tau ' \\
&\quad \,-\int_0^\tau \frac{\delta}{\mu(\tau ')}\int_0^{\tau '} \mu(\tau '')^{4-3 \gamma} \text{Curl}_{\Lambda \mathscr{A}} \left( \Lambda e^{-\frac{\bar{S}}{\gamma}} \eta [ \tau '']\right) d \tau '' d \tau '. \label{E:CURLTHETAPREAPPROACHGAMMALEQ5OVER3}
\end{align}
Now note
\begin{equation}\label{E:CURLDISTGAMMALEQ5OVER3}
\text{Curl}_{\Lambda\mathscr{A}}\left(e^{-\frac{\bar{S}}{\gamma}}{\bf F}\right)=e^{-\frac{\bar{S}}{\gamma}}\text{Curl}_{\Lambda\mathscr{A}}{\bf F}+\Lambda\mathscr{A} \nabla (e^{-\frac{\bar{S}}{\gamma}}) \times \mathbf{F},
\end{equation}
where we recall the definitions introduced in (\ref{E:CURLCROSSPRODMATRICES}). Also via (\ref{E:CURLCOMMUTATOR}) we have
\begin{equation}\label{E:CURLCOMMUTATORDISTGAMMALEQ5OVER3}
[\partial_\tau, \text{Curl}_{\Lambda\mathscr{A}}] \left(e^{-\frac{\bar{S}}{\gamma}} \mathbf{F} \right)= e^{-\frac{\bar{S}}{\gamma}} [\partial_\tau,\text{Curl}_{\Lambda\mathscr{A}}] \mathbf{F} + [\partial_\tau,\Lambda \mathscr{A} \nabla ( e^{-\frac{\bar{S}}{\gamma}}) \times ] \mathbf{F},
\end{equation}   
Then using 
\begin{equation}
\text{Curl}_{\Lambda \mathscr{A}}(\Lambda \eta) = 0,
\end{equation}
we can rewrite (\ref{E:CURLEXPMULTVGAMMALEQ5OVER3}) as
\begin{align}\label{E:CURLVPOSTCURLEXPAND1GAMMALEQ5OVER3}
e^{-\frac{\bar{S}}{\gamma}} \text{Curl}_{\Lambda\mathscr{A}}{\bf V}  &= -\Lambda\mathscr{A} \nabla (e^{-\frac{\bar{S}}{\gamma}}) \times \mathbf{V} \notag \\
&\qquad + e^{-\frac{\bar{S}}{\gamma}} \frac{\mu(0) \text{Curl}_{\Lambda \mathscr{A}} ({\bf V}(0))}{\mu} + \frac{\mu(0) \Lambda\mathscr{A} \nabla (e^{-\frac{\bar{S}}{\gamma}}) \times \mathbf{V}(0)}{\mu} \notag \\
&\qquad  + \frac{1}{\mu}\int_0^\tau \mu e^{-\frac{\bar{S}}{\gamma}} [\partial_\tau, \text{Curl}_{\Lambda\mathscr{A}}] {\bf V} d\tau' +\frac{1}{\mu}\int_0^\tau \mu [\partial_\tau,\Lambda \mathscr{A} \nabla ( e^{-\frac{\bar{S}}{\gamma}}) \times ] \mathbf{V} d\tau' \notag \\
&\qquad - \frac{2}{\mu} \int_0^\tau \mu \, e^{-\frac{\bar{S}}{\gamma}} \text{Curl}_{\Lambda\mathscr{A}}(\Gamma^\ast{\bf V}) d\tau' - \frac{2}{\mu} \int_0^\tau \mu \, \Lambda\mathscr{A} \nabla (e^{-\frac{\bar{S}}{\gamma}}) \times (\Gamma^\ast \mathbf{V}) d\tau' \notag \\
&\qquad -\frac{\delta}{\mu}\int_0^\tau \mu^{4-3 \gamma} \Lambda\mathscr{A} \nabla (e^{-\frac{\bar{S}}{\gamma}}) \times (\Lambda \eta) d \tau'.
\end{align}
Now multiply (\ref{E:CURLVPOSTCURLEXPAND1GAMMALEQ5OVER3}) by $e^{\frac{\bar{S}}{\gamma}}$ and note
\begin{equation}
\left( e^{-\frac{\bar{S}}{\gamma}}\right),_s = -\tfrac{1}{\gamma} (\bar{S}),_s e^{-\frac{\bar{S}}{\gamma}},
\end{equation}
to obtain
\begin{align}\label{E:CURLVPOSTCURLEXPAND2GAMMALEQ5OVER3}
\text{Curl}_{\Lambda\mathscr{A}}{\bf V}  &= \tfrac{1}{\gamma}\Lambda\mathscr{A} \nabla (\bar{S}) \times \mathbf{V} \notag \\
&\qquad +  \frac{\mu(0) \text{Curl}_{\Lambda \mathscr{A}} ({\bf V}(0))}{\mu} - \frac{\mu(0) \Lambda\mathscr{A} \nabla (\bar{S}) \times \mathbf{V}(0)}{\gamma \mu} \notag \\
&\qquad  + \frac{1}{\mu}\int_0^\tau \mu [\partial_\tau, \text{Curl}_{\Lambda\mathscr{A}}] {\bf V} d\tau' -\frac{1}{\gamma \mu}\int_0^\tau \mu [\partial_\tau,\Lambda \mathscr{A} \nabla ( \bar{S}) \times ] \mathbf{V} d\tau ' \notag \\
&\qquad - \frac{2}{\mu} \int_0^\tau \mu \, \text{Curl}_{\Lambda\mathscr{A}}(\Gamma^\ast{\bf V}) d\tau' + \frac{2}{\gamma \mu} \int_0^\tau \mu \, \Lambda\mathscr{A} \nabla (\bar{S}) \times (\Gamma^\ast \mathbf{V}) d\tau ' \notag \\
&\qquad +\frac{\delta}{\gamma \mu}\int_0^\tau \mu^{4-3 \gamma} \Lambda\mathscr{A} \nabla (\bar{S}) \times (\Lambda \eta) d \tau '.
\end{align}
Now since $\eta=O(1)$ for our estimates, it is not a priori clear if the last term of (\ref{E:CURLVPOSTCURLEXPAND2GAMMALEQ5OVER3}) can be controlled. To this end, first note $\mathscr{A}=[D \eta]^{-1}$ and $\eta=y+\uptheta$. We then have
\begin{equation}\label{E:CURLADELTAIDENTITYPROOF}
\mathscr{A}^k_j-\delta^k_j = \mathscr{A}^k_l  \delta^l_j  -\mathscr{A}^k_l [D\eta]^l_j  = \mathscr{A}^k_l (\delta^l_j - [D  y]^l_j - [D\uptheta]^l_j) = -\mathscr{A}^k_l  [D\uptheta]^l_j,
\end{equation}
Second notice that, since $\bar{S}$ is a radial function from (\ref{E:BARSDEMAND}) and hence $\bar{S},_s=(y_s /|y|)\bar{S}'(|y|)$,
\begin{equation}
\Lambda_{js}(\bar{S}),_s\Lambda_{ik}y^k - \Lambda_{is} (\bar{S}),_s \Lambda_{jk} y^k=\frac{\bar{S}'(|y|)}{|y|}\left( \Lambda_{js} y^s \Lambda_{ik} y^k - \Lambda_{is} y^s \Lambda_{jk} y^k \right) =0. 
\end{equation} 
Then
\begin{align}\label{E:CURLLASTTERMIDENTITYGAMMALEQ5OVER3}
&[\Lambda\mathscr{A} \nabla (\bar{S}) \times (\Lambda \eta)]_j^i = \Lambda_{jm}(\delta_m^s - \mathscr{A}_\ell^s [D \uptheta]_m^\ell)(\bar{S}),_s \Lambda_{ik} \eta^k - \Lambda_{im}(\delta_m^s - \mathscr{A}_\ell^s [D \uptheta]_m^\ell)(\bar{S}),_s \Lambda_{jk} \eta^k \notag \\
&= \Lambda_{js}(\bar{S}),_s\Lambda_{ik}\uptheta^k - \Lambda_{is} (\bar{S}),_s \Lambda_{jk} \uptheta^k+\Lambda_{js}(\bar{S}),_s\Lambda_{ik}y^k - \Lambda_{is} (\bar{S}),_s \Lambda_{jk} y^k \notag \\
&\qquad \qquad \qquad -(\Lambda_{jm}\mathscr{A}_\ell^s [D \uptheta]_m^\ell(\bar{S}),_s \Lambda_{ik} \eta^k - \Lambda_{im} \mathscr{A}_\ell^s [D \uptheta]_m^\ell(\bar{S}),_s \Lambda_{jk} \eta^k) \notag \\
&=(\Lambda_{js}(\bar{S}),_s\Lambda_{ik}\uptheta^k - \Lambda_{is} (\bar{S}),_s \Lambda_{jk} \uptheta^k) -(\Lambda_{jm}\mathscr{A}_\ell^s [D \uptheta]_m^\ell(\bar{S}),_s \Lambda_{ik} \eta^k - \Lambda_{im} \mathscr{A}_\ell^s [D \uptheta]_m^\ell(\bar{S}),_s \Lambda_{jk} \eta^k) \notag \\
&:=[\Lambda \nabla(\bar{S}) \times  \Lambda \uptheta]_j^i- [\Lambda \mathscr{A}[D \uptheta]\nabla (\bar{S}) \times \Lambda \eta]_j^i.
\end{align}
Hence using (\ref{E:CURLLASTTERMIDENTITYGAMMALEQ5OVER3}), we rewrite (\ref{E:CURLVPOSTCURLEXPAND2GAMMALEQ5OVER3}) and obtain our desirable form for $\text{Curl}_{\Lambda\mathscr{A}}{\bf V}$, (\ref{E:FINALCURLVEQNGAMMALEQ5OVER3}). 
Following the same approach as for obtaining (\ref{E:FINALCURLVEQNGAMMALEQ5OVER3}), notably using (\ref{E:CURLDISTGAMMALEQ5OVER3}), (\ref{E:CURLCOMMUTATORDISTGAMMALEQ5OVER3}), multiplying by $e^{\frac{\bar{S}}{\gamma}}$ and finally using (\ref{E:CURLLASTTERMIDENTITYGAMMALEQ5OVER3}), from (\ref{E:CURLTHETAPREAPPROACHGAMMALEQ5OVER3}) we obtain our desirable form for $\text{Curl}_{\Lambda\mathscr{A}}{\bf \uptheta}$, (\ref{E:FINALCURLTHETAEQNGAMMALEQ5OVER3}).
\end{proof} 
Next we prove some preliminary bounds for several of the terms unique to our nonisentropic setting. These bounds on terms arising from the derivation above will be used in our main curl estimate. 
\begin{lemma}\label{L:CURLLASTTERMBOUNDS}
Suppose $\gamma \in (1,\frac{5}{3}]$. Let $(\uptheta, {\bf V}):\Omega \rightarrow \mathbb R^3\times \mathbb R^3$ be a unique local solution to (\ref{E:THETAGAMMALEQ5OVER3})-(\ref{E:THETAICGAMMALEQ5OVER3}) on $[0,T]$ with $T>0$ fixed and assume $(\uptheta, {\bf V})$ satisfies the a priori assumptions (\ref{E:APRIORI}).  Fix $N\geq 2\lceil \alpha \rceil +12$. Let $k \geq N+1$ in (\ref{E:PHIDEMAND}). Then for all $\tau \in [0,T]$, the following bounds hold for $a+|\beta| \leq N$, $|\nu| \leq N$   
\begin{align} 
& \|\frac{\delta}{\gamma \mu}\int_0^\tau \mu^{4-3 \gamma} X_r^a \slashed\partial^\beta (\Lambda \nabla(\bar{S}) \times  \Lambda \uptheta) d \tau ' \|^2_{\alpha+a+1,\psi e^{\bar{S}}} \notag \\
& \ \ \ \ \ + \|\frac{\delta}{\gamma \mu}\int_0^\tau \mu^{4-3 \gamma} \partial^\nu (\Lambda \nabla(\bar{S}) \times  \Lambda \uptheta) d \tau ' \|^2_{1+\alpha,(1-\psi) e^{\bar{S}}} \lesssim e^{-2\mu_0 \tau} \mathcal{S}^N(\tau), \label{E:CURLLASTTERMBOUND1} \\
&\|\frac{\delta}{\gamma \mu}\int_0^\tau \mu^{4-3 \gamma} X_r^a \slashed\partial^\beta (\Lambda \mathscr{A}[D \uptheta]\nabla (\bar{S}) \times \Lambda \eta) d \tau ' \|^2_{\alpha+a+1,\psi e^{\bar{S}}} \notag \\
& \ \ \ \ \ + \|\frac{\delta}{\gamma \mu}\int_0^\tau \mu^{4-3 \gamma} \partial^\nu  (\Lambda \mathscr{A}[D \uptheta]\nabla (\bar{S}) \times \Lambda \eta) d \tau ' \|^2_{1+\alpha,(1-\psi) e^{\bar{S}}} \lesssim e^{-2\mu_0 \tau} \mathcal{S}^N(\tau), \label{E:CURLLASTTERMBOUND2} \\
& \| \frac{\delta}{\gamma} \int_0^\tau \frac{2}{\mu} \int_0^{\tau'} \mu^{4-3 \gamma} X_r^a \slashed\partial^\beta (\Lambda \nabla(\bar{S}) \times  \Lambda \uptheta ) d \tau '' d \tau ' \|_{\alpha+a+1,\psi e^{\bar{S}}}^2 \notag \\
&  \ \ \ \ \ + \| \frac{\delta}{\gamma} \int_0^\tau \frac{2}{\mu} \int_0^{\tau'} \mu^{4-3 \gamma} \partial^\nu (\Lambda \nabla(\bar{S}) \times  \Lambda \uptheta) d \tau '' d \tau ' \|_{1+\alpha,(1-\psi) e^{\bar{S}}}^2 \lesssim \int_0^\tau e^{-\mu_0 \tau'} \mathcal{S}^N(\tau') d \tau' ,\label{E:CURLLASTTERMBOUND3} \\
& \| \frac{\delta}{\gamma} \int_0^\tau \frac{1}{\mu}\int_0^{\tau'} \mu^{4-3 \gamma} X_r^a \slashed\partial^\beta ( \Lambda \mathscr{A}[D \uptheta]\nabla (\bar{S}) \times \Lambda \eta ) d \tau '' d \tau ' \|_{\alpha+a+1,\psi e^{\bar{S}}}^2 \notag \\ 
& \ \ \ \ \ +\| \frac{\delta}{\gamma} \int_0^\tau \frac{1}{\mu}\int_0^{\tau'} \mu^{4-3 \gamma} \partial^\nu (\Lambda \mathscr{A}[D \uptheta]\nabla (\bar{S}) \times \Lambda \eta) d \tau '' d \tau ' \|_{1+\alpha,(1-\psi) e^{\bar{S}}}^2 \lesssim \int_0^\tau e^{-\mu_0 \tau'} \mathcal{S}^N(\tau') d \tau'. \label{E:CURLLASTTERMBOUND4} 
\end{align}
\end{lemma}
\begin{proof}
\textit{Proof of} (\ref{E:CURLLASTTERMBOUND1})-(\ref{E:CURLLASTTERMBOUND2}). By the definition of $\Lambda \nabla (\bar{S}) \times \Lambda \uptheta$ introduced in (\ref{E:CURLLASTTERMIDENTITYGAMMALEQ5OVER3}), we have
\begin{equation}
X_r^a \slashed\partial^\beta[(\Lambda \nabla(\bar{S}) \times  \Lambda \uptheta)]_j^i=X_r^a \slashed\partial^\beta(\Lambda_{js} \bar{S},_s \Lambda_{ik} \uptheta^k - \Lambda_{is} \bar{S},_s \Lambda_{jk} \uptheta^k).
\end{equation}
We restrict our focus to the left term only
\begin{align}\label{E:CURLRESTRICTLEFTGAMMALEQ5OVER3}
X_r^a \slashed\partial^\beta (\Lambda_{js} \bar{S},_s \Lambda_{ik} \uptheta^k) &= \Lambda_{js} \Lambda_{ik} ((X_r^a \slashed\partial^\beta \bar{S},_s) \uptheta^k + \bar{S},_s (X_r^a \slashed\partial^\beta \uptheta^k)) \notag \\
&\qquad + \sum_{1\le a'+|\beta'|\le N-1} C_{a',\beta'} (\Lambda_{js} (X_r^{a'} \slashed\partial^{\beta'} \bar{S},_s) \Lambda_{ik} (X_r^{a-a'} \slashed\partial^{\beta-\beta'} \uptheta^k)).
\end{align}
Schematically consider the first two terms on the right hand side of (\ref{E:CURLRESTRICTLEFTGAMMALEQ5OVER3})
\begin{equation}
\underbrace{\Lambda \Lambda (X_r^a \slashed\partial^\beta (D \bar{S}))\uptheta}_{=:C_1}+\underbrace{\Lambda\Lambda(D\bar{S})(X_r^a \slashed\partial^\beta \uptheta)}_{=:C_2}
\end{equation}
Notice that for $C_1$, since $X_r^a \slashed\partial^\beta (D \bar{S})$ is bounded by Corollary \ref{C:ENTROPYREGULARITYCOROLLARY} and Lemma \ref{L:DERIVATIVERADIAL},
\begin{align}\label{E:C1BOUNDFIRSTGAMMALEQ5OVER3}  
\| \frac{\delta}{\gamma \mu}\int_0^\tau \mu^{4-3 \gamma}  C_1 d \tau ' \|_{\alpha+a+1,\psi e^{\bar{S}}}^2 &= \int_\Omega \psi w^{1+\alpha +a} e^{\bar{S}} \frac{\delta^2}{\gamma^2 \mu^2} \Big| \int_0^\tau \mu^{4-3\gamma} \Lambda \Lambda (X_r^a \slashed\partial^\beta (D \bar{S})) \uptheta d \tau' \Big|^2 d y \notag \\
&\lesssim \frac{1}{\mu^2} \Big| \int_0^\tau \mu^{4-3 \gamma} d \tau' \Big|^2 \sup_{0 \leq \tau \leq \tau'} \|  \uptheta \|^2_{\alpha,\psi e^{\bar{S}}} \notag \\
&\lesssim U(\tau) \mathcal{S}^N(\tau),
\end{align}  
where we also have used (\ref{E:WDEMAND}) and (\ref{E:LAMBDABOUNDSGAMMALEQ5OVER3}), and we define $U(\tau)=\frac{1}{\mu^2} | \int_0^\tau \mu^{4-3 \gamma} d \tau'|^2$. Now for $U$, using (\ref{E:EXPMU1MUINEQGAMMALEQ5OVER3}) and (\ref{E:MU0MU1INEQGAMMALEQ5OVER3}), we have
\begin{align}\label{E:UBOUNDFIRSTGAMMALEQ5OVER3}
U &\lesssim e^{-2 \mu_1 \tau} \Big| \int_0^\tau e^{(4-3 \gamma)\mu_1 \tau'} d \tau' \Big|^2 \notag \\
&\lesssim \begin{cases}
e^{-2 \mu_1\tau}e^{2(4-3\gamma)\mu_1\tau} & \text{ if } \  1<\gamma<\frac43 \\
e^{-2 \mu_1\tau}\tau^2 & \text{ if } \ \gamma =\frac43 \\
e^{-2 \mu_1\tau}(e^{2(4-3\gamma)\mu_1\tau}+1) & \text{ if } \  \frac43<\gamma \leq \frac53
\end{cases}.
\end{align}   
By the definition of $\mu_1$ (\ref{E:MU1MU0DEFNGAMMALEQ5OVER3}), we have $-2\mu_1 + 2 (4-3\gamma)\mu_1=-4\mu_0$ and also, $\mu_1 \geq \mu_0$. Furthermore, for $\gamma=\frac43$, $\mu_0=\frac12 \mu_1$, and so in this case
\begin{equation} 
\tau^2 e^{-2 \mu_1\tau}=\tau^2 e^{-\mu_1 \tau} e^{-2 \mu_0 \tau} \lesssim e^{-2 \mu_0 \tau}.
\end{equation}
Therefore by (\ref{E:UBOUNDFIRSTGAMMALEQ5OVER3})
\begin{equation}\label{E:UBOUNDSECONDGAMMALEQ5OVER3}
U \lesssim e^{-2\mu_0 \tau}.
\end{equation}
Hence finally for $C_1$ by (\ref{E:C1BOUNDFIRSTGAMMALEQ5OVER3}), we have 
\begin{equation}\label{E:C1BOUNDSECONDGAMMALEQ5OVER3}
\| \frac{\delta}{\gamma \mu}\int_0^\tau \mu^{4-3 \gamma}  C_1 d \tau ' \|_{\alpha+a+1,\psi e^{\bar{S}}}^2 \lesssim e^{-2\mu_0 \tau} S^N(\tau).
\end{equation}
For $C_2$, by employing Fubini's theorem to interchange the spatial and time integrals,
\begin{align}
\| \frac{\delta}{\gamma \mu}\int_0^\tau \mu^{4-3 \gamma}  C_1 d \tau ' \|_{\alpha+a+1,\psi e^{\bar{S}}}^2 &= \int_\Omega \psi w^{1+\alpha +a} e^{\bar{S}} \frac{\delta^2}{\gamma^2 \mu^2} \Big| \int_0^\tau \mu^{4-3\gamma} \Lambda\Lambda(D\bar{S})(X_r^a \slashed\partial^\beta \uptheta) d \tau' \Big|^2 d y \notag \\
&\lesssim U(\tau) \sup_{0 \leq \tau \leq \tau'} \| X_r^a \slashed\partial^\beta \uptheta \|^2_{a+\alpha,\psi e^{\bar{S}}} \notag \\
&\lesssim e^{-2\mu_0 \tau} \mathcal{S}^N(\tau) \label{E:C2BOUNDGAMMALEQ5OVER3},
\end{align}
where we have again used Corollary \ref{C:ENTROPYREGULARITYCOROLLARY} and Lemma \ref{L:DERIVATIVERADIAL} to bound the entropy term. We have also used our bound for $U$ (\ref{E:UBOUNDSECONDGAMMALEQ5OVER3}) as well as (\ref{E:WDEMAND}) and (\ref{E:LAMBDABOUNDSGAMMALEQ5OVER3}). The low order commutator terms on the right hand side of  (\ref{E:CURLRESTRICTLEFTGAMMALEQ5OVER3}) can be estimated in the same way as $C_1$ and $C_2$ 
\begin{equation}\label{E:LOWORDERSAMEWAYASC1C2GAMMALEQ5OVER3}
\|  \frac{\delta}{\gamma \mu}\int_0^\tau \mu^{4-3 \gamma} C_{a',\beta'} (\Lambda_{js} (X_r^{a'} \slashed\partial^{\beta'} \bar{S},_s) \Lambda_{ik} (X_r^{a-a'} \slashed\partial^{\beta-\beta'} \uptheta^k)) d \tau ' \|_{1+\alpha+a,\psi e^{\bar{S}}}^2 \lesssim e^{-2\mu_0 \tau} \mathcal{S}^N(\tau),
\end{equation}
for $a'+|\beta'|\le N-1$. Hence by (\ref{E:C1BOUNDSECONDGAMMALEQ5OVER3}), (\ref{E:C2BOUNDGAMMALEQ5OVER3}) and (\ref{E:LOWORDERSAMEWAYASC1C2GAMMALEQ5OVER3}), and an analogous argument for the $\|\cdot\|_{1+\alpha,(1-\psi)e^{\bar{S}}}$ norms, we obtain (\ref{E:CURLLASTTERMBOUND1}).
By an analogous argument additionally using the a priori assumptions (\ref{E:APRIORI}) in the case and also Lemma \ref{L:EMBEDDING} for the low order commutator terms, we get (\ref{E:CURLLASTTERMBOUND2}). \\

\textit{Proof of} (\ref{E:CURLLASTTERMBOUND3})-(\ref{E:CURLLASTTERMBOUND4}). With $0 < \lambda < 1$ fixed to be specified later and using the Cauchy-Schwarz inequality in conjunction with Fubini's theorem to take advantage of the fact that time integrals of negative powers $\mu$ are bounded by (\ref{E:EXPMU1MUINEQGAMMALEQ5OVER3})
\begin{align}
& \| \frac{\delta}{\gamma} \int_0^\tau \frac{2}{\mu(\tau')} \int_0^{\tau'} \mu(\tau'')^{4-3 \gamma} (X_r^a \slashed\partial^\beta \Lambda \nabla(\bar{S}) \times  \Lambda \uptheta) d \tau '' d \tau ' \|_{\alpha+a+1,\psi e^{\bar{S}}}^2  \notag \\
& \lesssim \int_\Omega  \int_0^\tau \frac{1}{(\mu(\tau'))^{2(1-\lambda)}} d \tau' \int_0^\tau \frac{1}{\mu(\tau')^{2 \lambda}}   \left(\int_0^{\tau'} \mu(\tau'')^{4-3\gamma} (X_r^a \slashed\partial^\beta \Lambda \nabla(\bar{S}) \times  \Lambda \uptheta) d \tau''\right)^2 d \tau' w^{1+\alpha +a} \psi e^{\bar{S}} dy \notag \\
& \lesssim \int_0^\tau \frac{1}{\mu(\tau')^{2 \lambda}} \| \int_0^{\tau'} \mu(\tau'')^{4-3\gamma} (X_r^a \slashed\partial^\beta \Lambda \nabla(\bar{S}) \times  \Lambda \uptheta) d \tau'' \|^2_{1+\alpha+a,\psi e^{\bar{S}}} d \tau ' \notag \\
& \lesssim \int_0^\tau e^{-\mu_0 \tau'} \mathcal{S}^N(\tau') d \tau',
\end{align}
where we conclude the final bound by a similar argument to that giving (\ref{E:CURLLASTTERMBOUND1}) if we can show
\begin{equation}
\frac{1}{\mu(\tau')^{2 \lambda}} \Big| \int_0^{\tau'} \mu(\tau'')^{4-3\gamma} d \tau'' \Big|^2 \lesssim e^{-\mu_0 \tau}.
\end{equation}
Now if $\gamma\neq\frac43$
\begin{align}
\frac{1}{\mu(\tau')^{2 \lambda}} \Big| \int_0^{\tau'} \mu(\tau'')^{4-3\gamma} d \tau'' \Big|^2 & \lesssim e^{-2 \lambda \mu_1 \tau'} \Big|  \int_0^{\tau'} e^{(4-3\gamma)\mu_1 \tau''} d \tau'' \Big|^2  \notag \\
& \lesssim \begin{cases}
e^{-2 \lambda \mu_1\tau'}e^{2(4-3\gamma)\mu_1\tau'} & \text{ if } \  1<\gamma<\frac43 \\
e^{-2 \lambda \mu_1\tau'}(e^{2(4-3\gamma)\mu_1\tau'}+1) & \text{ if } \  \frac43<\gamma \leq \frac53
\end{cases} \notag  \\ 
& = \begin{cases}
e^{(8-6\gamma-2\lambda)\mu_1\tau'} & \text{ if } \  1<\gamma<\frac43 \\
e^{(8-6\gamma-2\lambda)\mu_1\tau'}+e^{-2 \lambda \mu_1\tau'} & \text{ if } \  \frac43<\gamma \leq \frac53
\end{cases}.
\end{align}
Here, specify $\lambda$ as follows
\begin{align}
\lambda=\begin{cases}
\frac{1}{2} & \text{ if } \  \frac{11}{9}\leq\gamma<\frac43 \ \text{ or } \ \frac43<\gamma\leq \frac53 \quad (i)\\
r \; \text{ where } r \in [1-\frac{9 \gamma-9}{4},1) & \text{ if } \ 1<\gamma <\frac{11}{9} \quad (ii)
\end{cases}
\end{align}
In case $(i)$
\begin{equation}
(8-6\gamma-2\lambda)\mu_1=\frac{2}{3\gamma-3}\mu_0-4\mu_0 \leq -\mu_0,
\end{equation}
since $\frac{2}{3 \gamma -3} \leq 3$ for $\frac{11}{9} \leq \gamma$. Also in case $(i)$, $-2 \lambda \mu_1 =-\mu_1 \leq -\mu_0$. In case $(ii)$
\begin{equation}
(8-6\gamma-2\lambda)\mu_1=-4\mu_0 + \frac{2(2-2 \lambda)}{3\gamma-3} \mu_0 \leq -\mu_0
\end{equation}
since $\frac{2(2-2 \lambda)}{3\gamma-3} \leq 3$ for $\lambda \geq 1-\frac{9 \gamma-9}{4}$. Finally for $\gamma=\frac43$, set $\lambda=\frac34$, and then
\begin{align}
\frac{1}{\mu(\tau')^{2 \lambda}} \Big| \int_0^{\tau'} \mu(\tau'')^{4-3\gamma} d \tau'' \Big|^2 &\lesssim e^{-\frac32 \mu_1 \tau'} \Big|  \int_0^{\tau'} e^{(4-3\gamma)\mu_1 \tau''} d \tau'' \Big|^2 \notag \\
&\lesssim e^{-\frac32 \mu_1 \tau'} (\tau')^2 = e^{-\mu_0 \tau'} (\tau')^2 e^{-\mu_1 \tau'} \lesssim e^{-\mu_0 \tau'}.
\end{align}
In all cases, we have the same result and hence
\begin{equation}
\frac{1}{\mu(\tau')^{2 \lambda}} \Big| \int_0^{\tau'} \mu(\tau'')^{4-3\gamma} d \tau'' \Big|^2 \lesssim e^{-\mu_0 \tau}.
\end{equation}
So with an analogous argument for the $\|\cdot\|_{1+\alpha,(1+\psi)e^{\bar{S}}}$ norms, we have (\ref{E:CURLLASTTERMBOUND3}). By an analogous argument, with (\ref{E:CURLLASTTERMBOUND2}) replacing (\ref{E:CURLLASTTERMBOUND1}), we can get (\ref{E:CURLLASTTERMBOUND4}). 
\end{proof}
Before proving our main curl estimate result, we also prove bounds on commutator terms that will  immediately arise when applying our derivative operators.  
\begin{lemma}\label{L:COMMUTATORBOUNDS} 
Suppose $\gamma \in (1,\frac{5}{3}]$. Let $(\uptheta, {\bf V}):\Omega \rightarrow \mathbb R^3\times \mathbb R^3$ be a unique local solution to (\ref{E:THETAGAMMALEQ5OVER3})-(\ref{E:THETAICGAMMALEQ5OVER3}) on $[0,T]$ with $T>0$ fixed and assume $(\uptheta, {\bf V})$ satisfies the a priori assumptions (\ref{E:APRIORI}). Fix $N\geq 2\lceil \alpha \rceil +12$. Let $k \geq N+1$ in (\ref{E:PHIDEMAND}). Then for all $\tau \in [0,T]$, the following bounds hold for some $0<\kappa\ll 1$,
\begin{align} 
&\sum_{1\leq a+|\beta|\le N} \Big\|\left[X_r^a \slashed\partial^\beta,\text{{\em Curl}}_{\Lambda\mathscr{A}}\right]{\bf V}(\tau)\Big\|_{1+\alpha+a,\psi e^{\bar{S}}}^2
+ \sum_{|\nu|\le N} \Big\|\left[\partial^\nu,\text{{\em Curl}}_{\Lambda\mathscr{A}}\right]{\bf V}(\tau)\Big\|_{1+\alpha,(1-\psi)e^{\bar{S}}}^2 \notag  \\
& \ \ \ \ \  \lesssim e^{-2\mu_0\tau}\mathcal{S}^N(\tau)\left(1+P(\mathcal{S}^N(\tau))\right), \label{E:COMMUTATORBOUND1}  \\
&\sum_{1 \leq a+|\beta|\le N} \Big\|\left[X_r^a \slashed\partial^\beta,\text{{\em Curl}}_{\Lambda\mathscr{A}}\right]\uptheta(\tau) \Big\|_{1+\alpha+a,\psi e^{\bar{S}}}^2
+ \sum_{|\nu|\le N} \Big\|\left[\partial^\nu,\text{{\em Curl}}_{\Lambda\mathscr{A}}\right]\uptheta(\tau)\Big\|_{1+\alpha,(1-\psi)e^{\bar{S}}}^2 \notag  \\
& \ \ \ \ \  \lesssim \mathcal{S}^N(0)+ \kappa\mathcal{S}^N(\tau) +(1+P(\mathcal{S}^N(\tau)))\int_0^\tau e^{-\mu_0\tau'} \mathcal{S}^N(\tau')\,d\tau', \label{E:COMMUTATORBOUND2}
\end{align}
with $P$ a polynomial of degree at least 1.
\end{lemma}
\begin{proof}
\textit{Proof of} (\ref{E:COMMUTATORBOUND1}). Note
\begin{align}
&\left[X_r^a \slashed\partial^\beta,\text{{Curl}}_{\Lambda\mathscr{A}}\right]{\bf V}^k_i
 =  X_r^a \slashed\partial^\beta\left(\text{{Curl}}_{\Lambda\mathscr{A}}{\bf V}^k_i\right)-\text{{Curl}}_{\Lambda\mathscr{A}}(X_r^a \slashed\partial^\beta {\bf V})^k_i  \notag \\
& =X_r^a \slashed\partial^\beta \left(\Lambda_{im}\mathscr{A}^s_m{\bf V}^k,_s- \Lambda_{km}\mathscr{A}^s_m{\bf V}^i,_s\right)- \left(\Lambda_{im}\mathscr{A}^s_m(X_r^a \slashed\partial^\beta {\bf V})^k,_s- \Lambda_{km}\mathscr{A}^s_m(X_r^a \slashed\partial^\beta {\bf V})^i,_s \right)   \notag\\
&  =\Lambda_{km}\left(\mathscr{A}^s_m(X_r^a \slashed\partial^\beta {\bf V})^i,_s  - X_r^a \slashed\partial^\beta(\mathscr{A}^s_m{\bf V})^i,_s)\right)- \Lambda_{im}\left(\mathscr{A}^s_m(X_r^a \slashed\partial^\beta {\bf V})^k,_s - X_r^a \slashed\partial^\beta(\mathscr{A}^s_m{\bf V}^k,_s)\right).  \label{E:VCOMM}
\end{align}
Since the two terms in the last line above are estimated similarly noting derivative count, we restrict our focus to the second term only. Via the Leibniz rule
\begin{align}
&\mathscr{A}^s_m(X_r^a\slashed\partial^\beta {\bf V})^k,_s - X_r^a\slashed\partial^\beta(\mathscr{A}^s_m{\bf V}^k,_s)  = - \underbrace{ (X_r^a\slashed\partial^\beta\mathscr{A}^s_m) {\bf V}^k,_s}_{=:A} \notag \\
 & - \sum_{1\le a'+|\beta'|\le N-1\atop a'\le a, \beta'\le\beta}X_r^{a'}\slashed\partial^{\beta'}\mathscr{A}^s_m X_r^{a-a'} \slashed\partial^{\beta-\beta'}{\bf V}^k,_s - 
\underbrace{\mathscr{A}^s_m [X_r^a\slashed\partial^\beta,\partial_s]{\bf V}^k}_{=:B} \label{E:CURLCOMM1} 
\end{align} 
Now for $A$ first note if $|\beta|=0$ then
\begin{align}
X_r^a\slashed\partial^\beta \mathscr{A}^s_m  = X_r^a\mathscr{A}^s_m &= - \mathscr{A}^s_i(X_r^a \uptheta^s),_k\mathscr{A}^k_m- \mathscr{A}^s_i[X_r^{a-1},\partial_m]X_r\uptheta^i\mathscr{A}^k_i \notag \\
&-\sum_{1\le a'\le a } c_{a'}X_r^{a'}(\mathscr{A}^s_i\mathscr{A}^k_m)X_r^{a-a'}\left(X_r\uptheta^i\right),_k 
 +X_r^{a-1}\left(\mathscr{A}^s_i[\partial_k,X_r]\uptheta^i\mathscr{A}^k_m \right), \label{E:TOPORDER1}
\end{align}
and if $|\beta| > 0$ then for some $e_\ell$ given by $(e_\ell)^\ell = 1, (e_\ell)^i = 0 \text{ for } i\neq\ell$ and $\ell_1,\ell_2 \in \{1,2,3\}, \ell_1 \neq \ell_2$ 
\begin{align}
X_r^a\slashed\partial^\beta\mathscr{A}^s_m & =- \mathscr{A}^s_i(X_r^a\slashed\partial^\beta \uptheta^i),_k
\mathscr{A}^k_m - \mathscr{A}^s_i[X_r^a\slashed\partial^{\beta-e_\ell},\partial_k]\slashed\partial_{\ell_1\ell_2}\uptheta^i\mathscr{A}^k_m \notag \\
&  \ \ \ \ -\sum_{0<a'+|\beta'|\le a+|\beta| \atop \beta'\le\beta-e_{\ell} }c_{a'\beta'}X_r^{a'}\slashed\partial^{\beta'}(\mathscr{A}^s_i\mathscr{A}^k_m)X_r^{a-a'}\slashed\partial^{\beta-e_{\ell}-\beta'}\left(\slashed\partial_{\ell_1\ell_2} \uptheta^i\right),_k \notag \\
&  \ \ \ \ -X_r^a\slashed\partial^{\beta-e_\ell}\left( \mathscr{A}^s_i[\partial_k,\slashed\partial_{\ell_1\ell_2}]\uptheta^i\mathscr{A}^k_m \right), \label{E:TOPORDER2}
\end{align}
where we have used the formulae
\begin{align}
X_r\mathscr{A}^s_m  & = -\mathscr{A}^s_i X_r\partial_k \eta^i \mathscr{A}^k_m =  -\mathscr{A}^s_i X_r\left(\partial_k \uptheta^i+\delta^i_k\right)\mathscr{A}^k_m \notag \\
& = -\mathscr{A}^s_i (X_r\uptheta^i),_k\mathscr{A}^k_m  + \mathscr{A}^s_i [\partial_k,X_r]\uptheta^i\mathscr{A}^k_m. \label{E:COMMFORMULA1} \\
\slashed\partial_{\ell_1 \ell_2} \mathscr{A}^s_m & =  -\mathscr{A}^s_i(\slashed\partial_{\ell_1 \ell_2}\uptheta^i),_k\mathscr{A}^k_m  +\mathscr{A}^s_i[\partial_k,\slashed\partial_{\ell_1 \ell_2}]\uptheta^i\mathscr{A}^k_m. \label{E:COMMFORMULA2}
\end{align}
Without loss of generality, suppose $|\beta| > 0$ because using (\ref{E:TOPORDER1}), the $|\beta|=0$ case follows analogously. Using (\ref{E:TOPORDER2})
\begin{align}
&A  = - \underbrace{\mathscr{A}^k_j(X_r^a\slashed\partial^\beta \uptheta^j),_m\mathscr{A}^m_i{\bf V}^k,_s}_{=:A_1} -  \mathscr{A}^k_i[X_r^a\slashed\partial^{\beta-e_\ell},\partial_m]\slashed\partial_{\ell_1 \ell_2}\uptheta^s\mathscr{A}^m_i \notag \\
& -{\bf V}^k,_s\Big(\sum_{0<a'+|\beta'|\le a+|\beta| \atop \beta'\le\beta-e_\ell }c_{a'\beta'}X_r^{a'}\slashed\partial^{\beta'}(\mathscr{A}^k_j\mathscr{A}^m_i)X_r^{a-a'}\slashed\partial^{\beta-e_\ell-\beta'}\left(\slashed\partial_{\ell_1 \ell_2} \uptheta^j\right),_m +X_r^a\slashed\partial^{\beta-e_\ell}\left( \mathscr{A}^k_j[\partial_m,\slashed\partial_{\ell_1 \ell_2} ]\uptheta^j\mathscr{A}^m_i \right) \Big). \label{E:ATERM}
\end{align}
By Lemma \ref{L:EMBEDDING}
\begin{align}\label{E:COMMUTATORBOUNDA1ESTIMATEGAMMALEQ5OVER3}
\|A_1\|_{1+\alpha+a,\psi e^{\bar{S}}}^2& \lesssim \|D{\bf V}\|_{L^\infty(\Omega)}^2\|\nabla_\eta X_r^a \slashed\partial^\beta \uptheta\|_{1+\alpha+a,\psi e^{\bar{S}} }^2 \notag \\
& \lesssim e^{-2\mu_0\tau} (\mathcal{S}^N(\tau))^2\left(1+P(\mathcal{S}^N(\tau))\right).
\end{align} 
Similar arguments using Lemma \ref{L:EMBEDDING} and additionally (\ref{E:DIFFCOMMUTATOR}) below let us estimate the remaining terms on the right hand side of (\ref{E:ATERM}). Next considering $B$, we apply our derivative commutator formulae (\ref{E:DERIVATIVECOMMUTATOR}) and at each step using (\ref{E:DERIVATIVEDECOMP}) to rewrite $\partial_s$ in terms of $X_r$ and $\slashed\partial_{js}$ in conjunction with our simple forms for differentiating $\frac{y_{j}}{r}$ given by Lemma \ref{L:DERIVATIVERADIAL}, we reduce $[X_r^a\slashed\partial^\beta,\partial_s]$ to the following desirable form
\begin{equation}\label{E:DIFFCOMMUTATOR}
[X_r^a\slashed\partial^\beta,\partial_s] = \sum_{j=0}^{a+|\beta|}\sum_{a'+|\beta'|=j \atop a'\leq a+1, |\beta'| \leq |\beta|+1} C^s_{a',\beta',j}X_r^{a'}\slashed\partial^{\beta'}.
\end{equation}
where $C^s_{a',\beta',j}$ are smooth coefficients on $B_1(\mathbf{0})$ away from the origin. For $B$ we note $a+|\beta| \leq N$ in (\ref{E:DIFFCOMMUTATOR}). Then this desriable formula (\ref{E:DIFFCOMMUTATOR}) and Lemma \ref{L:EMBEDDING} gives us
\begin{equation}\label{E:COMMUTATORBOUNDBESTIMATEGAMMALEQ5OVER3}
\|B\|_{1+\alpha+a,\psi e^{\bar{S}}}^2 \lesssim e^{-2\mu_0\tau} \mathcal{S}^N(\tau).
\end{equation}
Hence by the above estimates and an analogous argument for the $\|\cdot\|_{1+\alpha,(1-\psi)e^{\bar{S}}}$ norms, we obtain (\ref{E:COMMUTATORBOUND1}). \\ 

\textit{Proof of} (\ref{E:COMMUTATORBOUND2}). Via (\ref{E:VCOMM}) with $\uptheta$, we estimate
\begin{align}
&\mathscr{A}^s_m(X_r^a\slashed\partial^\beta  \uptheta)^k,_s - X_r^a\slashed\partial^\beta (\mathscr{A}^s_m\uptheta^k,_s)  = - \underbrace{ (X_r^a\slashed\partial^\beta \mathscr{A}^s_m) \uptheta^k,_s}_{=:A_\uptheta} \notag \\
 & - \sum_{1\le a'+|\beta'|\le N-1\atop a'\le a, \beta'\le\beta}X_r^{a'}\slashed\partial^{\beta'}\mathscr{A}^s_m X_r^{a-a'}\slashed\partial^{\beta-\beta'}\uptheta^k,_s - \mathscr{A}^s_m [X_r^a\slashed\partial^\beta ,\partial_s]\uptheta^k. \label{E:CURLCOMM2}
\end{align}
By estimating $A_\uptheta$, using the fact that $\| X_r^a\slashed\partial^\beta \mathscr{A}\|_{1+\alpha+a,\psi e^{\bar{S}}}^2 \lesssim \mathcal{S}^N(\tau)$ which was proven above in the estimation of the term $A$: most importantly note (\ref{E:ATERM})-(\ref{E:COMMUTATORBOUNDA1ESTIMATEGAMMALEQ5OVER3}), we give the sketch for the proof
\begin{align}
\|A_\uptheta\|_{1+\alpha+a,\psi e^{\bar{S}}}^2 & \lesssim \|X_r^a\slashed\partial^\beta \mathscr{A}\|_{1+\alpha+a,\psi e^{\bar{S}}}^2 \| D\uptheta\|_{L^\infty(\Omega)}^2 \notag \\
& \lesssim  \| X_r^a\slashed\partial^\beta \mathscr{A}\|_{1+\alpha+a,\psi e^{\bar{S}}}^2 \left(\|D\uptheta_0\|_{L^\infty(\Omega)} + \int_0^\tau \| D{\bf V}\|_{L^\infty(\Omega)}\right)^2 \notag \\
& \lesssim \mathcal{S}^N(\tau) \left(\mathcal{S}^N(0) + \left(\int_0^\tau e^{-\mu_0\tau'}\sqrt{\mathcal{S}^N(\tau')}\,d\tau'\right)^2\right) \notag \\
&\lesssim  \mathcal{S}^N(\tau) \left(\mathcal{S}^N(0) + \int_0^\tau e^{-\mu_0\tau'}\mathcal{S}^N(\tau')\,d\tau'\right), \label{E:COMMUTATORBOUNDATHETAESTIMATEGAMMALEQ5OVER3}
\end{align}
where we have applied Lemma \ref{L:EMBEDDING}. The remaining terms in (\ref{E:CURLCOMM2}) are estimated in a similar way.
\end{proof} 
We are now ready to prove our key curl estimate result which will be used crucially in the proof of our main result Theorem \ref{T:MAINTHEOREMGAMMALEQ5OVER3}. 
\begin{proposition}\label{P:CURLBOUNDS}
Suppose $\gamma \in (1,\frac{5}{3}]$. Let $(\uptheta, {\bf V}):\Omega \rightarrow \mathbb R^3\times \mathbb R^3$ be a unique local solution to (\ref{E:THETAGAMMALEQ5OVER3})-(\ref{E:THETAICGAMMALEQ5OVER3}) on $[0,T]$ with $T>0$ fixed and assume $(\uptheta, {\bf V})$ satisfies the a priori assumptions (\ref{E:APRIORI}). Fix $N\geq 2\lceil \alpha \rceil +12$. Let $k \geq N+1$ in (\ref{E:PHIDEMAND}). Then for all $\tau \in [0,T]$, the following inequalities hold for some $0<\kappa\ll 1$
\begin{align}
&\mathcal{B}^N[{\bf V}](\tau) \lesssim 
\begin{cases}
 e^{-2\mu_0\tau}\left(\mathcal{S}^N(0)+\mathcal{B}^N[\mathbf{V}](0)\right)+ e^{-2\mu_0\tau}\mathcal{S}^N(\tau)   & \text{if } \ 1<\gamma<\frac53 \\ 
e^{-2\mu_0\tau}\left(\mathcal{S}^N(0)+\mathcal{B}^N[\mathbf{V}](0)\right)+ (1+\tau^2)e^{-2\mu_0\tau}\mathcal{S}^N(\tau) & \text{if } \ \gamma =\frac53
\end{cases},
\label{E:CURLVBOUND} \\ 
&
\mathcal{B}^N[\uptheta](\tau) \lesssim \mathcal{S}^N(0)+\mathcal{B}^N[\mathbf{V}](0) + \kappa \mathcal{S}^N(\tau) + \int_0^\tau e^{-\mu_0\tau'}  \mathcal{S}^N(\tau')\,d\tau'. \label{E:CURLTHETABOUND}
\end{align}
\end{proposition}
\begin{proof} \textit{Proof of} (\ref{E:CURLVBOUND}). 
Apply $X_r^a \slashed\partial^\beta$ to (\ref{E:FINALCURLVEQNGAMMALEQ5OVER3})
\begin{align}\label{E:POSTDERVCURLVEQNGAMMALEQ5OVER3}
&\text{Curl}_{\Lambda\mathscr{A}}{X_r^a \slashed\partial^\beta \bf V}  = - [X_r^a \slashed\partial^\beta,\text{Curl}_{\Lambda \mathscr{A}}] \mathbf{V}+ \tfrac{1}{\gamma} X_r^a \slashed\partial^\beta (\Lambda\mathscr{A} \nabla (\bar{S}) \times \mathbf{V}) \notag \\  
& +  \frac{\mu(0) X_r^a \slashed\partial^\beta \text{Curl}_{\Lambda \mathscr{A}} ({\bf V}(0))}{\mu} - \frac{\mu(0) X_r^a \slashed\partial^\beta (\Lambda\mathscr{A} \nabla (\bar{S}) \times \mathbf{V}(0))}{\gamma \mu} \notag \\
&  + \frac{1}{\mu}\int_0^\tau \mu X_r^a \slashed\partial^\beta [\partial_\tau, \text{Curl}_{\Lambda\mathscr{A}}] {\bf V} d\tau' -\frac{1}{\gamma \mu}\int_0^\tau \mu X_r^a \slashed\partial^\beta [\partial_\tau,\Lambda \mathscr{A} \nabla ( \bar{S}) \times ] \mathbf{V} d\tau ' \notag \\
&- \frac{2}{\mu} \int_0^\tau \mu \, X_r^a \slashed\partial^\beta \text{Curl}_{\Lambda\mathscr{A}}(\Gamma^\ast{\bf V}) d\tau' + \frac{2}{\gamma \mu} \int_0^\tau \mu \, X_r^a \slashed\partial^\beta (\Lambda\mathscr{A} \nabla (\bar{S}) \times (\Gamma^\ast \mathbf{V})) d\tau ' \notag \\
& +\frac{\delta}{\gamma \mu}\int_0^\tau \mu^{4-3 \gamma} X_r^a \slashed\partial^\beta (\Lambda \nabla(\bar{S}) \times  \Lambda \uptheta) d \tau '-\frac{\delta}{\gamma \mu}\int_0^\tau \mu^{4-3 \gamma} X_r^a \slashed\partial^\beta (\Lambda \mathscr{A}[D \uptheta]\nabla (\bar{S}) \times \Lambda \eta) d \tau '.
\end{align}
The bound on the first term on the right hand side of (\ref{E:POSTDERVCURLVEQNGAMMALEQ5OVER3}) follows from Lemma \ref{L:COMMUTATORBOUNDS}. The second term is similar to the first term but lower order and hence is also bounded. For the third term we have
\begin{equation}\label{E:CURLEST0}
\left\|\frac{\mu(0)X^a_r\slashed\partial^\beta\text{Curl}_{\Lambda\mathscr{A}}{\bf V} (0)}{\mu} \right\|_{1+\alpha+a,\psi e^{\bar{S}}}^2 \lesssim e^{-2\mu_1\tau} (\mathcal{S}^N(0)+\mathcal{B}^N[\mathbf{V}](0)),
\end{equation}
where we apply (\ref{E:EXPMU1MUINEQGAMMALEQ5OVER3}). The fourth term is similar to the third term but lower order and hence is also bounded. We now consider the fifth term on the right hand side of (\ref{E:POSTDERVCURLVEQNGAMMALEQ5OVER3})
\begin{align}
X_r^a\slashed\partial^\beta[\partial_\tau, \text{Curl}_{\Lambda\mathscr{A}}] {\bf V}_j^k  =X_r^a\slashed\partial^\beta \left(\partial_\tau \left(\Lambda_{jm}\mathscr{A}^s_m\right) {\bf V},_s^k - \partial_\tau \left(\Lambda_{km}\mathscr{A}^s_m\right) {\bf V},_s^j \right).
\end{align}
As the other tern can be estimated in the same way, we restrict our focus to the first term only
\begin{align}
&X_r^a\slashed\partial^\beta \left(\partial_\tau \left(\Lambda_{jm}\mathscr{A}^s_m\right) {\bf V},_s^k\right) \notag \\
& = X_r^a\slashed\partial^\beta\left[\left(\partial_\tau\Lambda_{jm}\mathscr{A}^s_m + \Lambda_{jm}\partial_\tau \mathscr{A}^s_m \right){\bf V},_s^k\right] \notag \\
& = \partial_\tau\Lambda_{jm} \left(X_r^a\slashed\partial^\beta\mathscr{A}^s_m {\bf V},_s^k +\mathscr{A}^s_mX_r^a\slashed\partial^\beta{\bf V},_s^k\right) +  \Lambda_{jm}\left(X_r^a\slashed\partial^\beta\partial_\tau\mathscr{A}^s_m {\bf V},_s^k +\partial_\tau\mathscr{A}^s_mX_r^a\slashed\partial^\beta{\bf V},_s^k\right) \notag \\
& \ \ \ \ + \sum_{1\le a'+|\beta'|\le N-1} C_{a',\beta'}\left(\partial_\tau\Lambda_{jm}X_r^{a'}\slashed\partial^{\beta'}\mathscr{A}^s_m +\Lambda_{jm} X_r^{a'}\slashed\partial^{\beta'}\partial_\tau\mathscr{A}^s_m\right) 
 X_r^{a-a'}\slashed\partial^{\beta-\beta'} {\bf V}^k,_s.
\label{E:CURLERROREXPANSION}
\end{align}
Schematically consider the first two terms on the right hand side of (\ref{E:CURLERROREXPANSION})
\begin{equation}\label{E:CURLERROR1}
\underbrace{\partial_\tau\Lambda \,X_r^a\slashed\partial^\beta D\uptheta D{\bf V}}_{=:D_1} + \underbrace{\partial_\tau\Lambda  D\eta \, X_r^a\slashed\partial^\beta D{\bf V}}_{=:D_2} + \underbrace{\Lambda D{\bf V} \, X_r^a\slashed\partial^\beta D{\bf V}}_{=:D_3} 
\end{equation}
For $D_1$, using the exponential bounds on $\mu$ and $\Lambda$ given by (\ref{E:EXPMU1MUINEQGAMMALEQ5OVER3}) and (\ref{E:LAMBDABOUNDSGAMMALEQ5OVER3}) respectively
\begin{align}
& \int_\Omega \psi w^{1+\alpha + a}e^{\bar{S}} \frac1{\mu^{2}}\Big|\int_0^\tau \mu \partial_\tau\Lambda \,X_r^a\slashed\partial^\beta D\uptheta D{\bf V} \,d\tau'\Big|^2\,dy \notag \\
& \lesssim e^{-2\mu_1\tau} \int_\Omega \psi w^{1+\alpha + a} e^{\bar{S}} \Big|\int_0^\tau \left|e^{\mu_1\tau'} \partial_\tau\Lambda \,X_r^a\slashed\partial^\beta D\uptheta D{\bf V} \right|\,d\tau'\Big|^2\,dy  \notag \\
& \lesssim e^{-2\mu_1\tau} \sup_{0\le\tau'\le\tau}\int_\Omega\psi w^{1+\alpha + a} e^{\bar{S}} |X_r^a\slashed\partial^\beta D\uptheta|^2 \Big|\int_0^\tau| e^{\mu_1\tau'}\left|\partial_\tau\Lambda D{\bf V} \right|\,d\tau'\Big|^2 \,dy \notag \\
& \lesssim  e^{-2\mu_1\tau}  \sup_{0\le\tau'\le\tau}\left( \|X_r^a\slashed\partial^\beta D\uptheta\|_{1+\alpha+a,\psi e^{\bar{S}}}^2\right) \Big| \int_0^\tau e^{\mu_1\tau'}\left|\partial_\tau\Lambda\right| e^{-\mu_0\tau'}  (\mathcal{S}^N(\tau'))^{\frac12} d\tau'  \Big|^2  \notag \\
& \lesssim e^{-2\mu_1\tau} ( \mathcal{S}^N(\tau))^2, \label{E:CURLEST1}
\end{align}
where we applied Lemma \ref{L:EMBEDDING}. For $D_2$, one must first integrate by parts in $\tau$
\begin{align}
\frac1\mu\int_0^\tau \mu \partial_\tau\Lambda  D\eta \, X_r^a\slashed\partial^\beta D{\bf V} d\tau'
= & \frac1\mu\left(\mu \partial_\tau\Lambda  D\eta \, X_r^a\slashed\partial^\beta D \uptheta\right)\big|^\tau_0 \label{tau2}\\ 
&  - \frac1\mu\int_0^\tau \mu(\frac{\mu_\tau}\mu\partial_\tau\Lambda D\eta+\partial_{\tau\tau}\Lambda  D\eta+\partial_\tau\Lambda D{\bf V}) \, X_r^a\slashed\partial^\beta D \uptheta\,d\tau'. \notag
\end{align}
By the same argument as for $D_1$
\begin{align}
\|\frac1{\mu}\int_0^\tau\mu\partial_\tau\Lambda  D\eta \, X_r^a\slashed\partial^\beta D{\bf V} d\tau' \|_{1+\alpha+a,\psi e^{\bar{S}}}^2  \lesssim e^{-2\mu_1\tau}\mathcal{S}^N(0) + \left((1+\tau^2) e^{-2\mu_1\tau} + Q(\tau)\right)  \mathcal{S}^N(\tau),
\label{E:NOTCRITICAL}
\end{align}
with
\begin{equation}
Q(\tau)=
\begin{cases}
e^{-4\mu_0\tau} &\text{if } \ 1<\gamma<\frac43  \\
\tau^2e^{-4\mu_0\tau} &\text{if } \ \gamma=\frac43\\
e^{-2\mu_1\tau} &\text{if } \ \frac43<\gamma\le\frac53
\end{cases}. 
\end{equation}
For $D_3$, we also first integrate by parts in $\tau$
\begin{align}
\frac1\mu\int_0^\tau \mu \Lambda  D{\bf V} \, X_r^a\slashed\partial^\beta D{\bf V} d\tau'
= & \frac1\mu \left[\mu \Lambda  D{\bf V} \, X_r^a\slashed\partial^\beta D \uptheta\right]^\tau_0 \notag \\ 
&  - \frac1\mu\int_0^\tau \mu(\frac{\mu_\tau}\mu\Lambda D{\bf V}+\partial_{\tau}\Lambda  D{\bf V}+\Lambda D{\bf V}_\tau) \, X_r^a\slashed\partial^\beta D \uptheta\,d\tau' \label{E:CRITICAL}.
\end{align}
The first term on the right-hand side of (\ref{E:CRITICAL}) is bounded by
\begin{equation}
\mu^{-1} \mathcal{S}^N(0) + \|D{\bf V}\|_{L^\infty(\Omega)}(\mathcal{S}^N(\tau))^{\frac12} \lesssim e^{-\mu_1\tau} \mathcal{S}^N(0)+ e^{-\mu_0\tau}\mathcal{S}^N(\tau),
\end{equation}
where we applied Lemma \ref{L:EMBEDDING}. For the second term using the expontential boundedness of $\mu$ (\ref{E:EXPMU1MUINEQGAMMALEQ5OVER3})
\begin{align}
\Big|\frac1\mu\int_0^\tau \mu_\tau\Lambda D{\bf V}  X_r^a\slashed\partial^\beta D \uptheta\,d\tau'\Big| & \lesssim \sup_{0\le\tau'\le\tau}| X_r^a\slashed\partial^\beta D \uptheta| e^{-\mu_1\tau}\int_0^\tau e^{\mu_1\tau'}\|D{\bf V}\|_{L^\infty(\Omega)}\,d\tau'  \notag \\
 & \lesssim  \sup_{0\le\tau'\le\tau}| X_r^a\slashed\partial^\beta D \uptheta|(\mathcal{S}^N(\tau))^{\frac12}\underbrace{ e^{-\mu_1\tau}\int_0^\tau e^{(\mu_1-\mu_0)\tau'} d\tau' }_{=:L}.
\end{align}
Notice that after integration and using $\mu_0 \leq \mu_1$
\begin{equation}
L \lesssim 
\begin{cases}
e^{-\mu_0\tau} & \text{ if } \  1<\gamma<\frac53 \\
\tau e^{-\mu_0\tau} & \text{ if } \ \gamma =\frac53
\end{cases}.
\end{equation}
Hence 
\begin{align}
\|\frac1\mu\int_0^\tau \mu_\tau\Lambda D\uptheta  X_r^a\slashed\partial^\beta D \uptheta\,d\tau'\|_{1+\alpha+a,\psi e^{\bar{S}}}^2\lesssim 
\begin{cases}
e^{-2\mu_0\tau} (\mathcal{S}^N(\tau))^2 & \text{ if } \ 1<\gamma<\frac53 \\
\tau^2 e^{-2\mu_0\tau} (\mathcal{S}^N(\tau))^2 & \text{ if } \ \gamma =\frac53
\end{cases}.
\end{align}
The remaining terms on the right-hand side of (\ref{E:CRITICAL}) are estimated in a similar way except when considering the last term which includes $D{\bf V}_\tau$: there we use (\ref{E:THETAGAMMALEQ5OVER3}) to rewrite $D{\bf V}_\tau$, and then an analogous proof completes the argument. Therefore
\begin{align}\label{E:CURLEST3}
\|\frac1{\mu}\int_0^\tau \mu(\tau') D_3(\tau')\,d\tau'\|_{1+\alpha+a,\psi e^{\bar{S}}}^2 \lesssim
\begin{cases}
 e^{-2\mu_0\tau} \mathcal{S}^N(\tau)   & \text{ if } \ 1<\gamma<\frac53 \\  
(1+ \tau^2) e^{-2\mu_0\tau} \mathcal{S}^N(\tau) & \text{ if } \ \gamma =\frac53
\end{cases}.
\end{align} 
Finally using a similar approach
\begin{align}
& \Big\|\sum_{1\le a'+|\beta'|\le N-1} C_{a',\beta'}\left(\partial_\tau\Lambda_{jm}X_r^{a'}\slashed\partial^{\beta'}\mathscr{A}^s_m +\Lambda_{jm} X_r^{a'}\slashed\partial^{\beta'}\partial_\tau\mathscr{A}^s_m\right) 
 X_r^{a-a'}\slashed\partial^{\beta-\beta'} {\bf V}^k,_s\Big\|_{1+\alpha+a,\psi e^{\bar{S}}}^2 \notag \\
 &  \lesssim e^{-2\mu_0\tau} \mathcal{S}^N(\tau)\left(1+P(\mathcal{S}^N(\tau))\right) \notag \\
 & \lesssim e^{-2\mu_0\tau} \mathcal{S}^N(\tau).
\end{align}
where we have used Lemma \ref{L:EMBEDDING} and (\ref{E:APRIORI}). Thus we have
\begin{align}\label{E:CURLEST4}
\|\frac{1}{\mu} \int_0^\tau \mu X_r^a\slashed\partial^\beta[\partial_\tau, \text{Curl}_{\Lambda\mathscr{A}}] {\bf V}d\tau'\|_{1+\alpha+a,\psi e^{\bar{S}}}^2 \lesssim  
\begin{cases}
 e^{-2\mu_0\tau} \mathcal{S}^N(\tau)   & \text{ if } \ 1<\gamma<\frac53 \\  
(1+ \tau^2) e^{-2\mu_0\tau} \mathcal{S}^N(\tau) & \text{ if } \ \gamma =\frac53
\end{cases}. 
\end{align}
The sixth term on the right hand side of (\ref{E:POSTDERVCURLVEQNGAMMALEQ5OVER3}) is similar to the fifth term but lower order and hence is also bounded. Proceeding in an analogous way to the fifth term estimate, we have for the seventh term on the right hand side of (\ref{E:POSTDERVCURLVEQNGAMMALEQ5OVER3})
\begin{align}\label{E:CURLEST5}
\|\frac{2}{\mu}\int_0^\tau \mu X^a_r \slashed\partial^\beta\text{Curl}_{\Lambda\mathscr{A}}\left(\Gamma^\ast{\bf V}\right)d\tau'\|_{1+\alpha+a,\psi e^{\bar{S}}}^2 \lesssim 
\begin{cases}
 e^{-2\mu_0\tau} \mathcal{S}^N(\tau)   & \text{ if } \ 1<\gamma<\frac53 \\  
(1+ \tau^2) e^{-2\mu_0\tau} \mathcal{S}^N(\tau) & \text{ if } \ \gamma =\frac53
\end{cases}.
\end{align}
The eighth term is similar to the seventh term but lower order and hence is also bounded. Bounds on the last two terms on the right hand side of (\ref{E:POSTDERVCURLVEQNGAMMALEQ5OVER3}) follow from Lemma \ref{L:CURLLASTTERMBOUNDS}. Combining the above analysis, and an analogous argument for the $\|\cdot\|_{1+\alpha,(1-\psi)e^{\bar{S}}}$ norms, we obtain (\ref{E:CURLVBOUND}). \\

\textit{Proof of} (\ref{E:CURLTHETABOUND}). Apply $X_r^a \slashed\partial^\beta$ to (\ref{E:FINALCURLTHETAEQNGAMMALEQ5OVER3})
\begin{align}\label{E:POSTDERVCURLTHETAEQNGAMMALEQ5OVER3}
&\text{Curl}_{\Lambda\mathscr{A}} X_r^a \slashed\partial^\beta \uptheta  = -[X_r^a \slashed\partial^\beta,\text{Curl}_{\Lambda \mathscr{A}}]\uptheta + \tfrac{1}{\gamma} X_r^a \slashed\partial^\beta \Lambda\mathscr{A} \nabla (\bar{S}) \times \uptheta \notag \\
&+X_r^a \slashed\partial^\beta \text{Curl}_{\Lambda \mathscr{A}}([\uptheta(0)]) - \tfrac{1}{\gamma} X_r^a \slashed\partial^\beta (\Lambda\mathscr{A} \nabla (\bar{S}) \times \uptheta)(0) \notag \\ 
& +\mu(0) X_r^a \slashed\partial^\beta \text{Curl}_{\Lambda \mathscr{A}} ({\bf V}(0)) \int_0^\tau \frac{1}{\mu(\tau')} d \tau' - \frac{\mu(0) X_r^a \slashed\partial^\beta \Lambda\mathscr{A} \nabla (\bar{S}) \times \mathbf{V}(0)}{\gamma} \int_0^\tau \frac{1}{\mu(\tau')} d \tau' \notag \\
&  + \int_0^\tau  X_r^a \slashed\partial^\beta [\partial_\tau, \text{Curl}_{\Lambda\mathscr{A}}] {\uptheta} d\tau' -\frac{1}{\gamma}\int_0^\tau  X_r^a \slashed\partial^\beta [\partial_\tau,\Lambda \mathscr{A} \nabla ( \bar{S}) \times ] \uptheta d\tau ' \notag \\
& +\int_0^\tau \frac{1}{\mu(\tau')} \int_0^{\tau'} \mu(\tau'') X_r^a \slashed\partial^\beta [\partial_\tau, \text{Curl}_{\Lambda\mathscr{A}}] {\bf V} d\tau'' d \tau' -\frac{1}{\gamma} \int_0^{\tau} \frac{1}{\mu(\tau')} \int_0^{\tau'} \mu(\tau '') X_r^a \slashed\partial^\beta [\partial_\tau,\Lambda \mathscr{A} \nabla ( \bar{S}) \times ] \mathbf{V} d\tau '' d \tau ' \notag \\
& - \int_0^\tau \frac{2}{\mu(\tau')} \int_0^{\tau '} \mu(\tau'') \, X_r^a \slashed\partial^\beta \text{Curl}_{\Lambda\mathscr{A}}(\Gamma^\ast{\bf V}) d\tau'' d \tau' + \frac{1}{\gamma}\int_0^{\tau} \frac{2}{\mu(\tau')} \int_0^{\tau '} \mu(\tau'') \, X_r^a \slashed\partial^\beta \Lambda\mathscr{A} \nabla (\bar{S}) \times (\Gamma^\ast \mathbf{V}) d\tau '' d \tau ' \notag \\
& +\frac{\delta}{\gamma} \int_0^\tau \frac{2}{\mu(\tau')} \int_0^{\tau'} \mu(\tau'')^{4-3 \gamma} X_r^a \slashed\partial^\beta \Lambda \nabla(\bar{S}) \times  \Lambda \uptheta d \tau '' d \tau ' \notag \\
& -\frac{\delta}{\gamma} \int_0^\tau \frac{1}{\mu(\tau')}\int_0^{\tau'} \mu(\tau '')^{4-3 \gamma} X_r^a \slashed\partial^\beta \Lambda \mathscr{A}[D \uptheta]\nabla (\bar{S}) \times \Lambda \eta d \tau '' d \tau '.
\end{align}  
The bound on the first term on the right hand side of (\ref{E:POSTDERVCURLTHETAEQNGAMMALEQ5OVER3}) follows from Lemma \ref{L:COMMUTATORBOUNDS}. The second term is similar to the first term but lower order and hence is also bounded.

The third term is bounded by $\mathcal{S}^N(0)$. Noting the fact that $\| X_r^a\slashed\partial^\beta \mathscr{A}\|_{1+\alpha+a,\psi e^{\bar{S}}}^2 \lesssim \mathcal{S}^N$ established in the proof of Lemma \ref{L:COMMUTATORBOUNDS}, the fourth term is similar to the third term but lower order and hence is also bounded in the same way as the third term.

The fifth term is bounded by $\mathcal{S}^N(0)+\mathcal{B}^N[{\bf V}](0)$ using (\ref{E:EXPMU1MUINEQGAMMALEQ5OVER3}). Using again the bound on $X_r^a\slashed\partial^\beta \mathscr{A}$, the sixth term is similar to the fifth term but lower order and hence is also bounded in the same way as the third term.
 
For the seventh term estimate, a similar argument to that used to prove the bound (\ref{E:CURLEST4}) gives us
\begin{equation} 
\| X_r^a \slashed\partial^\beta [\partial_\tau, \text{Curl}_{\Lambda\mathscr{A}}] {\uptheta}\|  \lesssim e^{-\mu_0 \tau} (\mathcal{S}^N(\tau))^{\tfrac12}.
\end{equation}
Then via the Cauchy-Schwarz inequality and Fubini's Theorem, we have for the seventh term
\begin{equation} 
\Big\| \int_0^\tau  X_r^a \slashed\partial^\beta [\partial_\tau, \text{Curl}_{\Lambda\mathscr{A}}] {\uptheta} d\tau'  \Big\|_{1+\alpha+a,\psi e^{\bar{S}}}^2 \lesssim \int_0^\tau e^{-\mu_0\tau} \mathcal{S}^N(\tau')\,d\tau'.
\end{equation}   
The eighth term is similar to the seventh term but lower order and hence is also bounded.

For the ninth term 
\begin{align}
&\Big\|\int_0^\tau \frac{1}{\mu(\tau')} \int_0^{\tau'} \mu(\tau'') X_r^a \slashed\partial^\beta[\partial_\tau, \text{Curl}_{\Lambda\mathscr{A}}] {\bf V}d\tau''\,d\tau' \Big\|_{1+\alpha+a,\psi e^{\bar{S}}}^2 \notag \\
& \lesssim \int_\Omega\left[\int_0^\tau \frac{(1+\tau')^2}{\mu(\tau')}\,d\tau'  \int_0^\tau \frac1{(1+\tau')^2\mu(\tau')  }\left(\int_0^{\tau'} \mu(\tau'') X_r^a \slashed\partial^\beta [\partial_\tau, \text{Curl}_{\Lambda\mathscr{A}}] {\bf V}d\tau''\right)^2\,d\tau'\right] 
 w^{1+\alpha+a} e^{\bar{S}} \psi\,dx\notag \\
& \lesssim \int_0^\tau \frac1{(1+\tau')^2\mu(\tau')} \Big\|\int_0^{\tau'} \mu(\tau'') X_r^a \slashed\partial^\beta[\partial_\tau, \text{Curl}_{\Lambda\mathscr{A}}] {\bf V}d\tau''\Big\|_{1+\alpha+a,\psi e^{\bar{S}} }^2\,d\tau' \notag \\
& \lesssim \int_0^\tau e^{-\mu_0\tau} \mathcal{S}^N(\tau')\,d\tau', \notag 
\end{align} 
where the final bound follows from similar arguments to those giving the bound (\ref{E:CURLEST4}). The tenth term is similar to the ninth term but lower order and hence is also bounded. Similarly to the ninth term proof but via (\ref{E:CURLEST5}), for the eleventh term we have
\begin{align}
& \Big\|\int_0^\tau \frac{2}{\mu(\tau')}\int_0^{\tau'} \mu(\tau'') X_r^a \slashed\partial^\beta\text{Curl}_{\Lambda\mathscr{A}}\left(\Gamma^\ast{\bf V}\right)d\tau''\,d\tau' \Big\|_{1+\alpha+a,\psi e^{\bar{S}}}^2 \lesssim \int_0^\tau e^{-\mu_0\tau'} \mathcal{S}^N(\tau')\,d\tau'.  \notag
\end{align}
The twelfth term is similar to the eleventh term but lower order and hence is also bounded. Bounds on the last two terms on the right hand side of (\ref{E:POSTDERVCURLTHETAEQNGAMMALEQ5OVER3}) follow from Lemma \ref{L:CURLLASTTERMBOUNDS}. Combining the above analysis, and an analogous argument for the $\|\cdot\|_{1+\alpha,(1-\psi)e^{\bar{S}}}$ norms, we obtain (\ref{E:CURLTHETABOUND}).
\end{proof}

\section{Energy Estimates}\label{S:ENERGYGAMMALEQ5OVER3}  
Before proving our main energy inequality, we first introduce the two energy based high order quantities which arise directly from the problem. Begin by diagonalizing then positive symmetric matrix $\Lambda=(\det A)^{\frac{2}{3}} A^{-1} A^{-\top} \in \text{SL}(3)$ as follows   
\begin{equation}\label{E:LAMBDADECOMP}
\Lambda=P^{\top}QP, \quad P \in \text{SO(3)}, \quad Q=\text{diag}(d_1,d_2,d_3), \quad d_i > 0 \text{ eigenvalues of } \Lambda,
\end{equation}
and then define
\begin{equation}\label{E:MATHSCRMMATHSCRN}
\mathscr M_{a,\beta} : =  P\:\nabla_\eta X_r^a \slashed\partial^\beta\uptheta\:P^{\top}  \text{   and   }  \mathscr N_\nu : =  P\:\nabla_\eta \partial^\nu\uptheta\:P^{\top}.
\end{equation} 
Denoting the usual dot product on $\mathbb{R}^3$ by $\langle \cdot , \cdot \rangle$, introduce the high-order energy functional
\begin{align}
& \mathcal{E}^N(\uptheta,\mathbf{V})(\tau)=\mathcal{E}^N(\tau) \notag \\ 
&=\frac12\sum_{a+|\beta|\le N}
\int_\Omega\psi\Big[\mu^{3\gamma-3} \left\langle\Lambda^{-1}X_r^a\slashed\partial^\beta\, \mathbf{V},\,X_r^a\slashed\partial^\beta \mathbf{V}\right\rangle+\delta\left\langle\Lambda^{-1}X_r^a\slashed\partial^\beta \uptheta,\,X_r^a\slashed\partial^\beta\uptheta\right\rangle\Big] w^{a+\alpha} \notag \\
& \ \ \ \  \ \ \ \ \  \ \ \ \ \  \ \ \ \ \  \ \ \ \ \  +\psi\mathscr{J}^{-\frac1\alpha}\Big[\sum_{i,j=1}^3 d_id_j^{-1}\left((\mathscr{M}_{a,\beta})^j_{i}\right)^2+\tfrac1\alpha\left(\text{div}_\eta X_r^a\slashed\partial^\beta \uptheta\right)^2\Big] w^{a+\alpha+1}e^{\bar{S}} \,dy \notag \\
&\ \ \ \  +\frac12\sum_{|\nu|\le N}
\int_\Omega (1-\psi)\Big[\mu^{3\gamma-3} \left\langle \Lambda^{-1} \partial^\nu\mathbf{V},\,\partial^\nu \mathbf{V}\right\rangle +\delta \left\langle\Lambda^{-1}\partial^\nu\mathbf{\uptheta},\,\partial^\nu\mathbf{\uptheta}\right\rangle\Big]  w^\alpha
\notag \\
& \ \ \ \  \ \ \ \ \ \ \ \ \ \ \ \ \ \ \ \ \ \ \ \ +(1-\psi)\mathscr{J}^{-\frac1\alpha}\Big[ \sum_{i,j=1}^3 d_id_j^{-1}\left((\mathscr{N}_{\nu})^j_{i}\right)^2 + \tfrac1\alpha\left(\text{div}_\eta\partial^\nu \uptheta\right)^2\Big] w^{1+\alpha}e^{\bar{S}} \,dy,   \label{E:EDEFGAMMALEQ5OVER3}
\end{align}
and the dissipation functional
\begin{align}
\mathcal D^N (\mathbf{V}) = \mathcal D^N(\tau) =\frac{5-3\gamma}{2}\mu^{3\gamma-3}\frac{\mu_\tau}{\mu} &\int_\Omega  \Big[\psi \sum_{a+|\beta|\le N} \left\langle\Lambda^{-1}X_r^a\slashed\partial^\beta\mathbf{V},\,X_r^a\slashed\partial^\beta\mathbf{V}\right\rangle  w^{a+\alpha} \notag
\\
& \quad+ (1-\psi)\sum_{|\nu|\le N}\left\langle\Lambda^{-1}\partial^\nu\mathbf{V},\,\partial^\nu\mathbf{V}\right\rangle  w^{\alpha} \Big]\,dy.
\end{align}
We are requiring $\gamma \leq \frac{5}{3}$ in this formulation, which gives $\mathcal{D}^N (\mathbf{V}) \geq 0$. Next, we give key identities which will be used in our estimates. First from Lemma 4.3 \cite{1610.01666} we have the following modified energy identity:
\begin{lemma} \label{L:KEYLEMMA}
Recalling the matrix quantities introduced in (\ref{E:LAMBDADECOMP}) and (\ref{E:MATHSCRMMATHSCRN}), the following identities hold
\begin{align}
 \Lambda_{\ell j}(\nabla_\eta X_r^a \slashed\partial^\beta\uptheta)^i_j\Lambda^{-1}_{im}(\nabla_\eta X_r^a \slashed\partial^\beta\uptheta)^m_{\ell,\tau}  & =  \frac12\frac{d}{d\tau}\left(\sum_{i,j=1}^3 d_id_j^{-1}(\mathscr M_{a,\beta})^j_{i})^2\right) + \mathcal T_{a,\beta}, \label{E:KEYONE}\\
  \Lambda_{\ell j}(\nabla_\eta\partial^\nu\uptheta)^i_j\Lambda^{-1}_{im}(\nabla_\eta\partial^\nu\uptheta)^m_{\ell,\tau}  & =  \frac12\frac{d}{d\tau}\left(\sum_{i,j=1}^3 d_id_j^{-1}(\mathscr N_\nu)^j_{i})^2\right) + \mathcal T_{\nu}, \label{E:KEYTWO}
\end{align}
where the error terms $\mathcal T_{a,\beta}$ and $\mathcal T_{\nu}$ are given as follows
\begin{align}
\mathcal T_{a,\beta} & = -  \frac12\sum_{i,j=1}^3\frac d{d\tau}\left(d_id_j^{-1}\right)(\left(\mathscr M_{a,\beta}\right)^j_i)^2 
 - \text{{\em Tr}}\left(Q\mathscr M_{a,\beta}Q^{-1}\left(\partial_\tau P P^\top \mathscr M_{a,\beta}^\top + \mathscr M_{a,\beta}^\top P\partial_\tau P^\top\right)\right), \label{E:TABETA}\\
\mathcal T_{\nu} & = -  \frac12\sum_{i,j=1}^3\frac d{d\tau}\left(d_id_j^{-1}\right)(\left(\mathscr N_\nu\right)^j_i)^2 
 -  \text{{\em Tr}}\left(Q\mathscr N_\nu Q^{-1}\left(\partial_\tau P P^\top \mathscr N_\nu^\top + \mathscr N_\nu^\top P\partial_\tau P^\top\right)\right).
\end{align}
\end{lemma}
Next we have two important commutation results which will be used crucially when differentiating the nonlinear pressure term in our equation. 
\begin{lemma}\label{L:COMM}
For $q \in \mathbb R_+$, ${\bf T}:\Omega\to\mathbb M^{3\times3}$ and any $i,j,\ell\in\{1,2,3\}$ the following identities hold
\begin{align}
X_r\left[ \frac{1}{ w^q} ( w^{1+q}  e^{\bar{S}} {\bf T}^k_i),_k  \right] & =  \frac{1}{ w^{1+q}} ( w^{2+q} e^{\bar{S}} X_r {\bf T}^k_i),_k + \mathcal C_i^{q+1}[{\bf T}], \label{E:COMMXR}\\ 
\slashed\partial_{j\ell}\left[ \frac{1}{ w^q} ( w^{1+q} e^{\bar{S}} {\bf T}^k_i),_k  \right] &=  \frac{1}{ w^{q}} ( w^{1+q} e^{\bar{S}} \slashed\partial_{j\ell}{\bf T}^k_i),_k + \mathcal C_{ij\ell}^{q}[{\bf T}], \label{E:COMMSLASHEDPARTIAL} 
\end{align}
where $ \mathcal C_i^{q+1}[{\bf T}]$ is
\begin{align}\label{E:COMMXRC}
\mathcal C_i^{q+1}[{\bf T}] &= (X_r(we^{\bar{S}})-we^{\bar{S}})\frac{y_j}{r^2}\slashed\partial_{jk}{\bf T}_i^k+w(X_r(e^{\bar{S}})-e^{\bar{S}})\frac{y_k}{r^2}X_r {\bf T}_i^k \notag \\
&+\left[(1+q) X_r (w,_k e^{\bar{S}})+ X_r((e^{\bar{S}}),_kw)\right]{\bf T}_i^k,
\end{align}
and $\mathcal C_{ij\ell}^{q}[{\bf T}]$ is
\begin{align}
&\mathcal C_{ij\ell}^{q}[{\bf T}]=  we^{\bar{S}}\left(y_m\left[\frac{\delta_{k \ell} \slashed\partial_{m j} - \delta_{k j} \slashed \partial_{m \ell}}{r^2}\right]+\left[\frac{\delta_{k \ell}y_j - \delta_{k j} y_{\ell}}{r^2}\right]X_r\right){\bf T}^k_i \notag \\
&\qquad \qquad \qquad +(1+q)(\slashed \partial_{j\ell} w,_k)e^{\bar{S}}{\bf T}^k_i+w \slashed\partial_{j\ell} (e^{\bar{S}}),_k {\bf T}_i^k \ \ \text{for } j,\ell=1,2,3. \label{E:COMMSLASHEDPARTIALC}
\end{align}
\end{lemma} 
\begin{proof}
Compute
\begin{align*}
&X_r\left[ \frac{1}{ w^q} ( w^{1+q} e^{\bar{S}} {\bf T}^k_i),_k  \right] = X_r \left[  w e^{\bar{S}} {\bf T}^k_i,_k + (1+q)  w,_k e^{\bar{S}} {\bf T}^k_i + (e^{\bar{S}}),_k w T_i^k\right] \\
&=w e^{\bar{S}} X_r {\bf T}^k_i,_k + X_r(w e^{\bar{S}}){\bf T}^k_i,_k + (1+q)w,_ke^{\bar{S}}X_r{\bf T}_i^k + (1+q){\bf T}_i^k X_r (w,_k e^{\bar{S}}) \\
&  \qquad \qquad \qquad + (e^{\bar{S}}),_k w X_r{\bf T}_i^k + {\bf T}_i^k X_r((e^{\bar{S}}),_kw) \\
&=we^{\bar{S}}(X_r {\bf T}^k_i),_k + w e^{\bar{S}}[X_r,\partial_k]{\bf T}_i^k + (2+q)w,_k e^{\bar{S}}X_r {\bf T}_i^k-w,_ke^{\bar{S}}X_r {\bf T}_i^k \\
&\qquad \qquad \qquad + (e^{\bar{S}}),_k w X_r {\bf T}_i^k+X_r(we^{\bar{S}}){\bf T}^k_i,_k+(1+q){\bf T}_i^k X_r (w,_k e^{\bar{S}})+{\bf T}_i^k X_r((e^{\bar{S}}),_kw) \\
&= \frac{1}{ w^{1+q}} ( w^{2+q} e^{\bar{S}} X_r {\bf T}^k_i),_k + w e^{\bar{S}}[X_r,\partial_k]{\bf T}_i^k -w,_ke^{\bar{S}}X_r {\bf T}_i^k +X_r(we^{\bar{S}}){\bf T}^k_i,_k\\
& \qquad \qquad \qquad +(1+q){\bf T}_i^k X_r (w,_k e^{\bar{S}})+{\bf T}_i^k X_r((e^{\bar{S}}),_kw).
\end{align*}
We have now obtained favorable terms except for the following three terms which we rewrite using our commutator identity (\ref{E:DERIVATIVECOMMUTATOR}), $[X_r,\partial_k]= -\partial_k$ and (\ref{E:DERIVATIVEDECOMP}) to express $\partial_k$ in terms of $X_r$ and $\slashed\partial_{jk}$,
\begin{align*}
&w e^{\bar{S}}[X_r,\partial_k]{\bf T}_i^k-w,_ke^{\bar{S}}X_r {\bf T}_i^k+X_r(we^{\bar{S}}){\bf T}^k_i,_k = (X_r(we^{\bar{S}})-we^{\bar{S}}){\bf T}^k_i,_k-w,_ke^{\bar{S}}X_r {\bf T}_i^k \\
&=(X_r(we^{\bar{S}})-we^{\bar{S}}) \left(\frac{y_j}{r^2}\slashed\partial_{jk}{\bf T}_i^k+\frac{y_k}{r^2}X_r {\bf T}_i^k\right)-w,_ke^{\bar{S}}X_r {\bf T}_i^k \\
&=(X_r(we^{\bar{S}})-we^{\bar{S}})\frac{y_j}{r^2}\slashed\partial_{jk}{\bf T}_i^k+w(X_r(e^{\bar{S}})-e^{\bar{S}})\frac{y_k}{r^2}X_r {\bf T}_i^k,
\end{align*}
where to obtain the last line we have used the fact that $w$ is a radial function, say $f(|y|)$, by (\ref{E:WDEMAND}) and hence $w,_k=\frac{y_k}{r}f'(|y|)$ which implies
\begin{align*}
(X_r w)\frac{y_k}{r^2} - w,_k &=y_{\ell}\frac{y_\ell}{r}f'(|y|)\frac{y_k}{r^2}-\frac{y_k}{r}f'(|y|) \\
&=0.
\end{align*} 
Combining the above calculations we obtain (\ref{E:COMMXR}). Now since $w$ and $\bar{S}$ are radial functions by (\ref{E:BARSDEMAND}) and (\ref{E:WDEMAND}), $\slashed\partial_{j\ell} w = \slashed\partial_{j\ell} (e^{\bar{S}})=0$ and so we have
\begin{align*}
&\slashed\partial_{j\ell}\left[ \frac{1}{ w^q} ( w^{1+q} e^{\bar{S}} {\bf T}^k_i),_k  \right]=\slashed\partial_{j\ell} \left[  w e^{\bar{S}} {\bf T}^k_i,_k + (1+q)  w,_k e^{\bar{S}} {\bf T}^k_i + (e^{\bar{S}}),_k w {\bf T}_i^k\right] \\
&=we^{\bar{S}}\slashed\partial_{j\ell} {\bf T}^k_i,_k + (1+q)w,_ke^{\bar{S}}\slashed\partial_{j\ell}{\bf T}^k_i + (1+q)(\slashed\partial_{j\ell} w,_k)e^{\bar{S}}{\bf T}^k_i \\
&\qquad \qquad \qquad +w(e^{\bar{S}}),_k \slashed\partial_{j\ell} {\bf T}^k_i,_k + w \slashed\partial_{j\ell} (e^{\bar{S}}),_k {\bf T}_i^k \\
&=we^{\bar{S}}(\slashed\partial_{j\ell} {\bf T}^k_i),_k + w e^{\bar{S}} [\slashed\partial_{j\ell},\partial_k]{\bf T}^k_i+ (1+q)w,_ke^{\bar{S}}\slashed\partial_{j\ell}{\bf T}^k_i + (1+q)(\slashed\partial_{j\ell} w,_k)e^{\bar{S}}{\bf T}^k_i\\
& \qquad \qquad \qquad +w(e^{\bar{S}}),_k \slashed\partial_{j\ell} {\bf T}^k_i,_k + w\slashed\partial_{j\ell} (e^{\bar{S}}),_k {\bf T}_i^k \\
&=\frac{1}{w^q} ( w^{1+q} e^{\bar{S}} \slashed\partial_{j\ell}{\bf T}^k_i ),_k +w e^{\bar{S}} (\delta_{k\ell }\partial_j-\delta_{kj}\partial_\ell){\bf T}^k_i+(1+q)(\slashed
\partial_{j\ell} w,_k)e^{\bar{S}}{\bf T}^k_i+w \slashed\partial_{j\ell} (e^{\bar{S}}),_k {\bf T}_i^k \\
&=\frac{1}{w^q} ( w^{1+q} e^{\bar{S}} \slashed\partial_{j\ell}{\bf T}^k_i ),_k+we^{\bar{S}}\left(y_m\left[\frac{\delta_{k \ell} \slashed\partial_{m j} - \delta_{k j} \slashed \partial_{m \ell}}{r^2}\right]+\left[\frac{\delta_{k \ell}y_j - \delta_{k j} y_{\ell}}{r^2}\right]X_r\right){\bf T}^k_i\\
&\qquad \qquad \qquad +(1+q)(\slashed \partial_{j\ell} w,_k)e^{\bar{S}}{\bf T}^k_i+w \slashed\partial_{j\ell} (e^{\bar{S}}),_k {\bf T}_i^k,
\end{align*}
where $[\slashed\partial_{j\ell},\partial_k]{\bf T}^k_i$ has been rewritten using (\ref{E:DERIVATIVECOMMUTATOR}) and (\ref{E:DERIVATIVEDECOMP}).
\end{proof}
In the next Lemma, we give some useful results concerning our quantities $\mathscr{A}$, $\mathscr{J}$ and $\Lambda$ and also our derivative operators,
\begin{lemma}\label{L:USEFULIDENTITIESPREENERGYINEQUALITY}       
For $\mathscr{A}$, $\mathscr{J}$ and $\Lambda$, the following identities hold
\begin{align}
\mathscr{A}^k_j\mathscr{J}^{-\frac1\alpha} - \delta^k_j &= ( \mathscr{A}^k_j-\delta^k_j) \mathscr{J}^{-\frac1\alpha} + \delta^k_j (\mathscr{J}^{-\frac1\alpha} -1),  \label{E:AJIDENTITYENERGY} \\  
\mathscr{A}^k_j-\delta^k_j &= -\mathscr{A}^k_l  [D\uptheta]^l_j,  \label{E:AIDENTITYENERGY}\\
\Lambda_{ij}&=\Lambda_{ip}(\mathscr{A}^j_p +  \mathscr{A}^j_l\uptheta^l,_p). \label{E:LAMBDAIDENTITYENERGY}
\end{align}  
For any derivative operator $L \in \{\slashed\partial_{ij},X_r\}_{i,j=1,2,3}$
\begin{align}      
L(\mathscr{A}_j^k \mathscr{J}^{-\frac{1}{\alpha}}) &=- \mathscr{J}^{-\frac1\alpha}\mathscr{A}^k_\ell\mathscr{A}^s_j (L\uptheta^\ell),_s  -\tfrac1\alpha \mathscr{J}^{-\frac1\alpha}\mathscr{A}^k_j\mathscr{A}^s_\ell (L\uptheta^\ell),_s \notag \\
&-  \mathscr{J}^{-\frac1\alpha}\mathscr{A}^k_\ell\mathscr{A}^s_j[L,\partial_s]\uptheta^\ell -\tfrac1\alpha \mathscr{J}^{-\frac1\alpha}\mathscr{A}^k_j\mathscr{A}^s_\ell [L,\partial_s]\uptheta^\ell. \label{E:HIGHORDERDERVIATIVEAJGAMMALEQ5OVER3} 
\end{align}
\end{lemma}  
\begin{proof} 
First (\ref{E:AJIDENTITYENERGY}) is straightforward to verify by expanding the right-hand side. Second, (\ref{E:AIDENTITYENERGY}) was proven in Section \ref{S:CURLGAMMALEQ5OVER3}, (\ref{E:CURLADELTAIDENTITYPROOF}). Next, (\ref{E:LAMBDAIDENTITYENERGY}) is proven in the following calculation where we recall $A=[D \eta]^{-1}$ and use $\eta=y+\uptheta$,
\begin{equation}   
\Lambda_{ij}=\Lambda_{ip}\delta^j_p = \Lambda_{ip} \mathscr{A}^j_l\eta^l,_p=\Lambda_{ip}(\mathscr{A}^j_p +  \mathscr{A}^j_l\uptheta^l,_p).
\end{equation}
Lastly, using our differentiation formulae for $\mathscr{A}$ and $\mathscr{J}$ (\ref{E:AJDIFFERENTIATIONFORMULAE}) which hold when generalized to $X_r$ and $\slashed\partial$, and also again the formula $\eta=y+\uptheta$,
\begin{align}     
&L(\mathscr{A}_j^k \mathscr{J}^{-\frac{1}{\alpha}}) = - \mathscr{J}^{-\frac1\alpha}\mathscr{A}^k_\ell\mathscr{A}^s_j L(\eta^\ell,_s) -\tfrac1\alpha \mathscr{J}^{-\frac1\alpha}\mathscr{A}^k_j\mathscr{A}^s_\ell L(\eta^\ell,_s) \notag \\         
&= - \mathscr{J}^{-\frac1\alpha}\mathscr{A}^k_\ell\mathscr{A}^s_j L(y^\ell,_s+\uptheta^\ell,_s) -\tfrac1\alpha \mathscr{J}^{-\frac1\alpha}\mathscr{A}^k_j\mathscr{A}^s_\ell L(y^\ell,_s+\uptheta^\ell,_s) \notag \\
&=- \mathscr{J}^{-\frac1\alpha}\mathscr{A}^k_\ell\mathscr{A}^s_j L(\uptheta^\ell,_s) -\tfrac1\alpha \mathscr{J}^{-\frac1\alpha}\mathscr{A}^k_j\mathscr{A}^s_\ell L(\uptheta^\ell,_s) \notag \\
&=- \mathscr{J}^{-\frac1\alpha}\mathscr{A}^k_\ell\mathscr{A}^s_j (L\uptheta^\ell),_s  -\tfrac1\alpha \mathscr{J}^{-\frac1\alpha}\mathscr{A}^k_j\mathscr{A}^s_\ell (L\uptheta^\ell),_s \notag \\
&-  \mathscr{J}^{-\frac1\alpha}\mathscr{A}^k_\ell\mathscr{A}^s_j[L,\partial_s]\uptheta^\ell -\tfrac1\alpha \mathscr{J}^{-\frac1\alpha}\mathscr{A}^k_j\mathscr{A}^s_\ell [L,\partial_s]\uptheta^\ell.
\end{align}      
\end{proof}
Finally, before we prove our main energy inequality, it is worth formally stating the equivalence of our high order norm $\mathcal{S}^N$ and high order energy functional $\mathcal{E}^N$.
\begin{lemma}\label{L:NORMENERGY} 
Let $(\uptheta, {\bf V}):\Omega \rightarrow \mathbb R^3\times \mathbb R^3$ be a unique local solution to (\ref{E:THETAGAMMALEQ5OVER3})-(\ref{E:THETAICGAMMALEQ5OVER3}) on $[0,T]$ with $T>0$ fixed and assume $(\uptheta, {\bf V})$ satisfies the a priori assumptions (\ref{E:APRIORI}).  Fix $N\geq 2\lceil \alpha \rceil +12$. Let $k \geq N+1$ in (\ref{E:PHIDEMAND}). Then there are constants $C_1,C_2>0$ so that 
\begin{align}  
C_1\mathcal{S}^N (\tau) \le \sup_{0\le\tau'\le\tau}\mathcal{E}^N(\tau') \le C_2\mathcal{S}^N(\tau). 
\end{align} 
\end{lemma}
\begin{proof}
Recall the quantities introduced in Section \ref{S:HOQGAMMALEQ5OVER3}: notably the definitions $\mathcal{S}^N$ (\ref{E:SNNORMGAMMALEQ5OVER3}) and $\mathcal{E}^N$ (\ref{E:EDEFGAMMALEQ5OVER3}), and for $\mathcal{E}^N$ the associated decomposition of $\Lambda$ (\ref{E:LAMBDADECOMP}) and definition of the conjugates $\mathscr{M}_{a,\beta}$, $\mathscr{N}_{\nu}$ (\ref{E:MATHSCRMMATHSCRN}). Then the equivalence of $\mathcal{S}^N$ and $\mathcal{E}^N$ is a straightforward application of Lemma \ref{L:USEFULTAULEMMAGAMMALEQ5OVER3} to give bounds on $\Lambda$ and associated matrix quantities, and importantly for the nonisentropic setting, we also use the positive lower and upper bounds of $e^{\bar{S}}$ given by Corollary \ref{C:ENTROPYREGULARITYCOROLLARY} to conclude the result.
\end{proof}
We are now ready to prove our central energy inequality which will be essential in the proof of our main result Theorem \ref{T:MAINTHEOREMGAMMALEQ5OVER3}.     
\begin{proposition}\label{P:ENERGYESTIMATEGAMMALEQ5OVER3}
Suppose $\gamma \in (1,\frac{5}{3}]$. Let $(\uptheta, {\bf V}):\Omega \rightarrow \mathbb R^3\times \mathbb R^3$ be a unique local solution to (\ref{E:THETAGAMMALEQ5OVER3})-(\ref{E:THETAICGAMMALEQ5OVER3}) on $[0,T]$ with $T>0$ fixed and assume $(\uptheta, {\bf V})$ satisfies the a priori assumptions (\ref{E:APRIORI}). Fix $N\geq 2\lceil \alpha \rceil +12$. Let $k \geq N+1$ in (\ref{E:PHIDEMAND}). Then for all $\tau \in [0,T]$, we have the following inequality for some $0<\kappa\ll 1$
\begin{align}
& \mathcal E^N(\tau) +\int_0^\tau \mathcal D^N(\tau')\,d\tau' \lesssim  \mathcal{S}^N(0)  +\mathcal{B}^N[\uptheta](\tau) + \int_0^\tau  (\mathcal{S}^N(\tau'))^{\frac12} (\mathcal{B}^N[{\bf V}](\tau'))^{\frac12}\,d\tau'  \notag \\
& \qquad \qquad \qquad \qquad \qquad \qquad  + \kappa\mathcal{S}^N(\tau) + \int_0^\tau e^{-\mu_0\tau'} \mathcal{S}^N(\tau') d\tau'. \label{E:ENERGYMAINGAMMALEQ5OVER3}
\end{align}
\end{proposition}
\begin{proof}
\textbf{Zeroth order estimate.} Multiplying (\ref{E:THETAGAMMALEQ5OVER3}) by $\Lambda_{im}^{-1}\partial_\tau\uptheta^m$ and integrating over $\Omega$
\begin{align}\label{E:ZEROORDERPOSTIPGAMMALEQ5OVER3}
\int_\Omega w^\alpha ( \mu^{3 \gamma -3} \partial_{\tau \tau} \uptheta_i &+  \mu^{3 \gamma -4} \mu_\tau \partial_\tau \uptheta_i  + 2 \mu^{3 \gamma -3} \Gamma^*_{ij} \partial_{\tau} \uptheta_j + \delta w^\alpha \Lambda_{i \ell} \uptheta_\ell )\Lambda_{im}^{-1}\partial_\tau\uptheta^m \, dy \notag \\
&+\int_\Omega \left(w^{1+\alpha} e^{\bar{S}} \Lambda_{ij} \left( \mathscr{A}_j^k \mathscr{J}^{-\frac{1}{\alpha}}-\delta_j^k\right)\right),_k  \Lambda_{im}^{-1} \partial_\tau \uptheta^m \, dy=0.
\end{align}
Recognizing the perfect time derivative structure of the first integral
\begin{align}\label{E:ZEROORDERPOSTIPINT1GAMMALEQ5OVER3}
&\frac12\frac{d}{d\tau}\left(\mu^{3\gamma-3} \int_\Omega w^\alpha  \langle  \Lambda^{-1}\partial_\tau\uptheta, \partial_\tau\uptheta \rangle \, dy  +\delta \int_\Omega w^\alpha |\uptheta |^2 dy\right)+\frac{5-3\gamma}{2}\mu^{3\gamma-4}\mu_\tau \int_\Omega  w^\alpha \langle \Lambda^{-1}\partial_\tau\uptheta, \partial_\tau\uptheta \rangle \, dy \notag \\
&\qquad \qquad -\frac{\mu^{3\gamma-3}}{2}\int_\Omega  w^\alpha \langle \partial_\tau\Lambda^{-1}\partial_\tau\uptheta, \partial_\tau \uptheta \rangle \, dy + 2\mu^{3\gamma-3} \int_\Omega  w^\alpha \langle \Lambda^{-1}\partial_\tau\uptheta, \Gamma^\ast \partial_\tau \uptheta \rangle  \, dy.
\end{align}
After integrating from $0$ to $\tau$, we see the first three terms in (\ref{E:ZEROORDERPOSTIPINT1GAMMALEQ5OVER3}) will contribute to the left hand side of (\ref{E:ENERGYMAINGAMMALEQ5OVER3}), and also to $\mathcal{S}^N(0)$ on the right hand side of (\ref{E:ENERGYMAINGAMMALEQ5OVER3}) by the fundamental theorem of calculus and the equivalence of our norm and energy functional, given by Lemma \ref{L:NORMENERGY}. For the last two terms in (\ref{E:ZEROORDERPOSTIPINT1GAMMALEQ5OVER3}), before time integration, we rewrite them as
\begin{equation}\label{E:ZEROORDERPOSTIPINT1LAST2GAMMALEQOVER3} 
-\frac{\mu^{3\gamma-3}}{2}\int_\Omega  w^\alpha \langle [\partial_\tau\Lambda^{-1}-4\Lambda^{-1}\Gamma^\ast] \partial_\tau\uptheta, \partial_\tau \uptheta \rangle \, dy.
\end{equation} 
Now with $A=\mu O$, $\mu=(\det A)^{\frac{1}{3}}, O \in \text{SL}(3)$, we have
\begin{equation} 
\partial_\tau \Lambda^{-1}-4 \Lambda^{-1}\Gamma^\ast = \partial_\tau(O^\top O) - 4 O^\top OO^{-1}O_\tau = - \partial_\tau\Lambda^{-1}  + 2(O^\top_\tau O - O^\top O_\tau).
\end{equation}
Since $O^\top_\tau O - O^\top O_\tau$ is anti-symmetric, we can reduce (\ref{E:ZEROORDERPOSTIPINT1LAST2GAMMALEQOVER3}) to
\begin{equation}
\frac{\mu^{3\gamma-3}}{2}\int_\Omega  w^\alpha \langle \partial_\tau\Lambda^{-1}\partial_\tau\uptheta, \partial_\tau \uptheta \rangle \, dy,
\end{equation}
which is then bounded by $e^{-\mu_0 \tau}\mathcal{S}^N(\tau)$ via $\partial_\tau\Lambda^{-1}=-\Lambda^{-1}(\partial_\tau \Lambda)\Lambda^{-1}$ and (\ref{E:LAMBDABOUNDSGAMMALEQ5OVER3}). So the last two terms in (\ref{E:ZEROORDERPOSTIPINT1GAMMALEQ5OVER3}) contribute to $\int_0^\tau e^{-\mu_0\tau'} \mathcal{S}^N(\tau') d\tau'$ in (\ref{E:ENERGYMAINGAMMALEQ5OVER3}) after time integration. Returning to the second integral in (\ref{E:ZEROORDERPOSTIPGAMMALEQ5OVER3}), we integrate by parts and apply the identities (\ref{E:AJIDENTITYENERGY})-(\ref{E:LAMBDAIDENTITYENERGY}) proven in Lemma \ref{L:USEFULIDENTITIESPREENERGYINEQUALITY} to obtain
\begin{align}
&-\int_\Omega w^{1+\alpha} e^{\bar{S}} \Lambda_{ij} \left( \mathscr{A}_j^k \mathscr{J}^{-\frac{1}{\alpha}}-\delta_j^k\right)  \Lambda_{im}^{-1} \partial_\tau \uptheta^m,_k \, dy \notag \\
&= \int_\Omega  w^{1+\alpha} e^{\bar{S}} \mathscr{J}^{-\frac1\alpha} \Lambda_{ij} \mathscr{A}^k_l \uptheta^l,_j  \Lambda_{im}^{-1}\partial_\tau \uptheta^m,_k   dy - \int_\Omega  w^{1+\alpha} e^{\bar{S}} (\mathscr{J}^{-\frac1\alpha} -1) \partial_\tau \uptheta^k,_k  dy \notag \\
&= \int_\Omega  w^{1+\alpha} e^{\bar{S}} \mathscr{J}^{-\frac1\alpha}  \Lambda_{ip} \mathscr{A}^j_p  \uptheta^\ell,_j  \Lambda_{im}^{-1} \mathscr{A}^k_\ell \partial_\tau \uptheta^m,_k   dy  \notag \\
& + \int_\Omega  w^{1+\alpha} e^{\bar{S}} \mathscr{J}^{-\frac1\alpha} \Lambda_{ip} \mathscr{A}^j_l\uptheta^l,_p \mathscr{A}^k_\ell \uptheta^\ell,_j  \Lambda_{im}^{-1}\partial_\tau \uptheta^m,_k  dy - \int_\Omega  w^{1+\alpha} e^{\bar{S}} (\mathscr{J}^{-\frac1\alpha} -1) \partial_\tau \uptheta^k,_k  dy \notag \\
&= \int_\Omega  w^{1+\alpha}e^{\bar{S}}  \mathscr{J}^{-\frac1\alpha} \Lambda_{\ell j}[\nabla_\eta\uptheta]^i_j  \Lambda_{im}^{-1} [\nabla_\eta \partial_\tau \uptheta ]^m_\ell dy +\int_\Omega  w^{1+\alpha}e^{\bar{S}}   \mathscr{J}^{-\frac1\alpha} [\text{Curl}_{\Lambda\mathscr{A}} \uptheta]^\ell_i  \Lambda_{im}^{-1} [\nabla_\eta \partial_\tau \uptheta ]^m_\ell dy \notag \\
&\qquad \qquad + \int_\Omega  w^{1+\alpha} e^{\bar{S}} \mathscr{J}^{-\frac1\alpha}  \mathscr{A}^j_l\uptheta^l,_m \mathscr{A}^k_\ell \uptheta^\ell,_j \partial_\tau \uptheta^m,_k   dy-\int_\Omega  w^{1+\alpha} e^{\bar{S}} (\mathscr{J}^{-\frac1\alpha} -1) \partial_\tau \uptheta^k,_k  dy \notag \\
&=:(i)+(ii)+(iii)+(iv).
\end{align}
By Lemma \ref{L:KEYLEMMA} 
\begin{align}
&(i)=\frac{1}{2}\frac{d}{d\tau} \int_\Omega  w^{1+\alpha}e^{\bar{S}} \mathscr{J}^{-\frac1\alpha}\sum_{i,j=1}^3 d_i d_j^{-1}(({\mathscr M}_{0,0})^j_i)^2 dy  + \int_\Omega  w^{1+\alpha} e^{\bar{S}} \mathscr{J}^{-\frac1\alpha} \mathcal T_{0,0} dy \notag \\
&+ \frac{1}{2\alpha}\int_\Omega  w^{1+\alpha}e^{\bar{S}} \mathscr{J}^{-\frac1\alpha-1}\mathscr{J}_\tau \sum_{i,j=1}^3 d_i d_j^{-1}(({\mathscr M}_{0,0})^j_i)^2 dy \notag \\
&+ \int_\Omega w^{1+\alpha}e^{\bar{S}}\mathscr{J}^{-\frac{1}{\alpha}}\Lambda_{\ell j} [\nabla_\eta \uptheta]_i^j \Lambda_{im}^{-1} [\nabla_\eta \uptheta]_p^m [\nabla_\eta \partial_{\tau} \uptheta]_\ell^p dy, \label{E:ZEROORDERIEXPGAMMALEQ5OVER3}
\end{align}
with the last term above arising from commuting $\partial_\tau$ and $\nabla_\eta$ which is needed for Lemma \ref{L:KEYLEMMA}. Now after time integration, the first term in (\ref{E:ZEROORDERIEXPGAMMALEQ5OVER3}) contributes to $\mathcal{E}^N(\tau)$ in (\ref{E:ENERGYMAINGAMMALEQ5OVER3}), and also to $\mathcal{S}^N(0)$ by the fundamental theorem of calculus and Lemma \ref{L:NORMENERGY}.

Before time integration, the second term in (\ref{E:ZEROORDERIEXPGAMMALEQ5OVER3}) is bounded by $e^{-\mu_0 \tau} \mathcal{S}^N(\tau)$ due to the desirable form of $\mathcal T_{0,0}$ (\ref{E:TABETA}). In particular $\mathcal T_{0,0}$ contains $\tau$ derivatives of $P$ and $d_i$ which are exponentially bounded using (\ref{E:EIGENVALUEPLUSPBOUNDGAMMALEQ5OVER3}). Also before time integration, the third term in (\ref{E:ZEROORDERIEXPGAMMALEQ5OVER3}) is bounded by $e^{-\mu_0 \tau} \mathcal{S}^N(\tau)$ due to Lemma \ref{L:EMBEDDING} which we use to bound the $D \uptheta_\tau$ term arising from $\mathscr{J}_\tau$. Hence the second and third terms contribute to $\int_0^\tau e^{-\mu_0\tau'} \mathcal{S}^N(\tau') d\tau'$ in (\ref{E:ENERGYMAINGAMMALEQ5OVER3}) after time integration.

For the last term in (\ref{E:ZEROORDERIEXPGAMMALEQ5OVER3}), note schematically we can write, 
$$\mathscr{J}^{-\frac{1}{\alpha}}\Lambda_{\ell j} [\nabla_\eta \uptheta]_i^j \Lambda_{im}^{-1} [\nabla_\eta \uptheta]_p^m [\nabla_\eta \partial_{\tau} \uptheta]_\ell^p \text{ as } \mathscr{J}^{-\frac{1}{\alpha}}\Lambda \mathscr{A} D \uptheta \Lambda^{-1} \mathscr{A} D \uptheta \mathscr{A} D \uptheta_\tau.$$
Then using the a priori assumptions (\ref{E:APRIORI}) to bound $\mathscr{J}^{-\frac{1}{\alpha}}$ and $\mathscr{A}$, the boundedness of $\Lambda$ and $\Lambda^{-1}$ (\ref{E:LAMBDABOUNDSGAMMALEQ5OVER3}), and the exponential bound on $D \uptheta_\tau$ from Lemma \ref{L:EMBEDDING}, we obtain
\begin{align}
\int_0^\tau \int_\Omega w^{1+\alpha}e^{\bar{S}} \mathscr{J}^{-\frac{1}{\alpha}}\Lambda \mathscr{A} D \uptheta \Lambda^{-1} \mathscr{A} D \uptheta \mathscr{A} D \uptheta_\tau dy d \tau' & \lesssim \int_0^\tau e^{-\mu_0\tau'} (\mathcal{S}^N(\tau'))^{\tfrac32} d \tau' \notag \\
& \lesssim \int_0^\tau e^{-\mu_0\tau'} \mathcal{S}^N(\tau') d \tau',
\end{align}
where we have also used the a priori assumption $\mathcal{S}^N(\tau) < \frac{1}{3}$.

Now $(iii)$ is similar to the last term in (\ref{E:ZEROORDERIEXPGAMMALEQ5OVER3}) and so an analogous argument used to bound that term above will give us that $(iii)$ also contributes to $\int_0^\tau e^{-\mu_0\tau'} \mathcal{S}^N(\tau') d\tau'$. We consider $(ii)$. Schematically
\begin{equation}\label{E:IISCHEMATIC}
[\text{Curl}_{\Lambda\mathscr{A}} \uptheta] \mathscr{A} D \uptheta_\tau = \frac{1}{2} \frac{d}{d \tau}([\text{Curl}_{\Lambda\mathscr{A}} \uptheta]^2) - (\partial_\tau \Lambda) \mathscr{A} D \uptheta \Lambda \mathscr{A} D \uptheta - \Lambda (\partial_\tau \mathscr{A}) D \uptheta \Lambda  \mathscr{A} D \uptheta.
\end{equation}
Then using the boundedness of $\mathscr{J}^{-\frac1\alpha}$ and $\Lambda^{-1}$ we have that after time integration, the first term on the right hand side of (\ref{E:IISCHEMATIC}) will lead to $(ii)$ contributing to $\mathcal{B}^N[\uptheta]$ in (\ref{E:ENERGYMAINGAMMALEQ5OVER3}). Now using the exponential boundedness of $\partial_\tau \Lambda$ and applying Lemma \ref{L:EMBEDDING} to $\partial_\tau \mathscr{A}$, an analogous argument to that used to bound the last term in (\ref{E:ZEROORDERIEXPGAMMALEQ5OVER3}) will give us that the last two terms on the right hand side of (\ref{E:IISCHEMATIC}) will lead to contributions to $\int_0^\tau e^{-\mu_0\tau'} \mathcal{S}^N(\tau') d\tau'$ after time integration.

For $(iv)$, write
 \begin{equation}
 1-\mathscr{J}^{-\frac1\alpha}=\frac{1}{\alpha}\text{Tr}[D\uptheta ] + O(|D\uptheta|^2) \text{ using } \mathscr{J}=\det [D \eta ] = \det [\textbf{Id} + D\uptheta ] = 1 + \text{Tr}[D \uptheta ] + O(|D\uptheta|^2),
 \end{equation}
 to bound $(iv)$ by $e^{-\mu_0 \tau} \mathcal{S}^N(\tau)$ where we again apply Lemma \ref{L:EMBEDDING}.
 
To complete the zeroth order estimate, we obtain the full expression for $\mathcal{E}^N$ in (\ref{E:ENERGYMAINGAMMALEQ5OVER3}) by adding the following formula 
 \begin{equation}\label{E:ZEROORDERFULLENGAMMALEQ5OVER3}
 \frac{1}{2}\frac{d}{d\tau} \frac1\alpha \int_\Omega  w^{1+\alpha} e^{\bar{S}} \mathscr{J}^{-\frac1\alpha} |\text{div}_\eta \uptheta |^2 dy =  \frac{1}{2\alpha} \int_\Omega  w^{1+\alpha} e^{\bar{S}} \left( \partial_\tau (\mathscr{J}^{-\frac1\alpha} )|\text{div}_\eta \uptheta |^2
 + \mathscr{J}^{-\frac1\alpha} \partial_\tau(|\text{div}_\eta \uptheta |^2) \right) dy,
\end{equation}
with the right-hand side in turn bounded by $e^{-\mu_0 \tau} \mathcal{S}^N(\tau)$ by Lemma \ref{L:EMBEDDING}. Hence after time integration, the left-hand side of (\ref{E:ZEROORDERFULLENGAMMALEQ5OVER3}) completes $\mathcal{E}^N$ in (\ref{E:ENERGYMAINGAMMALEQ5OVER3}), contributes to $\mathcal{S}^N(0)$ by the fundamental theorem of calculus and the right-hand side of (\ref{E:ZEROORDERFULLENGAMMALEQ5OVER3}) contributes to $\int_0^\tau e^{-\mu_0\tau'} \mathcal{S}^N(\tau') d\tau'$.
\\ \\
\textbf{High order estimates.} Fix $(a,\beta)$ with $a+|\beta| \geq 1$. First, rearrange (\ref{E:THETAGAMMALEQ5OVER3}) as     
\begin{equation}\label{E:HIGHORDERREARRANGEGAMMALEQ5OVER3}
\mu^{3 \gamma -4} (\mu \partial_{\tau \tau} \uptheta_i + \mu_\tau \partial_\tau \uptheta_i  + 2 \mu \Gamma^*_{ij} \partial_{\tau} \uptheta_j) + \delta \Lambda_{i \ell} \uptheta_\ell + \frac{1}{w^\alpha}\left(w^{1+\alpha} e^{\bar{S}} \Lambda_{ij} \left(\mathscr{A}_j^k \mathscr{J}^{-\frac{1}{\alpha}}-\delta_j^k\right)\right) ,_k = 0.
\end{equation}
Apply $X_r^a \slashed\partial^{\beta}$ to (\ref{E:HIGHORDERREARRANGEGAMMALEQ5OVER3}) and multiply by $\psi$
\begin{align}
\psi \mu^{3 \gamma -4} (\mu \partial_{\tau \tau} X_r^a \slashed\partial^{\beta} \uptheta_i &+ \mu_\tau \partial_\tau X_r^a \slashed\partial^{\beta} \uptheta_i  + 2 \mu \Gamma^*_{ij} \partial_{\tau} X_r^a \slashed\partial^{\beta} \uptheta_j) + \delta\psi  \Lambda_{i \ell} X_r^a \slashed\partial^{\beta} \uptheta_\ell \notag \\
&+ \psi \frac{1}{w^{\alpha+a}}\left(w^{1+\alpha+a} e^{\bar{S}} \Lambda_{ij} X_r^a \slashed\partial^{\beta} \left(\mathscr{A}_j^k \mathscr{J}^{-\frac{1}{\alpha}}-\delta_j^k\right)\right) ,_k = -\psi \mathcal{R}_i^{a,\beta},  \label{E:HIGHORDERDERIVATIVEPSIGAMMALEQO5OVER3}
\end{align}
where we have used Lemma \ref{L:COMM} to compute $X^a_r \slashed\partial^\beta \left[\frac{1}{w^\alpha}\left(w^{1+\alpha} e^{\bar{S}} \Lambda_{ij} \left(\mathscr{A}_j^k \mathscr{J}^{-\frac{1}{\alpha}}-\delta_j^k\right)\right) ,_k\right]$ with $\mathcal{R}_i^{a,\beta}$ defined as the lower order terms arising from multiple applications of (\ref{E:COMMXRC}) and (\ref{E:COMMSLASHEDPARTIAL}) in Lemma \ref{L:COMM}, and applying our derivative operators to the resultant expressions from (\ref{E:COMMXRC}) and (\ref{E:COMMSLASHEDPARTIAL}).

More specifically, using the favorable expressions $\mathcal{C}^{q+1}_i$ and $\mathcal{C}^{q}_{i j \ell}$, (\ref{E:COMMXRC}) and (\ref{E:COMMSLASHEDPARTIALC}) respectively, derived in Lemma \ref{L:COMM}, the regularity of our entropy term given by Corollary \ref{C:ENTROPYREGULARITYCOROLLARY}, the differentiation formulae for $\mathscr{A}$ and $\mathscr{J}$ (\ref{E:AJDIFFERENTIATIONFORMULAE}), and also the decomposition of the spatial derivative into $X_r$ and $\slashed\partial$ (\ref{E:DERIVATIVEDECOMP}), we have the following favorable schematic form for $\mathcal{R}_i^{a,\beta}$ 
\begin{align}
\mathcal{R}_i^{a,\beta} &=  \sum_{\ell=0}^{a+|\beta|} \left( \left[ \sum_{a'+|\beta'|=\ell \atop a'\leq a, |\beta'| \leq |\beta|+1} c^1_{a',\beta',\ell} X_r^{a'}\slashed\partial^{\beta'} \uptheta + c^2_{a',\beta',\ell} w \, X_r^{a'}\slashed\partial^{\beta'} D \uptheta  \right] \right.  \notag \\  
&\times  \left. \prod_{a''+|\beta''| \leq N-\max(\ell,N-\ell)} c_{a'',\beta'',\ell} X_r^{a'' }\slashed\partial^{\beta''} \uptheta  +w \left[\prod_{a'''+|\beta'''| \leq N-\max(\ell,N-\ell)} c_{a''',\beta''',\ell} X_r^{a'''}\slashed\partial^{\beta'''} D \uptheta \right] \right), \label{E:FAVORABLEREMAINDERFORMGAMMALEQ5OVER3}  
\end{align}
for all $i\in \{1,2,3\}$, where $c^1_{a',\beta',\ell}$, $c^2_{a',\beta',\ell}$, $c_{a'',\beta'',\ell}$ and $c_{a''',\beta''',\ell}$ are bounded coefficients on $B_1(\mathbf{0})$ away from the origin. Note here for simplicity we are including the terms $\mathscr{A}$ and $\mathscr{J}$ without derivatives in our bounded coefficients since they are bounded by our a priori assumptions (\ref{E:APRIORI}). Since $\Lambda$ is also bounded by (\ref{E:LAMBDABOUNDSGAMMALEQ5OVER3}), we in addition include $\Lambda$ in our smooth coefficients.

The important feature of the above expression is that any time we have a potentially top order $D \uptheta$ term, we have a weight $w$ multiplying such terms: this is seen in the second summation in the expression. This is a direct result of the commutator forms (\ref{E:COMMXRC})-(\ref{E:COMMSLASHEDPARTIALC}) we carefully obtained in Lemma \ref{L:COMM} to have this property. Then with the expression (\ref{E:FAVORABLEREMAINDERFORMGAMMALEQ5OVER3}), Lemma \ref{L:EMBEDDING} and the desirable form of $[X_r^a\slashed\partial^\beta,\partial_s]$ (\ref{E:DIFFCOMMUTATOR}) derived in the curl estimates,
\begin{equation}\label{E:HIGHORDERRBOUNDGAMMALEQ5OVER3}
\|\mathcal{R}^{a,\beta}_i\|^2_{a+\alpha,\psi e^{\bar{S}}} \lesssim (1+P(\mathcal{S}^N))\mathcal{S}^N,
\end{equation}       
with $P$ a polynomial of degree at least 1.

Multiply (\ref{E:HIGHORDERDERIVATIVEPSIGAMMALEQO5OVER3}) by $w^{\alpha+a}\Lambda^{-1}_{im}\partial_\tau X^a_r \slashed\partial^\beta \uptheta^m$ and integrate over $\Omega$
\begin{align}
&\frac12\frac{d}{d\tau}\left(\mu^{3\gamma-3} \int_\Omega \psi \, w^{\alpha+a} \langle\Lambda^{-1}X^a_r \slashed\partial^\beta{\bf V},\,X^a_r \slashed\partial^\beta{\bf V}\rangle \,dy+\delta \int_\Omega\psi  w^{\alpha+a}\, |X^a_r \slashed\partial^\beta\uptheta |^2 
\,dy \right)  \notag \\
&+\frac{5-3\gamma}{2}\mu^{3\gamma-3} \frac{\mu_\tau}{\mu}\int_\Omega \psi  w^{\alpha+a}\,\langle \Lambda^{-1}X^a_r \slashed\partial^\beta{\bf V},\,X^a_r \slashed\partial^\beta{\bf V}\rangle dy \notag \\
& - \frac12\mu^{3\gamma-3}\int_\Omega
\psi  w^{\alpha+a}\, \left\langle \left[\partial_\tau\Lambda^{-1}-4 \Lambda^{-1}\Gamma^*\right]X^a_r \slashed\partial^\beta{\bf V},X^a_r \slashed\partial^\beta{\bf V}\right\rangle dy\notag\\
& + \int_\Omega  \psi \left( w^{1+\alpha+a} e^{\bar{S}} \Lambda_{ij} X^a_r \slashed\partial^\beta \left(\mathscr{A}^k_j\mathscr{J}^{-\frac1\alpha}-\delta^k_j\right)\right),_k \Lambda^{-1}_{im} \partial_\tau X^a_r \slashed\partial^\beta \uptheta^m \,dy  \notag \\
&=-\int_\Omega \psi  w^{\alpha+a}\mathcal R^i_{a,\beta} \Lambda^{-1}_{im} X^a_r \slashed\partial^\beta \uptheta^m_\tau \,dy. \label{E:HIGHORDERPOSTIPGAMMALEQ5OVER3}
\end{align} 
By the same reasoning as the zeroth order case, the first four integrals on the left hand side of (\ref{E:HIGHORDERPOSTIPGAMMALEQ5OVER3}) contribute to the energy inequality (\ref{E:ENERGYMAINGAMMALEQ5OVER3}) in the same way. Also, after time integration, the right hand side of (\ref{E:HIGHORDERPOSTIPGAMMALEQ5OVER3}) in addition contributes to $\int_0^\tau e^{-\mu_0\tau'} \mathcal{S}^N(\tau') d\tau'$ using (\ref{E:HIGHORDERRBOUNDGAMMALEQ5OVER3}) and (\ref{E:APRIORI}).

Now compute the last integral on the left hand side of (\ref{E:HIGHORDERPOSTIPGAMMALEQ5OVER3}) using the derivative formula for $\mathscr{A}_j^k \mathscr{J}^{-\frac{1}{\alpha}}$ (\ref{E:HIGHORDERDERVIATIVEAJGAMMALEQ5OVER3}) and integration by parts
\begin{align}
& \int_\Omega  \psi \left( w^{1+\alpha+a} e^{\bar{S}} \Lambda_{ij} X^a_r \slashed\partial^\beta \left(\mathscr{A}^k_j\mathscr{J}^{-\frac1\alpha}-\delta^k_j\right)\right),_k \Lambda^{-1}_{im} \partial_\tau X^a_r \slashed\partial^\beta \uptheta^m \,dy \notag \\
& = \int_\Omega  \psi \left( w^{1+\alpha+a} e^{\bar{S}} \Lambda_{ij} \left(- \mathscr{J}^{-\frac1\alpha}\mathscr{A}^k_\ell\mathscr{A}^s_j X_r^a \slashed\partial^\beta \uptheta^\ell,_s -\tfrac1\alpha \mathscr{J}^{-\frac1\alpha}\mathscr{A}^k_j\mathscr{A}^s_\ell X_r^a \slashed\partial^\beta \uptheta^\ell,_s  \right. \right. \notag \\ 
&\left. \left. \qquad \qquad \qquad + \mathcal C^{a,\beta,k}_j(\uptheta)\right) \right),_k \Lambda^{-1}_{im} \partial_\tau X^a_r \slashed\partial^\beta \uptheta^m \,dy \notag \\
&= -\int_\Omega  \psi \left( w^{1+\alpha+a} e^{\bar{S}} \left( \mathscr{J}^{-\frac1\alpha}\mathscr{A}^k_\ell \left(\Lambda_{ij}\mathscr{A}^s_jX^a_r\slashed\partial^\beta\uptheta^\ell,_s - \Lambda_{\ell j}\mathscr{A}^s_jX^a_r\slashed\partial^\beta\uptheta^i,_s\right)  \right. \right. \notag \\
& \left. \left. + \mathscr{J}^{-\frac1\alpha}\mathscr{A}^k_\ell\Lambda_{\ell j}\mathscr{A}^s_jX^a_r\slashed\partial^\beta\uptheta^i,_s + \tfrac1\alpha\mathscr{J}^{-\frac1\alpha}\Lambda_{ij}\mathscr{A}^k_j\mathscr{A}^s_\ell  X^a_r \slashed\partial^\beta \uptheta^\ell,_s \right) \right),_k \Lambda^{-1}_{im} \partial_\tau X^a_r \slashed\partial^\beta \uptheta^m \,dy \notag\\
& \qquad \qquad \qquad+\int_\Omega \psi \Lambda_{ij} \left(  w^{1+\alpha+a} e^{\bar{S}} \mathcal C^{a,\beta,k}_j \right),_k \Lambda^{-1}_{im}\partial_\tau X^a_r \slashed\partial^\beta \uptheta^m \,dy \notag \\
&=\int_\Omega  \psi w^{1+\alpha+a} e^{\bar{S}} \mathscr{J}^{-\frac1\alpha} \left(\mathscr{A}^k_\ell  [\text{Curl}_{\Lambda\mathscr{A}}X^a_r \slashed\partial^\beta\uptheta]^\ell_i + \mathscr{A}^k_\ell \Lambda_{\ell j}[\nabla_\eta X^a_r \slashed\partial^\beta\uptheta]^i_j + \tfrac1\alpha \Lambda_{ij}\mathscr{A}^k_j \text{div}_\eta (X^a_r \slashed\partial^\beta \uptheta)\right) \notag \\
&\quad \quad \Lambda^{-1}_{im}\partial_\tau \left(X^a_r \slashed\partial^\beta \uptheta^m\right),_k \,dy \notag \\
&+\int_\Omega \psi,_k w^{1+\alpha+a} e^{\bar{S}} \mathscr{J}^{-\frac1\alpha} \left(\mathscr{A}^k_\ell  [\text{Curl}_{\Lambda\mathscr{A}}X^a_r \slashed\partial^\beta\uptheta]^\ell_i + \mathscr{A}^k_\ell \Lambda_{\ell j}[\nabla_\eta X^a_r \slashed\partial^\beta\uptheta]^i_j \notag \right.\\
&\left.+ \tfrac1\alpha \Lambda_{ij}\mathscr{A}^k_j \text{div}_\eta (X^a_r \slashed\partial^\beta \uptheta)\right) \Lambda^{-1}_{im}\partial_\tau X^a_r \slashed\partial^\beta \uptheta^m \, dy +\int_\Omega \psi \Lambda_{ij} \left(  w^{1+\alpha+a} e^{\bar{S}} \mathcal C^{a,\beta,k}_j \right),_k \Lambda^{-1}_{im}\partial_\tau X^a_r \slashed\partial^\beta \uptheta^m \,dy \notag \\
&=:I_1+I_2+I_3,
\end{align}
where we define $\mathcal C^{a,\beta,k}_j(\uptheta)$ to be the lower order terms arising from successive iterations of (\ref{E:HIGHORDERDERVIATIVEAJGAMMALEQ5OVER3}) in Lemma \ref{L:USEFULIDENTITIESPREENERGYINEQUALITY}, and differentiating the result of (\ref{E:HIGHORDERDERVIATIVEAJGAMMALEQ5OVER3}). In particular, using the desirable form of $[X_r^a\slashed\partial^\beta,\partial_s]$ (\ref{E:DIFFCOMMUTATOR}) obtained in the curl estimates, the Leibniz rule, and also the decomposition of the spatial derivative into $X_r$ and $\slashed\partial$ (\ref{E:DERIVATIVEDECOMP}), we have the following desirable schematic form for $\mathcal C^{a,\beta,k}_j(\uptheta)$    
\begin{equation}\label{E:LOWERORDERCGAMMALEQ5OVER3}
\mathcal{C}^{a,\beta,k}_j(\uptheta) =  \sum_{\ell=\tfrac N2}^{a+|\beta|} \left( \left[ \sum_{a'+|\beta'|=\ell \atop a'\leq a+1, |\beta'| \leq |\beta|+1} \bar{c}_{a',\beta',\ell}X_r^{a'}\slashed\partial^{\beta'} \uptheta \right]  \prod_{a''+|\beta''| \leq N-\ell } \bar{c}_{a'',\beta'',\ell} X_r^{a''}\slashed\partial^{\beta''} \uptheta \right),
\end{equation}  
for every $j,k \in \{1,2,3\}$, where $\bar{c}_{a',\beta',\ell}$ and $\bar{c}_{a'',\beta'',\ell}$ are bounded coefficients on $B_1(\mathbf{0})$ away from the origin. Note here for simplicity we are including the terms $\mathscr{A}$ and $\mathscr{J}$ without derivatives in our bounded coefficients since they are bounded by our a priori assumptions (\ref{E:APRIORI}). With this formula (\ref{E:LOWERORDERCGAMMALEQ5OVER3}), we can see using Lemma \ref{L:EMBEDDING} to handle the lower order terms in the product $\prod_{a''+|\beta''| \leq \tfrac N2}$ that
\begin{equation}\label{E:HIGHORDERDERVIATIVECOMMUTATORBOUND1GAMMALEQ5OVER3}
\|\mathcal C^{a,\beta,k}_j(\uptheta) \|_{\alpha+a,\psi e^{\bar{S}}}^2 \lesssim \mathcal S^N(\tau).
\end{equation}
Furthermore, again using the desirable form of $[X_r^a\slashed\partial^\beta,\partial_s]$ (\ref{E:DIFFCOMMUTATOR}) and the decomposition of the spatial derivative into $X_r$ and $\slashed\partial$ (\ref{E:DERIVATIVEDECOMP}), we can obtain
\begin{equation}\label{E:HIGHORDERDERVIATIVECOMMUTATORBOUND2GAMMALEQ5OVER3}
\|\mathcal D \mathcal{C}^{a,\beta,k}_j(\uptheta) \|_{1+\alpha+a,\psi e^{\bar{S}}}^2 \lesssim \mathcal S^N(\tau).
\end{equation}
\underline{Estimation of $I_3$:} 
\begin{align}
&\Big|\int_0^\tau I_3 \, d\tau' \Big| = \Big|\int_0^\tau \int_\Omega \psi \left(  w^{1+\alpha+a} e^{\bar{S}} \mathcal C^{a,\beta,k}_j \right),_k \Lambda^{-1}_{im}\partial_\tau X^a_r \slashed\partial^\beta \uptheta^m \,dyd\tau'\Big| \notag \\
& \lesssim \int_0^\tau \|X^a_r \slashed\partial^\beta{\bf V}^m\|_{\alpha+a,\psi e^{\bar{S}}} \|\mathcal C^{a,\beta,k}_j(\uptheta) \|_{\alpha+a,\psi e^{\bar{S}}} + \|X^a_r \slashed\partial^\beta{\bf V}^m\|_{1+\alpha+a,\psi e^{\bar{S}}} \|D\mathcal C^{a,\beta,k}_j \|_{1+\alpha+a,\psi e^{\bar{S}}}\,d\tau' \notag \\
& \lesssim \int_0^\tau e^{-\mu_0\tau'} \mathcal{S}^N(\tau') d\tau',
\end{align}
where we have used (\ref{E:LAMBDABOUNDSGAMMALEQ5OVER3}), (\ref{E:WDEMAND}), Lemma \ref{L:ENTROPYREGULARITY} and (\ref{E:HIGHORDERDERVIATIVECOMMUTATORBOUND1GAMMALEQ5OVER3})-(\ref{E:HIGHORDERDERVIATIVECOMMUTATORBOUND2GAMMALEQ5OVER3}). \\ \\
\underline{Estimation of $I_2$:} 
\begin{align}
\Big|\int_0^\tau I_2 \, d\tau' \Big| &= \Big| \int_0^\tau \int_{B_{\frac34}({\bf 0})\setminus B_{\frac14}({\bf 0})} \psi,_k w^{1+\alpha+a} e^{\bar{S}} \mathscr{J}^{-\frac1\alpha} \left(\mathscr{A}^k_\ell  [\text{Curl}_{\Lambda\mathscr{A}}X^a_r \slashed\partial^\beta\uptheta]^\ell_i \right. \notag \\ 
&   \left. + \mathscr{A}^k_\ell \Lambda_{\ell j}[\nabla_\eta X^a_r \slashed\partial^\beta\uptheta]^i_j+ \tfrac1\alpha \Lambda_{ij}\mathscr{A}^k_j  \text{div}_\eta (X^a_r \slashed\partial^\beta \uptheta) \right) \Lambda^{-1}_{im}\partial_\tau X^a_r \slashed\partial^\beta \uptheta^m \,dyd\tau' \Big|  \notag \\
&\lesssim \int_0^\tau\|D\psi\|_{L^\infty(\Omega)} \Big(   \|X_r^a\slashed\partial^\beta{\bf V}^m\|_{\alpha+a,\psi e^{\bar{S}}} \|\nabla_\eta X_r^a \slashed\partial^\beta \uptheta\|_{1+\alpha+a,\psi e^{\bar{S}} } \notag \\ 
&\qquad \qquad \qquad  + \sum_{|\nu|\le a+|\beta|}\|\partial^\nu{\bf V}\|_{\alpha,(1-\psi)e^{\bar{S}}} \sum_{|\nu|\le a+|\beta|}\|\nabla_\eta\partial^\nu\uptheta\|_{1+\alpha,(1-\psi)e^{\bar{S}}}\Big)\,d\tau' \notag \\
& \lesssim \int_0^\tau e^{-\mu_0\tau'} \mathcal{S}^N(\tau') d\tau',
\end{align}
where we have used that $\psi,_k$ has support in $B_{\frac34}({\bf 0})\setminus B_{\frac14}({\bf 0})$ and the fact that we can write $X_r^a \slashed\partial^\beta$ in terms of $\partial^\nu$ in $B_{\frac34}({\bf 0})\setminus B_{\frac14}({\bf 0})$. \\ \\
\underline{Estimation of $I_1$:} 
\begin{align}
&\int_0^\tau I_1 \, d\tau' =  \int_0^\tau \int_\Omega  \psi w^{1+\alpha+a} e^{\bar{S}} \mathscr{J}^{-\frac1\alpha} \mathscr{A}^k_\ell [\text{Curl}_{\Lambda\mathscr{A}}X^a_r \slashed\partial^\beta\uptheta]^\ell_i \Lambda^{-1}_{im}\partial_\tau \left(X^a_r \slashed\partial^\beta \uptheta^m\right),_k \,dy d\tau'  \notag \\
&+ \int_0^\tau  \int_\Omega  \psi w^{1+\alpha+a} e^{\bar{S}} \mathscr{J}^{-\frac1\alpha} \left( \mathscr{A}^k_\ell \Lambda_{\ell j}[\nabla_\eta X^a_r \slashed\partial^\beta\uptheta]^i_j + \tfrac1\alpha \Lambda_{ij}\mathscr{A}^k_j \text{div}_\eta (X^a_r \slashed\partial^\beta \uptheta) \right) \Lambda_{im}^{-1}\partial_\tau \left(X^a_r \slashed\partial^\beta \uptheta^m\right),_k \,dy d\tau'  \notag \\
&= \int_0^\tau \int_\Omega  \psi w^{1+\alpha+a} e^{\bar{S}} \mathscr{J}^{-\frac1\alpha} \mathscr{A}^k_\ell \Lambda^{-1}_{im} [\text{Curl}_{\Lambda\mathscr{A}}X^a_r \slashed\partial^\beta\uptheta]^\ell_i \partial_\tau \left(X^a_r \slashed\partial^\beta \uptheta^m\right),_k \,dy d\tau'  \notag \\
&+ \int_0^\tau  \int_\Omega  \psi w^{1+\alpha+a} e^{\bar{S}} \mathscr{J}^{-\frac1\alpha} \left( \Lambda_{\ell j}[\nabla_\eta X^a_r \slashed\partial^\beta \uptheta]^i_j \Lambda^{-1}_{im}[\nabla_\eta\partial_\tau X^a_r \slashed\partial^\beta \uptheta ]_\ell^m+\tfrac1\alpha \text{div}_\eta (X^a_r \slashed\partial^\beta\uptheta) \text{div}_\eta (\partial_\tau X^a_r \slashed\partial^\beta \uptheta) \right) \,dy d\tau' \notag.
\end{align}
Integrating by parts in $\tau$ the curl term
\begin{align}
&\int_0^\tau \int_\Omega  \psi w^{1+\alpha+a} e^{\bar{S}} \mathscr{J}^{-\frac1\alpha} \mathscr{A}^k_\ell \Lambda^{-1}_{im} [\text{Curl}_{\Lambda\mathscr{A}}X^a_r \slashed\partial^\beta\uptheta]^\ell_i \partial_\tau \left(X^a_r \slashed\partial^\beta \uptheta^m\right),_k \,dy d\tau' \notag \\
& =  \int_\Omega \psi w^{1+\alpha+a} e^{\bar{S}} \mathscr{J}^{\frac1\alpha} \mathscr{A}^k_\ell\Lambda^{-1}_{im} [\text{Curl}_{\Lambda\mathscr{A}}X^a_r \slashed\partial^\beta\uptheta]^\ell_i \left(X^a_r \slashed\partial^\beta{\bf\uptheta}^m\right),_k \,dy\Big|^\tau_0 \notag \\
&  -\int_0^\tau \int_\Omega \psi w^{1+\alpha+a} e^{\bar{S}}   \mathscr{J}^{\frac1\alpha} \mathscr{A}^k_\ell\Lambda^{-1}_{im}
[\text{Curl}_{\Lambda\mathscr{A}}X^a_r \slashed\partial^\beta{\bf V}]^\ell_i \left(X^a_r \slashed\partial^\beta{\bf\uptheta}^m\right),_k \,dy\,d\tau' \notag\\
&  -\int_0^\tau \int_\Omega \psi w^{1+\alpha+a} e^{\bar{S}}   \partial_\tau\left(\mathscr{J}^{\frac1\alpha} \mathscr{A}^k_\ell\Lambda^{-1}_{im}\right)
[\text{Curl}_{\Lambda\mathscr{A}}X^a_r \slashed\partial^\beta{\bf \uptheta}]^\ell_i \left(X^a_r \slashed\partial^\beta{\bf\uptheta}^m\right),_k \,dy\,d\tau'  \notag \\
&  -\int_0^\tau \int_\Omega \psi w^{1+\alpha+a} e^{\bar{S}}   \mathscr{J}^{\frac1\alpha} \mathscr{A}^k_\ell\Lambda^{-1}_{im}
[\text{Curl}_{\Lambda_\tau\mathscr{A}}X^a_r \slashed\partial^\beta{\bf \uptheta}]^\ell_i \left(X^a_r \slashed\partial^\beta{\bf\uptheta}^m\right),_k \,dy\,d\tau' \notag \\
&- \int_0^\tau \int_\Omega \psi w^{1+\alpha+a} e^{\bar{S}}   \mathscr{J}^{\frac1\alpha} \mathscr{A}^k_\ell\Lambda^{-1}_{im}
[\text{Curl}_{\Lambda\mathscr{A}_\tau}X^a_r \slashed\partial^\beta{\bf \uptheta}]^\ell_i \left(X^a_r \slashed\partial^\beta{\bf\uptheta}^m\right),_k \,dy\,d\tau' \notag \\
&:=B_1+B_2+\mathcal{R}_1+\mathcal{R}_2+\mathcal{R}_3.
\end{align} 
Now rewriting the gradient and divergence terms
\begin{align}
&\int_0^\tau  \int_\Omega  \psi w^{1+\alpha+a} e^{\bar{S}} \mathscr{J}^{-\frac1\alpha} \left( \Lambda_{\ell j}[\nabla_\eta X^a_r \slashed\partial^\beta \uptheta]^i_j \Lambda^{-1}_{im}[\nabla_\eta\partial_\tau X^a_r \slashed\partial^\beta \uptheta ]_\ell^m+\tfrac1\alpha \text{div}_\eta (X^a_r \slashed\partial^\beta\uptheta) \text{div}_\eta (\partial_\tau X^a_r \slashed\partial^\beta \uptheta) \right) \,dy d\tau' \notag \\
&=\int_0^\tau \int_\Omega \psi  w^{1+\alpha+a} e^{\bar{S}}  \mathscr{J}^{-\frac1\alpha} \left( \Lambda_{\ell j}[\nabla_\eta X^a_r \slashed\partial^\beta \uptheta]^i_j \Lambda^{-1}_{im}\partial_\tau [\nabla_\eta X^a_r \slashed\partial^\beta \uptheta ]_\ell^m+\tfrac1\alpha\text{div}_\eta (X^a_r \slashed\partial^\beta \uptheta) \partial_\tau \text{div}_\eta ( X^a_r \slashed\partial^\beta \uptheta) \right)dy d\tau' \notag \\
&-\int_0^\tau \int_\Omega \mathscr{J}^{-\frac1\alpha} \left( \Lambda_{\ell j}[\nabla_\eta X^a_r \slashed\partial^\beta \uptheta]^i_j \Lambda^{-1}_{im}\partial_\tau\mathscr{A}^k_\ell \left(X^a_r \slashed\partial^\beta \uptheta^m\right),_k+ \tfrac1\alpha\text{div}_\eta (X^a_r \slashed\partial^\beta\uptheta)  \partial_\tau \mathscr{A}^k_j \left(X^a_r \slashed\partial^\beta \uptheta^j\right),_k  \right) dy d\tau' \notag \\
&=E_1+\mathcal{R}_4.
\end{align} 
For $B_1$
\begin{align}
B_1=&\int_\Omega \psi w^{1+\alpha+a} e^{\bar{S}} \mathscr{J}^{\frac1\alpha} \mathscr{A}^k_\ell\Lambda^{-1}_{im} \sum_{a+|\beta| \leq N} [\text{Curl}_{\Lambda\mathscr{A}}X^a_r \slashed\partial^\beta\uptheta]^\ell_i \left(X^a_r \slashed\partial^\beta{\bf\uptheta}^m\right),_k \,dy\Big|^\tau_0 \notag \\
&\lesssim \mathcal{S}^N(0) + \kappa \| \nabla_\eta X_r^a \slashed\partial^\beta \uptheta \|_{1+\alpha+a,\psi e^{\bar{S}}}^2 + \frac{1}{\kappa}\|\text{Curl}_{\Lambda\mathscr{A}}X^a_r \slashed\partial^\beta\uptheta\|^2_{1+\alpha+a,\psi e^{\bar{S}}} \notag \\
&\lesssim \mathcal{S}^N(0) + \kappa \mathcal{S}^N(\tau) + \mathcal{B}^N[\uptheta](\tau),
\end{align}
where we have introduced $\kappa$ through the Young inequality. For $B_2$
\begin{align}
|B_2|&=\Big| \int_0^\tau \int_\Omega \psi w^{1+\alpha+a} e^{\bar{S}}   \mathscr{J}^{\frac1\alpha} \mathscr{A}^k_\ell\Lambda^{-1}_{im}
[\text{Curl}_{\Lambda\mathscr{A}}X^a_r \slashed\partial^\beta{\bf V}]^\ell_i \left(X^a_r \slashed\partial^\beta{\bf\uptheta}^m\right),_k \,dx\,d\tau' \Big| \notag \\
&\lesssim \int_0^\tau  \| \nabla_\eta X_r^a \slashed\partial^\beta \uptheta \|_{1+\alpha+a,\psi e^{\bar{S}}} \|\text{Curl}_{\Lambda\mathscr{A}}X^a_r \slashed\partial^\beta\mathbf{V}\|_{1+\alpha+a,\psi e^{\bar{S}}} d\tau' \notag \\
&\lesssim \int_0^\tau (\mathcal{S}^N(\tau'))^{\frac12} (\mathcal{B}^N[{\bf V}](\tau'))^{\frac12}\,d\tau'.
\end{align}
For $E_1$
\begin{align}
E_1&=\frac12 \int_0^\tau \frac{d}{d\tau}\left\{\int_\Omega \psi  w^{1+\alpha+a} e^{\bar{S}}  \mathscr{J}^{-\frac1\alpha}\left(\sum_{i,j=1}^3 d_id_j^{-1}(\mathscr M_{a,\beta})^j_{i})^2 + \frac1\alpha\left(\text{div}_\eta X^a_r \slashed\partial^\beta \uptheta\right)^2\right)\,dy\right\}\,d\tau' \notag\\
&+ \int_0^\tau \int_\Omega \psi  w^{1+\alpha+a} e^{\bar{S}}   \mathscr{J}^{-\frac1\alpha} \mathcal T_{a,\beta} dy \,d\tau' \notag \\
& - \frac12 \int_0^\tau \int_\Omega \psi  w^{1+\alpha+a} e^{\bar{S}}  \partial_\tau\left(\mathscr{J}^{-\frac1\alpha}\right)\left(\sum_{i,j=1}^3 d_id_j^{-1}(\mathscr M_{a,\beta})^j_{i})^2 + \frac1\alpha\left(\text{div}_\eta X^a_r \slashed\partial^\beta \uptheta\right)^2\right)\,dy \,d\tau' \notag \\
&:=E_2 + \mathcal{R}_5 + \mathcal{R}_6,
\end{align}
where we have used Lemma \ref{L:KEYLEMMA}. Then $E_2$ contributes to $\mathcal{E}^N(\tau)$ in (\ref{E:ENERGYMAINGAMMALEQ5OVER3}), and also to $\mathcal{S}^N(0)$ by the fundamental theorem of calculus. Finally, by the Young inequality, Lemma \ref{L:USEFULTAULEMMAGAMMALEQ5OVER3}, a priori bounds (\ref{E:APRIORI}) and Lemma \ref{L:EMBEDDING} 
\begin{equation}
|\mathcal{R}_i| \lesssim \int_0^\tau e^{-\mu_0\tau'} \mathcal{S}^N(\tau') d\tau' \; \text{for} \; i=1,2,3,4,5,6.
\end{equation}
We have thus obtained (\ref{E:ENERGYMAINGAMMALEQ5OVER3}) except for $\partial^\nu$ terms in $\mathcal{E}^N(\tau), \mathcal{D}^N(\tau), \mathcal{S}^N(\tau)$ and $\mathcal{B}^N(\tau)$. To estimate such $\partial^\nu$ terms, we note that by (\ref{E:WDEMAND}), $w>0$ in $B_{\frac{3}{4}}(\mathbf{0})$ where the $1-\psi$ estimates will be obtained. Therefore a similar calculation leads to the $\partial^\nu$ contribution to (\ref{E:ENERGYMAINGAMMALEQ5OVER3}). Hence we have obtained the full energy inequality (\ref{E:ENERGYMAINGAMMALEQ5OVER3}).
\end{proof}

\section{Large Adiabatic Constant}\label{S:GAMMAGREATER5OVER3}  
We now prove the case for $\gamma > \frac53$ by modifying the above analysis. We use the strategy from \cite{shkoller2017global}, applied to our nonisentropic setting. The first step is to eliminate the anti-damping effect encountered in \cite{1610.01666}. We divide (\ref{E:THETAGAMMALEQ5OVER3}) by $\mu^{3\gamma-5}$ to obtain
\begin{equation}\label{E:THETAGAMMAGREATER5OVER3}
w^\alpha \mu^{2} \left(\partial_{\tau \tau} \uptheta_i + \frac{\mu_\tau}{\mu} \partial_\tau \uptheta_i  + 2 \Gamma^*_{ij} \partial_{\tau} \uptheta_j\right) + \delta \mu^{5-3\gamma}w^\alpha \Lambda_{i \ell} \uptheta_\ell +  \mu^{5-3\gamma} \left(w^{1+\alpha} e^{\bar{S}} \Lambda_{ij} \left( \mathscr{A}_j^k \mathscr{J}_\eta^{-\frac{1}{\alpha}}-\delta_j^k\right)\right) ,_k = 0.
\end{equation}
Then with $\gamma > \frac53$, $\mathcal{S}^N$ and $\mathcal{B}^N$ are defining according to the expressions given in Section \ref{S:HOQGAMMALEQ5OVER3}. Also as defined in Section \ref{S:MAINTHEOREM}, in this case we do not distinguish between $\mu_0$ and $\mu_1$
\begin{equation}\label{E:NODISTINGUISHMU0MU1GAMMAGREATER5OVER3}
\mu_0=\mu_1.
\end{equation}
Next, the energy based high order quantities which arise directly from the problem are modified as follows since we are considering (\ref{E:THETAGAMMAGREATER5OVER3}).      
\begin{align} 
& \tilde{\mathcal{E}}^N(\uptheta,\mathbf{V})(\tau)=\tilde{\mathcal{E}}^N(\tau) \notag \\ 
&=\frac12\sum_{a+|\beta|\le N} 
\int_\Omega\psi\Big[\mu^{2} \left\langle\Lambda^{-1}X_r^a\slashed\partial^\beta\, \mathbf{V},\,X_r^a\slashed\partial^\beta \mathbf{V}\right\rangle+\delta\mu^{5-3 \gamma}\left\langle\Lambda^{-1}X_r^a\slashed\partial^\beta \uptheta,\,X_r^a\slashed\partial^\beta\uptheta\right\rangle\Big] w^{a+\alpha} \notag \\
& \ \ \ \  \ \ \ \ \  \ \ \ \ \  \ \ \ \ \  \ \ \ \ \  +\mu^{5-3 \gamma}\psi\mathscr{J}^{-\frac1\alpha}\Big[\sum_{i,j=1}^3 d_id_j^{-1}\left((\mathscr{M}_{a,\beta})^j_{i}\right)^2+\tfrac1\alpha\left(\text{div}_\eta X_r^a\slashed\partial^\beta \uptheta\right)^2\Big] w^{a+\alpha+1}e^{\bar{S}} \,dy \notag \\
&\ \ \ \  +\frac12\sum_{|\nu|\le N}
\int_\Omega (1-\psi)\Big[\mu^{2} \left\langle \Lambda^{-1} \partial^\nu\mathbf{V},\,\partial^\nu \mathbf{V}\right\rangle +\delta \mu^{5-3 \gamma} \left\langle\Lambda^{-1}\partial^\nu\mathbf{\uptheta},\,\partial^\nu\mathbf{\uptheta}\right\rangle\Big]  w^\alpha   
\notag \\  
& \ \ \ \  \ \ \ \ \ \ \ \ \ \ \ \ \ \ \ \ \ \ \ \ +\mu^{5-3 \gamma}(1-\psi)\mathscr{J}^{-\frac1\alpha}\Big[ \sum_{i,j=1}^3 d_id_j^{-1}\left((\mathscr{N}_{\nu})^j_{i}\right)^2 + \tfrac1\alpha\left(\text{div}_\eta\partial^\nu \uptheta\right)^2\Big] w^{1+\alpha}e^{\bar{S}} \,dy, \label{E:EDEFGAMMAGREATER5OVER3}
\end{align}         
\begin{align}
&\tilde{\mathcal D}^N (\mathbf{\uptheta})=\frac{3\gamma-5}{2}\mu^{5-3\gamma}\frac{\mu_\tau}{\mu} \int_\Omega  \Big(\psi \sum_{a+|\beta|\le N}  |X_r^a\slashed\partial^\beta \uptheta|^2  w^{a+\alpha}  \notag \\
&+ \psi\mathscr{J}^{-\frac1\alpha}\Big[\sum_{i,j=1}^3 d_id_j^{-1}\left((\mathscr{M}_{a,\beta})^j_{i}\right)^2+\tfrac1\alpha\left(\text{div}_\eta X_r^a\slashed\partial^\beta \uptheta\right)^2\Big] w^{a+\alpha+1}e^{\bar{S}} \notag \\ 
&  + \sum_{|\nu|\le N}(1-\psi) |\partial^\nu \uptheta|^2 w^{\alpha}+(1-\psi) \mathscr{J}^{-\frac1\alpha}\Big[ \sum_{i,j=1}^3 d_id_j^{-1}\left((\mathscr{N}_{\nu})^j_{i}\right)^2 + \tfrac1\alpha\left(\text{div}_\eta\partial^\nu \uptheta\right)^2\Big] w^{1+\alpha}e^{\bar{S}} \Big)\,dy.
\end{align}
We are requiring $\gamma >\frac{5}{3}$ in this formulation, which gives $\tilde{\mathcal{D}}^N (\mathbf{\uptheta}) \geq 0$. Unique to the $\gamma > \frac{5}{3}$ case, we introduce a similar term to $\tilde{\mathcal{E}}^N$ which does not include top order quantities but will be controlled through our coercivity Lemma \ref{L:COERCIVITY} below
\begin{align}   
& \mathcal{C}^{N-1}(\uptheta,\mathbf{V})(\tau)=\mathcal{C}^{N-1}(\tau) \notag \\ 
&=\frac12\sum_{a+|\beta|\le N-1}              
\int_\Omega\psi \, \delta\left\langle\Lambda^{-1}X_r^a\slashed\partial^\beta \uptheta,\,X_r^a\slashed\partial^\beta\uptheta\right\rangle w^{a+\alpha} \notag \\ 
&   +\psi\mathscr{J}^{-\frac1\alpha}\Big[\sum_{i,j=1}^3 d_id_j^{-1}\left((\mathscr{M}_{a,\beta})^j_{i}\right)^2+\tfrac1\alpha\left(\text{div}_\eta X_r^a\slashed\partial^\beta \uptheta\right)^2\Big] w^{a+\alpha+1}e^{\bar{S}} \,dy \notag \\  
&\ \ \ \  +\frac12\sum_{|\nu|\le N-1}     
\int_\Omega (1-\psi)\, \delta \left\langle\Lambda^{-1}\partial^\nu\mathbf{\uptheta},\,\partial^\nu\mathbf{\uptheta}\right\rangle w^\alpha      
\notag \\  
&  +(1-\psi)\mathscr{J}^{-\frac1\alpha}\Big[ \sum_{i,j=1}^3 d_id_j^{-1}\left((\mathscr{N}_{\nu})^j_{i}\right)^2 + \tfrac1\alpha\left(\text{div}_\eta\partial^\nu \uptheta\right)^2\Big] w^{1+\alpha}e^{\bar{S}} \,dy. \label{E:CDEFGAMMAGREATER5OVER3}
\end{align}   

Before outlining the proofs of our main curl and energy inequalities for the large $\gamma > \frac53$ case, we give the most important and useful result in this setting which will allow us to overcome the time weights with negative powers which arise from the new equation structure (\ref{E:THETAGAMMAGREATER5OVER3}). 

\begin{lemma}[Coercivity Estimates]\label{L:COERCIVITY}
Suppose $\gamma > \frac{5}{3}$. Let $(\uptheta, {\bf V}):\Omega \rightarrow \mathbb R^3\times \mathbb R^3$ be a unique local solution to (\ref{E:THETAGAMMALEQ5OVER3})-(\ref{E:THETAICGAMMALEQ5OVER3}) on $[0,T]$ with $T>0$ fixed and assume $(\uptheta, {\bf V})$ satisfies the a priori assumptions (\ref{E:APRIORI}). 
Fix $N\geq 2\lceil \alpha \rceil +12$. Let $k \geq N+1$ in (\ref{E:PHIDEMAND}). Fix $(a,\beta)$ with $0 \leq a+|\beta| \leq N-1$ and $\nu$ with $0 \leq |\nu| \leq N-1$. Then for all $\tau \in [0,T]$, we have the following inequalities 
\begin{subequations} 
\begin{align} 
\| X^a_r \slashed\partial^\beta \uptheta\|^2_{a+\alpha,\psi e^{\bar{S}}} &\lesssim \sup_{0 \leq \tau \leq \tau'} 
\{ \mu^2 \| X^a_r \slashed\partial^\beta \mathbf{V}\|^2_{a+\alpha,\psi e^{\bar{S}}} \} + \| X^a_r \slashed\partial^\beta \uptheta (0)\|^2_{a+\alpha,\psi e^{\bar{S}}}, \label{E:COERCIVITY1} \\  
\| \partial^\nu \uptheta\|^2_{\alpha,(1-\psi)} &\lesssim \sup_{0 \leq \tau \leq \tau'} 
\{ \mu^2 \| \partial^\nu \mathbf{V}\|^2_{\alpha,(1-\psi)e^{\bar{S}} } \} + \| \partial^\nu \uptheta (0)\|^2_{\alpha,(1-\psi)e^{\bar{S}}}, \label{E:COERCIVITY2}
\end{align}
\end{subequations}
\begin{subequations}
\begin{align} 
\| \nabla_{\eta} X^a_r \slashed\partial^\beta \uptheta\|^2_{a+\alpha+1,\psi e^{\bar{S}}} &\lesssim \sup_{0 \leq \tau \leq \tau'}  \{ \sum_{a'+|\beta'| = a+|\beta|+1} \mu^2 \| X_r^{a'} \slashed\partial^{\beta'} \mathbf{V}\|^2_{a+\alpha+1,\psi e^{\bar{S}}} \} \notag \\
&\quad + \| \nabla_{\eta} X^a_r \slashed\partial^\beta \uptheta (0)\|^2_{a+\alpha+1,\psi e^{\bar{S}}}, \label{E:COERCIVITY3} \\ 
\| \nabla_{\eta} \partial^\nu \uptheta\|^2_{\alpha+1,\psi e^{\bar{S}}} &\lesssim \sup_{0 \leq \tau \leq \tau'}  \{ \sum_{|\nu'| = |\nu|+1} \mu^2 \| \partial^{\nu'} \mathbf{V}\|^2_{\alpha+1,\psi e^{\bar{S}}} \} + \| \nabla_{\eta} \partial^\nu \uptheta (0)\|^2_{\alpha+1,\psi e^{\bar{S}}}, \label{E:COERCIVITY4}
\end{align}    
\end{subequations}
\begin{subequations}
\begin{align}  
\| \text{\em div}_{\eta} X^a_r \slashed\partial^\beta \uptheta\|^2_{a+\alpha+1,\psi e^{\bar{S}}} &\lesssim \sup_{0 \leq \tau \leq \tau'}  \{ \sum_{a'+|\beta'| = a+|\beta|+1} \mu^2 \| \mathbf{V}\|^2_{a+\alpha+1,\psi e^{\bar{S}}} \} + \| \text{\em div}_{\eta} X^a_r \slashed\partial^\beta \uptheta (0)\|^2_{a+\alpha+1,\psi e^{\bar{S}}}, \label{E:COERCIVITY5} \\ 
\| \text{\em div}_{\eta} \partial^\nu \uptheta\|^2_{\alpha+1,\psi e^{\bar{S}}} &\lesssim \sup_{0 \leq \tau \leq \tau'}  \{ \sum_{|\nu'| = |\nu|+1} \mu^2 \| \partial^{\nu'} \mathbf{V}\|^2_{\alpha+1,\psi e^{\bar{S}}} \} + \| \text{\em div}_{\eta} \partial^\nu \uptheta (0)\|^2_{\alpha+1,\psi e^{\bar{S}}}. \label{E:COERCIVITY6}
\end{align}
\end{subequations}
\end{lemma} 
\begin{proof} 
\textit{Proof of} (\ref{E:COERCIVITY1})-(\ref{E:COERCIVITY2}). By the fundamental theorem of calculus, and the exponential boundedness of $\mu$ (\ref{E:EXPMU1MUINEQGAMMALEQ5OVER3}) and therefore time integrability of negative powers of $\mu$,
\begin{align}\label{E:COERCIVITYPROOF1A}
X^a_r \slashed\partial^\beta \uptheta = \int_0^\tau X^a_r \slashed\partial^\beta  \mathbf{V} d \tau' + X^a_r \slashed\partial^\beta  \uptheta (0) &= \int_0^\tau \mu^{-1} \mu X^a_r \slashed\partial^\beta \mathbf{V} d \tau' + X^a_r \slashed\partial^\beta \uptheta (0) \notag \\
& \lesssim \sup_{0 \leq \tau \leq \tau'} \{ \mu X^a_r \slashed\partial^\beta  \mathbf{V} \} + X^a_r \slashed\partial^\beta \uptheta (0).
\end{align}       
Therefore applying Cauchy's inequality ($ab \lesssim a^2 + b^2,$ $a,b \in \mathbb{R}$) 
\begin{equation}\label{E:COERCIVITYPROOF1B}
\| X^a_r \slashed\partial^\beta  \uptheta\|^2_{a+\alpha,\psi e^{\bar{S}}} \lesssim \sup_{0 \leq \tau \leq \tau'} \{\mu^2 \| X^a_r \slashed\partial^\beta \mathbf{V}\|^2_{a+\alpha,\psi e^{\bar{S}}} \} + \| X^a_r \slashed\partial^\beta \uptheta (0)\|^2_{a+\alpha,\psi e^{\bar{S}}}.    
\end{equation}         
This proves (\ref{E:COERCIVITY1}), and (\ref{E:COERCIVITY2}) is similar. \\ \\
\textit{Proof of} (\ref{E:COERCIVITY2})-(\ref{E:COERCIVITY3}). By a similar coercivity estimate to (\ref{E:COERCIVITYPROOF1A})-(\ref{E:COERCIVITYPROOF1B})
\begin{equation}\label{E:COERCIVITYPROOF3A}
\| \nabla_{\eta} X^a_r \slashed\partial^\beta \uptheta\|^2_{a+\alpha+1,\psi e^{\bar{S}}} \lesssim \sup_{0 \leq \tau \leq \tau'} \{ \mu^2 \| \nabla_{\eta} X^a_r \slashed\partial^\beta \mathbf{V}\|^2_{a+\alpha+1,\psi e^{\bar{S}}} \} + \| \nabla_{\eta} X^a_r \slashed\partial^\beta \uptheta (0)\|^2_{a+\alpha+1,\psi e^{\bar{S}}}. 
\end{equation}
Now using the decomposition of spatial derivatives into angular derivatives and radial derivative (\ref{E:DERIVATIVEDECOMP}), and also our a priori bounds (\ref{E:APRIORI}), we have
\begin{equation}\label{E:COERCIVITYPROOF3B}
\sup_{0 \leq \tau \leq \tau'} \mu^2 \| \nabla_{\eta} X^a_r \slashed\partial^\beta \mathbf{V}\|^2_{a+\alpha+1,\psi e^{\bar{S}}} \lesssim \sup_{0 \leq \tau \leq \tau'} \{  \sum_{a'+|\beta'| = a+|\beta|+1} \mu^2 \| X_r^{a'} \slashed\partial^{\beta'} \mathbf{V}\|^2_{a+\alpha+1,\psi e^{\bar{S}}} \} . 
\end{equation}      
Then (\ref{E:COERCIVITYPROOF3A})-(\ref{E:COERCIVITYPROOF3B}) imply (\ref{E:COERCIVITY3}). The proof of (\ref{E:COERCIVITY4}) is similar but simpler since we do not need the decomposition formula. \\ \\ 
\textit{Proof of} (\ref{E:COERCIVITY5})-(\ref{E:COERCIVITY6}). Finally the proofs of (\ref{E:COERCIVITY5})-(\ref{E:COERCIVITY6}) are similar to the proofs of (\ref{E:COERCIVITY3})-(\ref{E:COERCIVITY4}). 
\end{proof}  
\begin{remark}
Lemma \ref{L:COERCIVITY} allows us to control terms without time weights with negative powers using our norm $\mathcal{S}^N(\tau)$ and initial data $\mathcal{S}^N(0)$. In the $\gamma > \frac53$ case, such terms do not appear immediately from our equation because of the new structure (\ref{E:THETAGAMMAGREATER5OVER3}). Notably, Lemma \ref{L:COERCIVITY} will let us include the new quantity $\mathcal{C}^{N-1}$ in our energy inequality.

However we cannot use Lemma \ref{L:COERCIVITY} to control top order quantities since that would require control of $N+1$ derivatives of $\mathbf{V}$ which we do not have. This is why firstly we have the particular structure of $\mathcal{S}^N$ in this case where we separate top order terms from lower order terms, and secondly why we only include $N-1$ derivatives in $\mathcal{C}^{N-1}$.   
\end{remark}   
 
\subsection{Curl Estimates $\gamma > \frac53$} 
For the curl estimates for $\gamma > \frac53$ we note that we will start from the same equation (\ref{E:POSTMUINTROEQN}) as for the $\gamma \in (1,\frac53]$ case. So when controlling the curl quantities, we will artificially include the time weight with negative power $\mu^{5-3\gamma}$.  Therefore this will not pose an issue in most of the curl estimates since we either simply use the boundedness of $\mu^{5-3\gamma}$ if we do wish to include it in our high order quantities, or at the top order, include it in our high order quantities. The only time we will have to consider it more carefully is when controlling the $\text{Curl}_{\Lambda \mathscr{A}} \uptheta$ terms when we integrate this time weight. We will give the method for this situation in Lemma \ref{L:NEWCURLESTMETHODGAMMAGREATER5OVER3} below.

Otherwise, we note in the $\gamma > \frac53$ case that from the structure of $\mathcal{S}^N$ allowable by the coercivity Lemma \ref{L:COERCIVITY}, the fact that $\mu_0=\mu_1$ here, and our asymptotic behavior of $A$ which still holds by Lemma \ref{L:AASYMPTOTICS} because of our requirement that $\det A \sim 1 + t^3$, the methods of the curl estimates for $\gamma \in (1,\frac53]$ in Section \ref{S:CURLGAMMALEQ5OVER3} will hold for $\gamma > \frac53$ and we will obtain similar bounds on our curl quantities.
\begin{lemma}\label{L:NEWCURLESTMETHODGAMMAGREATER5OVER3}
Suppose $\gamma > \frac53$. Let $(\uptheta, {\bf V}):\Omega \rightarrow \mathbb R^3\times \mathbb R^3$ be a unique local solution to (\ref{E:THETAGAMMALEQ5OVER3})-(\ref{E:THETAICGAMMALEQ5OVER3}) on $[0,T]$ with $T>0$ fixed and assume $(\uptheta, {\bf V})$ satisfies the a priori assumptions (\ref{E:APRIORI}). Fix $N\geq 2\lceil \alpha \rceil +12$. Let $k \geq N+1$ in (\ref{E:PHIDEMAND}). Let $a+|\beta|=N$. Assume the following bound is known
\begin{equation}
 \Big\|\int_0^{\tau'} \mu(\tau'') X_r^a \slashed\partial^\beta[\partial_\tau, \text{\em Curl}_{\Lambda\mathscr{A}}] {\bf V}d\tau''\Big\|_{1+\alpha+a,\psi e^{\bar{S}} }^2\ \lesssim \mathcal{S}^N(\tau').
\end{equation}
Then
\begin{equation}
\mu^{5-3\gamma}\Big\|\int_0^\tau \frac{1}{\mu(\tau')} \int_0^{\tau'} \mu(\tau'') X_r^a \slashed\partial^\beta[\partial_\tau, \text{\em Curl}_{\Lambda\mathscr{A}}] {\bf V}d\tau''\,d\tau' \Big\|_{1+\alpha+a,\psi e^{\bar{S}}}^2 \lesssim \int_0^\tau e^{-\mu_0\tau} \mathcal{S}^N(\tau')\,d\tau'.
\end{equation}
\end{lemma}
\begin{proof}
Applying the boundedness of $\mu^{5-3\gamma}$
\begin{align}
&\mu^{5-3\gamma}\Big\|\int_0^\tau \frac{1}{\mu(\tau')} \int_0^{\tau'} \mu(\tau'') X_r^a \slashed\partial^\beta[\partial_\tau, \text{Curl}_{\Lambda\mathscr{A}}] {\bf V}d\tau''\,d\tau' \Big\|_{1+\alpha+a,\psi e^{\bar{S}}}^2 \notag \\
& \lesssim \mu^{5-3\gamma} \int_\Omega\left[\int_0^\tau \frac{(1+\tau')^2}{\mu(\tau')}\,d\tau'  \int_0^\tau \frac1{(1+\tau')^2\mu(\tau')  }\left(\int_0^{\tau'} \mu(\tau'') X_r^a \slashed\partial^\beta [\partial_\tau, \text{Curl}_{\Lambda\mathscr{A}}] {\bf V}d\tau''\right)^2\,d\tau'\right] 
 w^{1+\alpha+a} e^{\bar{S}} \psi\,dy\notag \\
& \lesssim \mu^{5-3\gamma} \int_0^\tau \frac1{(1+\tau')^2\mu(\tau')} \Big\|\int_0^{\tau'} \mu(\tau'') X_r^a \slashed\partial^\beta[\partial_\tau, \text{Curl}_{\Lambda\mathscr{A}}] {\bf V}d\tau''\Big\|_{1+\alpha+a,\psi e^{\bar{S}} }^2\,d\tau' \notag \\
& \lesssim \sup_{0 \leq \tau ' \leq \tau}[\mu(\tau ')^{5-3\gamma}] \int_0^\tau \frac1{(1+\tau')^2\mu(\tau')} \Big\|\int_0^{\tau'} \mu(\tau'') X_r^a \slashed\partial^\beta[\partial_\tau, \text{Curl}_{\Lambda\mathscr{A}}] {\bf V}d\tau''\Big\|_{1+\alpha+a,\psi e^{\bar{S}} }^2\,d\tau' \notag \\
& = \sup_{0 \leq \tau ' \leq \tau}[\mu(\tau ')^{5-3\gamma}] \int_0^\tau \frac{\mu(\tau')^{5-3\gamma}}{\mu(\tau')^{5-3\gamma}(1+\tau')^2\mu(\tau')} \Big\|\int_0^{\tau'} \mu(\tau'') X_r^a \slashed\partial^\beta[\partial_\tau, \text{Curl}_{\Lambda\mathscr{A}}] {\bf V}d\tau''\Big\|_{1+\alpha+a,\psi e^{\bar{S}} }^2\,d\tau' \notag \\
& \lesssim \sup_{0 \leq \tau ' \leq \tau}[\mu(\tau ')^{5-3\gamma}]\sup_{0 \leq \tau ' \leq \tau}\left[\frac{1}{\mu(\tau')^{5-3\gamma}}\right] \int_0^\tau \frac{\mu(\tau')^{5-3\gamma}}{(1+\tau')^2\mu(\tau')} \Big\|\int_0^{\tau'} \mu(\tau'') X_r^a \slashed\partial^\beta[\partial_\tau, \text{Curl}_{\Lambda\mathscr{A}}] {\bf V}d\tau''\Big\|_{1+\alpha+a,\psi e^{\bar{S}} }^2\,d\tau' \notag \\
& = \int_0^\tau \frac{\mu(\tau')^{5-3\gamma}}{(1+\tau')^2\mu(\tau')} \Big\|\int_0^{\tau'} \mu(\tau'') X_r^a \slashed\partial^\beta[\partial_\tau, \text{Curl}_{\Lambda\mathscr{A}}] {\bf V}d\tau''\Big\|_{1+\alpha+a,\psi e^{\bar{S}} }^2\,d\tau' \notag \\
& \lesssim \int_0^\tau e^{-\mu_0\tau} \mathcal{S}^N(\tau')\,d\tau'.
\end{align}       
\end{proof}
Using the technique of Lemma \ref{L:NEWCURLESTMETHODGAMMAGREATER5OVER3} for similar terms and then otherwise an analogous argument to the $\gamma \in (1,\frac53]$ case as described above, we obtain our curl estimates in this case.
\begin{proposition}\label{P:CURLBOUNDSGAMMAGREATER5OVER3}
Suppose $\gamma > \frac53$. Let $(\uptheta, {\bf V}):\Omega \rightarrow \mathbb R^3\times \mathbb R^3$ be a unique local solution to (\ref{E:THETAGAMMALEQ5OVER3})-(\ref{E:THETAICGAMMALEQ5OVER3}) on $[0,T]$ with $T>0$ fixed and assume $(\uptheta, {\bf V})$ satisfies the a priori assumptions (\ref{E:APRIORI}). Fix $N\geq 2\lceil \alpha \rceil +12$. Let $k \geq N+1$ in (\ref{E:PHIDEMAND}). Then for all $\tau \in [0,T]$, the following inequalities hold for some $0<\kappa\ll 1$  
\begin{align}     
&\mathcal{B}^N[{\bf V}](\tau) \lesssim 
e^{-2\mu_0\tau}\left(\mathcal{S}^N(0)+\tilde{\mathcal{B}}^N(0)\right)+ (1+\tau^2)e^{-2\tilde{\mu}_0\tau}\mathcal{S}^N(\tau), \label{E:CURLVBOUNDGAMMAGREATER5OVER3} \\
& \tilde{\mathcal{B}}^N[\uptheta](\tau) \lesssim \mathcal{S}^N(0)+\mathcal{B}^N(0) + \kappa \mathcal{S}^N(\tau) + \int_0^\tau e^{-\mu_0\tau'}  \mathcal{S}^N(\tau')\,d\tau'. \label{E:CURLTHETABOUNDGAMMAGREATER5OVER3}
\end{align}  
\end{proposition}        

\subsection{Energy Estimates $\gamma > \frac53$} 
Energy estimates for $\gamma > \frac53$ depend on the new equation structure (\ref{E:THETAGAMMAGREATER5OVER3}). This gives rise to the forms of $\tilde{\mathcal{E}}^N$ and $\tilde{\mathcal{D}}^N$. However with the coercivity Lemma \ref{L:COERCIVITY} and the norm-energy equivalence result in this case, term by term estimates will use the same techniques as in the $\gamma \in (1,\frac53]$ case since as in the curl estimates, we have the desirable asymptotic behavior of $A$ which still holds by Lemma \ref{L:AASYMPTOTICS} because of our requirement that $\det A \sim 1 + t^3$. Also we have $\mu_0=\mu_1$.

Furthermore, as in the curl estimates, when estimating terms which include the time weight $\mu^{5-3\gamma}$ we either simply use the boundedness of $\mu^{5-3\gamma}$ if we do wish to include it in our high order quantities, or at the top order, include it in our high order quantities. The only time we will have to consider it more carefully is when we integrate this time weight after integration by parts in $\tau$. We will give the details for the unique term that arises from this in the outline of our proof of the energy inequality below.
  
Therefore we give the norm-energy equivalence, and then outline the proof for the energy inequality in this case, placing emphasis on the new energy identity structure while omitting the details which are similar to the $\gamma \in (1,\frac53]$ case.               
\begin{lemma}\label{L:NORMENERGYGAMMAGREATER5OVER3} 
Let $(\uptheta, {\bf V}):\Omega \rightarrow \mathbb R^3\times \mathbb R^3$ be a unique local solution to (\ref{E:THETAGAMMALEQ5OVER3})-(\ref{E:THETAICGAMMALEQ5OVER3}) on $[0,T]$ with $T>0$ fixed and assume $(\uptheta, {\bf V})$ satisfies the a priori assumptions (\ref{E:APRIORI}).  Fix $N\geq 2\lceil \alpha \rceil +12$. Let $k \geq N+1$ in (\ref{E:PHIDEMAND}). Then there are constants $C_1,C_2>0$ so that  
\begin{align}   
C_1\mathcal{S}^N (\tau) \le \sup_{0\le\tau'\le\tau}\{\mathcal{E}^N(\tau')+\mathcal{C}^{N-1}(\tau')\} \le C_2(\mathcal{S}^N(\tau)+\mathcal{S}^N(0)).
\end{align} 
\end{lemma} 
\begin{proof} The proof follows from an analogous argument to the proof of Lemma \ref{L:NORMENERGY} from the $\gamma \in (1,\frac 53]$ case, in conjunction with Lemma \ref{L:COERCIVITY} to control terms without time weights with negative powers which are included in $\mathcal{C}^{N-1}(\tau')$ by $\mathcal{S}^N(\tau)+\mathcal{S}^N(0)$.
\end{proof}     
We are now ready to give the energy inequality in this case.  
\begin{proposition}\label{P:ENERGYESTIMATEGAMMAGREATER5OVER3}
Suppose $\gamma > 1$. Let $(\uptheta, {\bf V}):\Omega \rightarrow \mathbb R^3\times \mathbb R^3$ be a unique local solution to (\ref{E:THETAGAMMALEQ5OVER3})-(\ref{E:THETAICGAMMALEQ5OVER3}) on $[0,T]$ with $T>0$ fixed and assume $(\uptheta, {\bf V})$ satisfies the a priori assumptions (\ref{E:APRIORI}). Fix $N\geq 2\lceil \alpha \rceil +12$. Let $k \geq N+1$ in (\ref{E:PHIDEMAND}). Then for all $\tau \in [0,T]$, we have the following inequality for some $0<\kappa\ll 1$, 
\begin{align} 
& \tilde{\mathcal E}^N(\tau)+\mathcal{C}^{N-1}(\tau)+\int_0^\tau \tilde{\mathcal D}^N(\tau')\,d\tau' \lesssim  \mathcal{S}^N(0)  +\mathcal{B}^N[\uptheta](\tau) + \int_0^\tau  (\mathcal{S}^N(\tau'))^{\frac12} (\mathcal{B}^N[{\bf V}](\tau'))^{\frac12}\,d\tau'  \notag \\ 
& \qquad \qquad \qquad \qquad \qquad \qquad  + \kappa\mathcal{S}^N(\tau) + \int_0^\tau e^{-\mu_0\tau'} \mathcal{S}^N(\tau') d\tau'. \label{E:ENERGYMAINGAMMAGREATER5OVER3}
\end{align} 
\end{proposition}
\textit{Outline of Proof.}
\textbf{Zeroth order estimate.} Multiplying (\ref{E:THETAGAMMAGREATER5OVER3}) by $\Lambda_{im}^{-1}\partial_\tau\uptheta^m$ and integrating over $\Omega$
\begin{align}\label{E:ZEROORDERPOSTIPGAMMAGREATER5OVER3}
\int_\Omega w^\alpha ( \mu^{2} \partial_{\tau \tau} \uptheta_i &+  \mu_\tau \mu \partial_\tau \uptheta_i  + 2 \mu^2 \Gamma^*_{ij} \partial_{\tau} \uptheta_j + \delta \mu^{5-3\gamma} w^\alpha \Lambda_{i \ell} \uptheta_\ell )\Lambda_{im}^{-1}\partial_\tau\uptheta^m \, dy \notag \\
&+\int_\Omega \mu^{5-3\gamma} \left(w^{1+\alpha} e^{\bar{S}} \Lambda_{ij} \left( \mathscr{A}_j^k \mathscr{J}^{-\frac{1}{\alpha}}-\delta_j^k\right)\right),_k  \Lambda_{im}^{-1} \partial_\tau \uptheta^m \, dy=0.
\end{align}
Recognizing the perfect time derivative structure of the first integral
\begin{align}\label{E:ZEROORDERPOSTIPINT1GAMMAGREATER5OVER3}
\frac12\frac{d}{d\tau}\left(\mu^{2} \int_\Omega  w^\alpha  \langle \Lambda^{-1}\partial_\tau\uptheta, \partial_\tau\uptheta \rangle \, dy  +\delta \mu^{5-3\gamma} \int_\Omega w^\alpha |\uptheta |^2 dy\right)+\frac{3\gamma-5}{2}\mu^{4-3\gamma}\mu_\tau \int_\Omega  w^\alpha |\uptheta |^2 dy \notag \\
-\frac{\mu^{2}}{2}\int_\Omega  w^\alpha \langle \partial_\tau\Lambda^{-1}\partial_\tau\uptheta, \partial_\tau \uptheta \rangle \, dy + 2\mu^{2} \int_\Omega  w^\alpha \langle \Lambda^{-1}\partial_\tau\uptheta, \Gamma^\ast \partial_\tau \uptheta \rangle  \, dy.   
\end{align}
Returning to the second integral in (\ref{E:ZEROORDERPOSTIPGAMMAGREATER5OVER3}), we integrate by parts to obtain
\begin{align}\label{E:ZEROORDERPOSTIBPGAMMAGREATER5OVER3}
&-\mu^{5-3\gamma}\int_\Omega w^{1+\alpha} e^{\bar{S}} \Lambda_{ij} \left( \mathscr{A}_j^k \mathscr{J}^{-\frac{1}{\alpha}}-\delta_j^k\right)  \Lambda_{im}^{-1} \partial_\tau \uptheta^m,_k \, dy \notag \\
&= \mu^{5-3\gamma} \int_\Omega  w^{1+\alpha} e^{\bar{S}} \mathscr{J}^{-\frac1\alpha} \Lambda_{ij} \mathscr{A}^k_l \uptheta^l,_j  \Lambda_{im}^{-1}\partial_\tau \uptheta^m,_k   dy - \mu^{5-3\gamma} \int_\Omega  w^{1+\alpha} e^{\bar{S}} (\mathscr{J}^{-\frac1\alpha} -1) \partial_\tau \uptheta^k,_k  dy \notag \\
&= \mu^{5-3\gamma} \int_\Omega  w^{1+\alpha} e^{\bar{S}} \mathscr{J}^{-\frac1\alpha}  \Lambda_{ip} \mathscr{A}^j_p  \uptheta^\ell,_j  \Lambda_{im}^{-1} \mathscr{A}^k_\ell \partial_\tau \uptheta^m,_k   dy  \notag \\
&  + \mu^{5-3\gamma} \int_\Omega  w^{1+\alpha} e^{\bar{S}} \mathscr{J}^{-\frac1\alpha} \Lambda_{ip} \mathscr{A}^j_l\uptheta^l,_p \mathscr{A}^k_\ell \uptheta^\ell,_j  \Lambda_{im}^{-1}\partial_\tau \uptheta^m,_k dy - \mu^{5-3\gamma} \int_\Omega  w^{1+\alpha} e^{\bar{S}} (\mathscr{J}^{-\frac1\alpha} -1) \partial_\tau \uptheta^k,_k  dy \notag \\
&= \mu^{5-3\gamma} \int_\Omega  w^{1+\alpha}e^{\bar{S}}  \mathscr{J}^{-\frac1\alpha} \Lambda_{\ell j}[\nabla_\eta\uptheta]^i_j  \Lambda_{im}^{-1} [\nabla_\eta \partial_\tau \uptheta ]^m_\ell dy +\mu^{5-3\gamma} \int_\Omega  w^{1+\alpha}e^{\bar{S}}   \mathscr{J}^{-\frac1\alpha} [\text{Curl}_{\Lambda\mathscr{A}} \uptheta]^\ell_i  \Lambda_{im}^{-1} [\nabla_\eta \partial_\tau \uptheta ]^m_\ell dy \notag \\
&+ \mu^{5-3\gamma} \int_\Omega  w^{1+\alpha} e^{\bar{S}} \mathscr{J}^{-\frac1\alpha}  \mathscr{A}^j_l\uptheta^l,_m \mathscr{A}^k_\ell \uptheta^\ell,_j \partial_\tau \uptheta^m,_k   dy- \mu^{5-3\gamma} \int_\Omega  w^{1+\alpha} e^{\bar{S}} (\mathscr{J}^{-\frac1\alpha} -1) \partial_\tau \uptheta^k,_k  dy \notag \\
&=:(i)+(ii)+(iii)+(iv),
\end{align}
By Lemma \ref{L:KEYLEMMA}  
\begin{align}
&(i)=\frac{1}{2}\frac{d}{d\tau} \mu^{5-3\gamma} \int_\Omega  w^{1+\alpha}e^{\bar{S}} \mathscr{J}^{-\frac1\alpha}\sum_{i,j=1}^3 d_i d_j^{-1}(({\mathscr M}_{0,0})^j_i)^2 dy \notag \\
&\;+ \frac{3\gamma-5}{2} \mu^{5-3\gamma}\frac{\mu_\tau}{\mu} \int_\Omega  w^{1+\alpha}e^{\bar{S}} \mathscr{J}^{-\frac1\alpha}\sum_{i,j=1}^3 d_i d_j^{-1}(({\mathscr M}_{0,0})^j_i)^2 dy + \mu^{5-3\gamma} \int_\Omega  w^{1+\alpha} e^{\bar{S}} \mathscr{J}^{-\frac1\alpha} \mathcal T_{0,0} dy \notag \\
&\;+ \frac{\mu^{5-3\gamma}}{2\alpha}\int_\Omega  w^{1+\alpha}e^{\bar{S}} \mathscr{J}^{-\frac1\alpha-1}\mathscr{J}_\tau \sum_{i,j=1}^3 d_i d_j^{-1}(({\mathscr M}_{0,0})^j_i)^2 dy \notag \\ 
&\;+ \mu^{5-3\gamma} \int_\Omega w^{1+\alpha}e^{\bar{S}}\mathscr{J}^{-\frac{1}{\alpha}}\Lambda_{\ell j} [\nabla_\eta \uptheta]_i^j \Lambda_{im}^{-1} [\nabla_\eta \uptheta]_p^m [\nabla_\eta \partial_{\tau} \uptheta]_\ell^p dy. \label{E:ZEROORDERIEXPGAMMAGREATER5OVER3} 
\end{align}  
Now after time integration, the second term in (\ref{E:ZEROORDERIEXPGAMMAGREATER5OVER3}) contributes to $\int_0^\tau \tilde{\mathcal D}^N(\tau')\,d\tau'$ in (\ref{E:ENERGYMAINGAMMAGREATER5OVER3}). \\ \\
\textbf{High order estimate} Fix $(a,\beta)$ with $a+|\beta| \geq 1$. First, rearrange (\ref{E:THETAGAMMAGREATER5OVER3}) as     
\begin{equation}\label{E:HIGHORDERREARRANGEGAMMAGREATER5OVER3}
\mu^{2} \left(\partial_{\tau \tau} \uptheta_i + \frac{\mu_\tau}{\mu} \partial_\tau \uptheta_i  + 2 \Gamma^*_{ij} \partial_{\tau} \uptheta_j\right) + \delta \mu^{5-3\gamma} \Lambda_{i \ell} \uptheta_\ell +  \mu^{5-3\gamma} \frac{1}{w^\alpha}\left(w^{1+\alpha} e^{\bar{S}} \Lambda_{ij} \left( \mathscr{A}_j^k \mathscr{J}_\eta^{-\frac{1}{\alpha}}-\delta_j^k\right)\right),_k = 0.
\end{equation}        
Apply $X_r^a \slashed\partial^{\beta}$ to (\ref{E:HIGHORDERREARRANGEGAMMAGREATER5OVER3}) and multiply by $\psi$
\begin{align}
&\psi \mu^{2} \left( \partial_{\tau \tau} X_r^a \slashed\partial^{\beta} \uptheta_i + \frac{\mu_\tau}{\mu} \partial_\tau X_r^a \slashed\partial^{\beta} \uptheta_i  + 2 \Gamma^*_{ij} \partial_{\tau} X_r^a \slashed\partial^{\beta} \uptheta_j \right) + \delta\psi \mu^{5-3\gamma} \Lambda_{i \ell} X_r^a \slashed\partial^{\beta} \uptheta_\ell \notag \\ 
&\qquad + \psi \mu^{5-3\gamma} \frac{1}{w^{\alpha+a}}\left(w^{1+\alpha+a} e^{\bar{S}} \Lambda_{ij} X_r^a \slashed\partial^{\beta} \left(\mathscr{A}_j^k \mathscr{J}^{-\frac{1}{\alpha}}-\delta_j^k\right)\right) ,_k = -\psi \mu^{5-3\gamma} \mathcal{R}_i^{a,\beta}. \label{E:HIGHORDERDERIVATIVEPSIGAMMAGREATER5OVER3}
\end{align}
Multiply (\ref{E:HIGHORDERDERIVATIVEPSIGAMMAGREATER5OVER3}) by $w^{\alpha+a}\Lambda^{-1}_{im}\partial_\tau X^a_r \slashed\partial^\beta \uptheta^m$ and integrate over $\Omega$
\begin{align}
&\frac12\frac{d}{d\tau}\left(\mu^2 \int_\Omega \psi \, w^{\alpha+a} \langle\Lambda^{-1}X^a_r \slashed\partial^\beta{\bf V},\,X^a_r \slashed\partial^\beta{\bf V}\rangle \,dy+\delta \mu^{5-3\gamma} \int_\Omega\psi  w^{\alpha+a}\, |X^a_r \slashed\partial^\beta\uptheta |^2 
\,dy \right)  \notag \\
&+\frac{3\gamma-5}{2}\mu^{5-3\gamma} \frac{\mu_\tau}{\mu}\int_\Omega \psi  w^{\alpha+a} |X^a_r \slashed\partial^\beta\uptheta |^2 dy \notag \\
& - \frac{\mu^{2}}{2}\int_\Omega
\psi  w^{\alpha+a}\, \left\langle\left[\partial_\tau\Lambda^{-1}-4 \Lambda^{-1}\Gamma^*\right] X^a_r \slashed\partial^\beta{\bf V},X^a_r \slashed\partial^\beta{\bf V}\right\rangle dy \notag \\
& + \mu^{5-3\gamma} \int_\Omega  \psi \left( w^{1+\alpha+a} e^{\bar{S}} \Lambda_{ij} X^a_r \slashed\partial^\beta \left(\mathscr{A}^k_j\mathscr{J}^{-\frac1\alpha}-\delta^k_j\right)\right),_k \Lambda^{-1}_{im} \partial_\tau X^a_r \slashed\partial^\beta \uptheta^m \,dy  \notag \\
&=- \mu^{5-3\gamma} \int_\Omega \psi  w^{\alpha+a}\mathcal R^i_{a,\beta} \Lambda^{-1}_{im} X^a_r \slashed\partial^\beta \uptheta^m_\tau \,dy.\label{E:HIGHORDERPOSTIPGAMMAGREATER5OVER3}
\end{align} 
Now compute the last integral on the left hand side of (\ref{E:HIGHORDERPOSTIPGAMMAGREATER5OVER3}) using integration by parts
\begin{align}
& \mu^{5-3\gamma} \int_\Omega  \psi \left( w^{1+\alpha+a} e^{\bar{S}} \Lambda_{ij} X^a_r \slashed\partial^\beta \left(\mathscr{A}^k_j\mathscr{J}^{-\frac1\alpha}-\delta^k_j\right)\right),_k \Lambda^{-1}_{im} \partial_\tau X^a_r \slashed\partial^\beta \uptheta^m \,dy \notag \\
& = \mu^{5-3\gamma} \int_\Omega  \psi \left( w^{1+\alpha+a} e^{\bar{S}} \Lambda_{ij} \left(- \mathscr{J}^{-\frac1\alpha}\mathscr{A}^k_\ell\mathscr{A}^s_j X_r^a \slashed\partial^\beta \uptheta^\ell,_s -\tfrac1\alpha \mathscr{J}^{-\frac1\alpha}\mathscr{A}^k_j\mathscr{A}^s_\ell \slashed\partial^\beta \uptheta^\ell,_s \right. \right. \notag \\
&\qquad \qquad \qquad \left. \left. + \mathcal C^{a,\beta,k}_j(\uptheta)\right) \right),_k \Lambda^{-1}_{im} \partial_\tau X^a_r \slashed\partial^\beta \uptheta^m \,dy \notag \\
&= -\mu^{5-3\gamma} \int_\Omega  \psi \left( w^{1+\alpha+a} e^{\bar{S}} \left( \mathscr{J}^{-\frac1\alpha}\mathscr{A}^k_\ell \left(\Lambda_{ij}\mathscr{A}^s_jX^a_r\slashed\partial^\beta\uptheta^\ell,_s - \Lambda_{\ell j}\mathscr{A}^s_jX^a_r\slashed\partial^\beta\uptheta^i,_s\right) \right. \right.\notag \\
&\left. \left. + \mathscr{J}^{-\frac1\alpha}\mathscr{A}^k_\ell\Lambda_{\ell j}\mathscr{A}^s_jX^a_r\slashed\partial^\beta\uptheta^i,_s + \tfrac1\alpha\mathscr{J}^{-\frac1\alpha}\Lambda_{ij}\mathscr{A}^k_j\mathscr{A}^s_\ell  X^a_r \slashed\partial^\beta \uptheta^\ell,_s \right) \right),_k \Lambda^{-1}_{im} \partial_\tau X^a_r \slashed\partial^\beta \uptheta^m \,dy \notag \\
&\qquad \qquad  +\mu^{5-3\gamma} \int_\Omega \psi \Lambda_{ij} \left(  w^{1+\alpha+a} e^{\bar{S}} \mathcal C^{a,\beta,k}_j \right),_k \Lambda^{-1}_{im}\partial_\tau X^a_r \slashed\partial^\beta \uptheta^m \,dy \notag \\
&=\mu^{5-3\gamma} \int_\Omega  \psi w^{1+\alpha+a} e^{\bar{S}} \mathscr{J}^{-\frac1\alpha} \left(\mathscr{A}^k_\ell  [\text{Curl}_{\Lambda\mathscr{A}}X^a_r \slashed\partial^\beta\uptheta]^\ell_i + \mathscr{A}^k_\ell \Lambda_{\ell j}[\nabla_\eta X^a_r \slashed\partial^\beta\uptheta]^i_j \right. \notag \\
&\qquad \qquad  \qquad + \left. \tfrac1\alpha \Lambda_{ij}\mathscr{A}^k_j \text{div}_\eta (X^a_r \slashed\partial^\beta \uptheta)\right) \Lambda^{-1}_{im}\partial_\tau \left(X^a_r \slashed\partial^\beta \uptheta^m\right),_k \,dy \notag \\
&+ \mu^{5-3\gamma} \int_\Omega \psi,_k w^{1+\alpha+a} e^{\bar{S}} \mathscr{J}^{-\frac1\alpha} \left(\mathscr{A}^k_\ell  [\text{Curl}_{\Lambda\mathscr{A}}X^a_r \slashed\partial^\beta\uptheta]^\ell_i + \mathscr{A}^k_\ell \Lambda_{\ell j}[\nabla_\eta X^a_r \slashed\partial^\beta\uptheta]^i_j \right. \notag \\ 
& \left. \qquad \qquad \qquad + \tfrac1\alpha \Lambda_{ij}\mathscr{A}^k_j \text{div}_\eta (X^a_r \slashed\partial^\beta \uptheta)\right) \Lambda^{-1}_{im}\partial_\tau X^a_r \slashed\partial^\beta \uptheta^m \, dy \notag \\
&+ \mu^{5-3\gamma} \int_\Omega \psi \Lambda_{ij} \left(  w^{1+\alpha+a} e^{\bar{S}} \mathcal C^{a,\beta,k}_j \right),_k \Lambda^{-1}_{im}\partial_\tau X^a_r \slashed\partial^\beta \uptheta^m \,dy \notag \\
&=:I_1+I_2+I_3. \label{E:HIGHORDERLASTINTEGRALIBPGAMMAGREATER5OVER3}
\end{align} 
\underline{Estimation of $I_1$:}  
\begin{align}
&\int_0^\tau I_1 \, d\tau' =  \int_0^\tau \mu^{5-3\gamma} \int_\Omega  \psi w^{1+\alpha+a} e^{\bar{S}} \mathscr{J}^{-\frac1\alpha} \mathscr{A}^k_\ell [\text{Curl}_{\Lambda\mathscr{A}}X^a_r \slashed\partial^\beta\uptheta]^\ell_i \Lambda^{-1}_{im}\partial_\tau \left(X^a_r \slashed\partial^\beta \uptheta^m\right),_k \,dy d\tau'  \notag \\
&+ \int_0^\tau \mu^{5-3\gamma} \int_\Omega  \psi w^{1+\alpha+a} e^{\bar{S}} \mathscr{J}^{-\frac1\alpha} \left( \mathscr{A}^k_\ell \Lambda_{\ell j}[\nabla_\eta X^a_r \slashed\partial^\beta\uptheta]^i_j + \tfrac1\alpha \Lambda_{ij}\mathscr{A}^k_j \text{div}_\eta (X^a_r \slashed\partial^\beta \uptheta) \right) \Lambda_{im}^{-1}\partial_\tau \left(X^a_r \slashed\partial^\beta \uptheta^m\right),_k \,dy d\tau'  \notag \\
&= \int_0^\tau \mu^{5-3\gamma} \int_\Omega  \psi w^{1+\alpha+a} e^{\bar{S}} \mathscr{J}^{-\frac1\alpha} \mathscr{A}^k_\ell \Lambda^{-1}_{im} [\text{Curl}_{\Lambda\mathscr{A}}X^a_r \slashed\partial^\beta\uptheta]^\ell_i \partial_\tau \left(X^a_r \slashed\partial^\beta \uptheta^m\right),_k \,dy d\tau'  \notag \\
&+ \int_0^\tau  \mu^{5-3\gamma} \int_\Omega  \psi w^{1+\alpha+a} e^{\bar{S}} \mathscr{J}^{-\frac1\alpha} \left( \Lambda_{\ell j}[\nabla_\eta X^a_r \slashed\partial^\beta \uptheta]^i_j \Lambda^{-1}_{im}[\nabla_\eta\partial_\tau X^a_r \slashed\partial^\beta \uptheta ]_\ell^m+\tfrac1\alpha \text{div}_\eta (X^a_r \slashed\partial^\beta\uptheta) \text{div}_\eta (\partial_\tau X^a_r \slashed\partial^\beta \uptheta) \right) \,dy d\tau' \notag \\ 
& =  \mu^{5-3\gamma} \int_\Omega \psi w^{1+\alpha+a} e^{\bar{S}} \mathscr{J}^{\frac1\alpha} \mathscr{A}^k_\ell\Lambda^{-1}_{im} [\text{Curl}_{\Lambda\mathscr{A}}X^a_r \slashed\partial^\beta\uptheta]^\ell_i \left(X^a_r \slashed\partial^\beta{\bf\uptheta}^m\right),_k \,dx\Big|^\tau_0 \notag \\
&  -\int_0^\tau \mu^{5-3\gamma} \int_\Omega \psi w^{1+\alpha+a} e^{\bar{S}}   \mathscr{J}^{\frac1\alpha} \mathscr{A}^k_\ell\Lambda^{-1}_{im}
[\text{Curl}_{\Lambda\mathscr{A}}X^a_r \slashed\partial^\beta{\bf V}]^\ell_i \left(X^a_r \slashed\partial^\beta{\bf\uptheta}^m\right),_k \,dx\,d\tau' \notag\\
&+(3\gamma-5) \int_0^\tau \mu^{4-3\gamma} \mu_{\tau}  \int_\Omega \psi w^{1+\alpha+a} e^{\bar{S}}   \mathscr{J}^{\frac1\alpha} \mathscr{A}^k_\ell\Lambda^{-1}_{im}
[\text{Curl}_{\Lambda\mathscr{A}}X^a_r \slashed\partial^\beta{\bf \uptheta}]^\ell_i \left(X^a_r \slashed\partial^\beta{\bf\uptheta}^m\right),_k \,dx\,d\tau' \notag \\
&  -\int_0^\tau \mu^{5-3\gamma} \int_\Omega \psi w^{1+\alpha+a} e^{\bar{S}}   \partial_\tau\left(\mathscr{J}^{\frac1\alpha} \mathscr{A}^k_\ell\Lambda^{-1}_{im}\right)
[\text{Curl}_{\Lambda\mathscr{A}}X^a_r \slashed\partial^\beta{\bf \uptheta}]^\ell_i \left(X^a_r \slashed\partial^\beta{\bf\uptheta}^m\right),_k \,dx\,d\tau'  \notag \\
&  -\int_0^\tau \mu^{5-3\gamma} \int_\Omega \psi w^{1+\alpha+a} e^{\bar{S}}   \mathscr{J}^{\frac1\alpha} \mathscr{A}^k_\ell\Lambda^{-1}_{im}
[\text{Curl}_{\Lambda_\tau\mathscr{A}}X^a_r \slashed\partial^\beta{\bf \uptheta}]^\ell_i \left(X^a_r \slashed\partial^\beta{\bf\uptheta}^m\right),_k \,dx\,d\tau' \notag \\
&- \int_0^\tau \mu^{5-3\gamma} \int_\Omega \psi w^{1+\alpha+a} e^{\bar{S}}   \mathscr{J}^{\frac1\alpha} \mathscr{A}^k_\ell\Lambda^{-1}_{im}
[\text{Curl}_{\Lambda\mathscr{A}_\tau}X^a_r \slashed\partial^\beta{\bf \uptheta}]^\ell_i \left(X^a_r \slashed\partial^\beta{\bf\uptheta}^m\right),_k \,dx\,d\tau' \notag \\
&+\int_0^\tau \mu^{5-3\gamma} \int_\Omega \psi  w^{1+\alpha+a} e^{\bar{S}}  \mathscr{J}^{-\frac1\alpha} \left( \Lambda_{\ell j}[\nabla_\eta X^a_r \slashed\partial^\beta \uptheta]^i_j \Lambda^{-1}_{im}\partial_\tau [\nabla_\eta X^a_r \slashed\partial^\beta \uptheta ]_\ell^m+\tfrac1\alpha\text{div}_\eta (X^a_r \slashed\partial^\beta \uptheta) \partial_\tau \text{div}_\eta ( X^a_r \slashed\partial^\beta \uptheta) \right)dy d\tau' \notag \\
&-\int_0^\tau \mu^{5-3\gamma} \int_\Omega \mathscr{J}^{-\frac1\alpha} \left( \Lambda_{\ell j}[\nabla_\eta X^a_r \slashed\partial^\beta \uptheta]^i_j \Lambda^{-1}_{im}\partial_\tau\mathscr{A}^k_\ell \left(X^a_r \slashed\partial^\beta \uptheta^m\right),_k+ \tfrac1\alpha\text{div}_\eta (X^a_r \slashed\partial^\beta\uptheta)  \partial_\tau \mathscr{A}^k_j \left(X^a_r \slashed\partial^\beta \uptheta^j\right),_k  \right) d\tau' \notag \\
&:=B_1+B_2+B_3+\mathcal{R}_1+\mathcal{R}_2+\mathcal{R}_3+E_1+\mathcal{R}_4,\label{E:HIGHORDERI1EXPANDGAMMAGREATER5OVER3}
\end{align}
where we have used integration by parts. For $B_3$, which is a term unique to the $\gamma > \frac53$ case,
\begin{align}
&(3\gamma-5) \int_0^\tau \mu^{4-3\gamma} \mu_{\tau}  \int_\Omega \psi w^{1+\alpha+a} e^{\bar{S}}   \mathscr{J}^{\frac1\alpha} \mathscr{A}^k_\ell\Lambda^{-1}_{im}
[\text{Curl}_{\Lambda\mathscr{A}}X^a_r \slashed\partial^\beta{\bf \uptheta}]^\ell_i \left(X^a_r \slashed\partial^\beta{\bf\uptheta}^m\right),_k \,dx\,d\tau' \notag \\
&=(3\gamma-5) \int_0^\tau \mu^{4-3\gamma} \mu_{\tau} \frac{1}{\mu^{\tfrac{5-3\gamma}{2}}} \int_\Omega \psi \mu^{\tfrac{5-3\gamma}{2}} w^{1+\alpha+a} e^{\bar{S}}   \mathscr{J}^{\frac1\alpha} \mathscr{A}^k_\ell\Lambda^{-1}_{im}
[\text{Curl}_{\Lambda\mathscr{A}}X^a_r \slashed\partial^\beta{\bf \uptheta}]^\ell_i \left(X^a_r \slashed\partial^\beta{\bf\uptheta}^m\right),_k \,dx\,d\tau' \notag \\
&\lesssim  \left(\kappa \sup_{0 \leq \tau \leq \tau'}  \mu^{5-3\gamma} \| \nabla_\eta X_r^a \slashed\partial^\beta \uptheta \|_{1+\alpha+a,\psi e^{\bar{S}}}^2 + \frac{1}{\kappa} \sup_{0 \leq \tau \leq \tau'}   \mu^{5-3\gamma} \|\text{Curl}_{\Lambda\mathscr{A}}X^a_r \slashed\partial^\beta\uptheta\|^2_{1+\alpha+a,\psi e^{\bar{S}}}\right) \int_0^\tau \mu^{\tfrac{3-3\gamma}{2}} \mu_{\tau} d\tau' \notag \\  
&\lesssim \kappa \tilde{\mathcal{S}}^N(\tau) +\tilde{ \mathcal{B}}^N[\uptheta](\tau),
\end{align}
where we have used the Young inequality with $\varepsilon$ and (\ref{E:EXPMU1MUINEQGAMMALEQ5OVER3}). For $E_1$
\begin{align}
E_1&=\frac12 \int_0^\tau \frac{d}{d\tau}\left\{\mu^{5-3\gamma} \int_\Omega \psi  w^{1+\alpha+a} e^{\bar{S}}  \mathscr{J}^{-\frac1\alpha}\left(\sum_{i,j=1}^3 d_id_j^{-1}(\mathscr M_{a,\beta})^j_{i})^2 + \frac1\alpha\left(\text{div}_\eta X^a_r \slashed\partial^\beta \uptheta\right)^2\right)\,dy\right\}\,d\tau' \notag\\
&+\int_0^\tau \frac{3\gamma-5}{2}\mu^{5-3\gamma}\frac{\mu_\tau}{\mu} \int_\Omega   \psi\mathscr{J}^{-\frac1\alpha}\Big[\sum_{i,j=1}^3 d_id_j^{-1}\left((\mathscr{M}_{a,\beta})^j_{i}\right)^2+\tfrac1\alpha\left(\text{div}_\eta X_r^a\slashed\partial^\beta \uptheta\right)^2\Big] w^{a+\alpha+1}e^{\bar{S}} \, dy \, d\tau' \notag \\
&+ \int_0^\tau \mu^{5-3\gamma} \int_\Omega \psi  w^{1+\alpha+a} e^{\bar{S}}   \mathscr{J}^{-\frac1\alpha} \mathcal T_{a,\beta} dy \,d\tau' \notag \\
& - \frac12 \int_0^\tau \mu^{5-3\gamma} \int_\Omega \psi  w^{1+\alpha+a} e^{\bar{S}}  \partial_\tau\left(\mathscr{J}^{-\frac1\alpha}\right)\left(\sum_{i,j=1}^3 d_id_j^{-1}(\mathscr M_{a,\beta})^j_{i})^2 + \frac1\alpha\left(\text{div}_\eta X^a_r \slashed\partial^\beta \uptheta\right)^2\right)\,dy \,d\tau' \notag \\
&:=E_2+\mathcal{D}_1 + \mathcal{R}_5 + \mathcal{R}_6, \label{E:HIGHORDERE1EXPANDGAMMAGREATER5OVER3}
\end{align}
where we have used Lemma \ref{L:KEYLEMMA}. Then $\mathcal{D}_1$ contributes to $\int_0^\tau \tilde{\mathcal{D}}^N (\tau ') \, d\tau'$ in (\ref{E:ENERGYMAINGAMMAGREATER5OVER3}). This concludes our energy inequality proof outline in the $\gamma > \frac53$ case.

\section*{Acknowledgments}
The authors thank the anonymous referees for their helpful comments that have improved the presentation of the paper. C. Rickard  and J. Jang acknowledge the support of the NSF grant DMS-1608494. 
M. Had\v zi\'c acknowledges the support of the EPSRC grant 
EP/N016777/1. 

\appendix
\renewcommand{\theequation}{\Alph{section}.\arabic{equation}}
\setcounter{theorem}{0}\renewcommand{\theorem}{\Alph{section}.\??arabic{prop}}

\section{$\tau$ based inequalities}
We have the following useful $\tau$ based inequalities, summarized by the Lemma below.
\begin{lemma}\label{L:USEFULTAULEMMAGAMMALEQ5OVER3}
Suppose $\gamma > 1$. Fix an affine motion $A(t)$ from the set $\mathscr{S}$ under consideration, namely require
\begin{equation}
\det A(t) \sim 1 + t^3, \quad t \geq 0.
\end{equation}  
Let
\begin{equation}
\mu_1:=\lim_{\tau \rightarrow \infty} \frac{\mu_\tau(\tau)}{\mu(\tau)}, \quad \mu_0:=\frac{d(\gamma)}{2}\mu_1,
\end{equation}          
where $\mu(\tau)=\det A(\tau)$ and $d(\gamma)=\begin{cases}
3\gamma - 3 & \text{ if } \  1<\gamma\leq \frac53 \\ 
2 & \text{ if } \  \gamma>\frac53
\end{cases}.$ \\ \\
Then we have the following properties
\begin{align}
0<\mu_0&=\mu_0(\gamma)\le\mu_1, \label{E:MU0MU1INEQGAMMALEQ5OVER3}\\
e^{\mu_1\tau} \lesssim & \ \mu(\tau)  \lesssim  e^{\mu_1\tau}, \ \ \tau\ge0, \label{E:EXPMU1MUINEQGAMMALEQ5OVER3}\\
\sum_{a+|\beta|\le N}\|X_r^a \slashed\partial^\beta{\bf V}\|_{1+\alpha+a,\psi e^{\bar{S}}} &+ \sum_{|\nu|\leq N} \|\partial^\nu {\bf V}\|_{1+\alpha,(1-\psi)e^{\bar{S}}} \lesssim e^{-\mu_0\tau}\mathcal{S}^N(\tau)^{\frac12} \label{E:EXPSQRTSNBOUNDGAMMALEQ5OVER3}\\
\| \Lambda_\tau \| \lesssim e^{-\mu_1 \tau}, & \quad  \| \Lambda \| + \| \Lambda^{-1} \| \le C,  \label{E:LAMBDABOUNDSGAMMALEQ5OVER3}\\
\sum_{i=1}^3\left(d_i+\frac{1}{d_i}\right) &\le C \label{E:EIGENVALUESBOUNDGAMMALEQ5OVER3} \\
\sum_{i=1}^3|\partial_\tau d_i| + \|\partial_\tau P\| &\lesssim e^{-\mu_1\tau} \label{E:EIGENVALUEPLUSPBOUNDGAMMALEQ5OVER3}\\
|{\bf w}|^2 \lesssim & \ \langle \Lambda^{-1}{\bf w}, {\bf w}\rangle  \lesssim |{\bf w}|^2 , \ {\bf w}\in \mathbb R^3, \label{E:LAMBDAINVERSEIPINEQGAMMALEQ5OVER3}
\end{align}
for $C > 0$.       
\end{lemma} 
\begin{proof}
The result (\ref{E:MU0MU1INEQGAMMALEQ5OVER3}) is clear from the definition of $\mu_0$. For inequalities (\ref{E:EXPMU1MUINEQGAMMALEQ5OVER3}) and (\ref{E:LAMBDABOUNDSGAMMALEQ5OVER3}) through (\ref{E:LAMBDAINVERSEIPINEQGAMMALEQ5OVER3}) we first note that by Lemma \ref{L:AASYMPTOTICS}, there exist matrices $A_0,A_1,M(t)$ such that          
\begin{align}  
A(t) = A_0 + t A_1 + M(t), \quad t \geq 0.
\end{align} 
where $A_0,A_1$ are time-independent and $M(t)$ satisfies the bounds
\begin{align} 
\|M(t)\| = o_{t \rightarrow \infty}(1+t), \ \ \|\partial_t M(t)\| \lesssim (1+t)^{3-3\gamma}.
\end{align}          
We also note $\det A(t) \sim 1 + t^3$. Then inequalities (\ref{E:EXPMU1MUINEQGAMMALEQ5OVER3}) and (\ref{E:LAMBDABOUNDSGAMMALEQ5OVER3}) through (\ref{E:LAMBDAINVERSEIPINEQGAMMALEQ5OVER3}) follow from Lemma A.1 \cite{1610.01666}. Finally, (\ref{E:EXPSQRTSNBOUNDGAMMALEQ5OVER3}) follows from the definition of $\mathcal{S}^N$ (\ref{E:SNNORMGAMMALEQ5OVER3}) and properties (\ref{E:MU0MU1INEQGAMMALEQ5OVER3})-(\ref{E:EXPMU1MUINEQGAMMALEQ5OVER3}) above.
\end{proof}

\section{Derivative Operators}
Our derivative operators $\slashed\partial_{ji}$ and $X_r$ satisfy the following identities:
\begin{lemma}\label{L:DERIVATIVEIDENTITIES}
For $i\in\{1,2,3\}$ we have the decomposition
\begin{equation}\label{E:DERIVATIVEDECOMP}
\partial_i = \frac{y_j}{r^2}\slashed\partial_{ji}+\frac{y_i}{r^2}X_r.
\end{equation}
For $i,j,k,m\in\{1,2,3\},$ we have the commutator identities
\begin{equation}\label{E:DERIVATIVECOMMUTATOR}
[\slashed\partial_{ji},X_r]=0, \quad [\slashed\partial_{ji},\slashed\partial_{ik}]=\slashed\partial_{jk}, \quad [\partial_m,X_r]= \partial_m, \quad [\partial_m,\slashed\partial_{ji}]= \delta_{mj}\partial_i - \delta_{mi}\partial_j.
\end{equation}
\end{lemma}
\begin{proof}
These properties are straightforward consequences of the definitions introduced in (\ref{E:DERIVATIVEDEFN}) \cite{hadzic2017class}.
\end{proof}
The following Lemma will help when differentiating radial functions.
\begin{lemma}\label{L:DERIVATIVERADIAL}
Given a radial function $f:\Omega \rightarrow \mathbb{R}$, say $f(|y|)$ with derivatives with respect to $|y|$ denoted by the standard prime notation, and $a+|\beta| \leq N$, $k \in \{1,2,3\}$, we have
\begin{align}
X_r\left(\frac{y_k}{|y|}\right)&=0, \label{E:RADIALDERVIATIVEFRACYKMAGY}\\
X_r\left(\frac{y_k}{|y|^2}\right)&=-\frac{y_k}{|y|^2}, \label{E:RADIALDERVIATIVE2FRACYKMAGY}\\
X_r^a f(|y|) &= \sum_{j=1}^a p_j(|y|)f^{(j)}(|y|), \label{E:RADIALDERVIATIVERADIALFCT} \\
&\text{ where } p_j \text{ is some polynomial with max degree } j, \notag \\
\slashed\partial_{ji} f(|y|) &= 0, \label{E:ANGULARDERIVATIVERADIALFCT} \\
\slashed\partial^\beta \left(\frac{y_k}{|y|}\right)&=
\begin{cases}
y_{\ell}/|y| \\
0,
\end{cases}  \label{E:ANGULARDERIVATIVEFRACYKMAGY} \\
&\text{ for some } \ell \in \{1,2,3\}, \notag \\
\slashed\partial^\beta \left(\frac{y_k}{|y|^2}\right)&=
\begin{cases}
y_{\ell}/|y|^2 \\
0
\end{cases},  \label{E:ANGULARDERIVATIVE2FRACYKMAGY} \\
&\text{ for some } \ell \in \{1,2,3\}. \notag
\end{align}
\end{lemma}
\begin{proof}
Both (\ref{E:RADIALDERVIATIVEFRACYKMAGY}) and (\ref{E:ANGULARDERIVATIVERADIALFCT}) are straightforward applications of the definitions introduced in (\ref{E:DERIVATIVEDEFN}). Now (\ref{E:RADIALDERVIATIVERADIALFCT}) follows from induction using the following results
\begin{equation*}
X_r f(|y|)=|y|f'(|y|), \quad X_r (|y|^n)=n|y|^n \text{  for } n \in \mathbb{Z}_{\geq 0}.
\end{equation*}
Next (\ref{E:ANGULARDERIVATIVEFRACYKMAGY}) follows easily from the following facts, the first of which is a straightforward calculation,
\begin{equation*}
\slashed\partial_{ji}\left(\frac{y_k}{|y|}\right) = \frac{\delta_{jk}y_i-\delta_{ik}y_j}{|y|}, \quad \slashed\partial_{ii} = 0 \text{  for } i,j \in {1,2,3}.
\end{equation*} 
Similar arguments give (\ref{E:RADIALDERVIATIVE2FRACYKMAGY}) and (\ref{E:ANGULARDERIVATIVE2FRACYKMAGY}) to conclude the proof.
\end{proof}
\section{Weighted Sobolev-Hardy Inequality}
By (\ref{E:WDEMAND}), $w$ behaves like a distance function. Then from Proposition C.2 \cite{1610.01666}, and also using Corollary \ref{C:ENTROPYREGULARITYCOROLLARY}, we have the following modified weighted Sobolev-Hardy inequality:
\begin{lemma}\label{L:EMBEDDING}  
For any $u\in C^\infty(B_1({\bf 0}))$, we have 
\begin{align}
&\sup_{B_1({\bf 0})\setminus B_{\frac14}({\bf 0})}\left| w^{\frac{a_1}{2}} e^{\bar{S}}X_r^{a_1} \slashed\partial^{\beta_1} u \right| \notag\\
& \lesssim \sum_{a+|\beta|\leq a_1+|\beta_1|+\lceil \alpha\rceil+6} \| X_r^a \slashed\partial^{\beta} u \|_{a+\alpha,\psi e^{\bar{S}}} 
+ \sum_{|\nu|\leq a_1+|\beta_1| +2 } \| \partial^\nu u\|_{\alpha,(1-\psi)e^{\bar{S}}} ,\label{E:EMBEDDING1} \\
&\sup_{B_1({\bf 0})\setminus B_{\frac14}({\bf 0})}\left| w^{\frac{a_1}{2}} e^{\bar{S}} D X_r^{a_1}\slashed\partial^{\beta_1} u \right|\notag \\
& \lesssim \sum_{a+|\beta|\leq a_1+|\beta_1|+\lceil \alpha\rceil+6} \| \nabla_\eta X_r^a \slashed\partial^\beta u \|_{a+\alpha+1,\psi e^{\bar{S}}} 
+ \sum_{|\nu|\leq a_1+|\beta_1|+2} \|\nabla_\eta\partial^\nu u\|_{\alpha+1,(1-\psi)e^{\bar{S}}}. \label{E:EMBEDDING2}     
\end{align}
\end{lemma}

\section{Scaling Analysis and Affine Motion}\label{A:SCALING}
An alternative derivation of the nonisentropic affine motion discussed in Section \ref{S:AFF} is available directly from the structure of the original Eulerian equations.

First for any given $(\rho,{\bf u },S,p)$ and $A\in \text{GL}^+(3)$, consider the following mass critical transformation
\begin{align}
\rho(t,x)&=(\det A)^{-1}\hat{\rho}(s,y), \label{E:TRANSFORMPRE1} \\
\mathbf{u}(t,x)&=(\det A)^{\frac{1-3\gamma}{6}} A\, \hat{\mathbf{u}}(s,y), \label{E:TRANSFORMPRE2}  \\
S(t,x)&=\hat{S}(s,y), \\
p(t,x)&=(\det A)^{-\gamma}\hat{p}(s,y), \\
\frac{ds}{dt}&=(\det A)^{\frac{1-3\gamma}{6}}, \\
y&=A^{-1}x. \label{E:TRANSFORMPRE3}
\end{align}
If $(\rho,{\bf u},S,p)$ solve (\ref{E:MOM})-(\ref{E:M}) and (\ref{E:EEOS})-(\ref{E:S}) in $\Omega(t)$, then $(\hat{\rho},\hat{{\bf u}},\hat{S},\hat{p})$ solve
\begin{align}
\partial_s\hat\rho + \text{div}\, (\hat\rho \hat{\mathbf{u}})& = 0 \label{E:CONTINUITYG}  \\
\hat\rho (\partial_s \mathbf{\hat{u}} +\hat{\mathbf{u}} \cdot \nabla \hat{\mathbf{u}})  +\Lambda\nabla (\hat{p}) &= 0 \label{E:VELOCITYG} \\
\partial_s \hat{S} + \hat{\mathbf{u}} \cdot \nabla \hat{S} & = 0 \\
\hat{p}&=\hat{\rho}^{\gamma}e^{\hat{S}},
\end{align}
in $\hat{\Omega}(s)=A^{-1}\Omega(s)$ with $\Lambda = (\det A)^{\frac23}A^{-1}A^{-\top}$. 

Motivated by the nearly invariant transformation (\ref{E:TRANSFORMPRE1})-(\ref{E:TRANSFORMPRE3}), we are looking for a path
$$ \mathbb R_+\ni t\mapsto A(t) \in \text{GL}^+(3), $$
and consider the transformation
\begin{align}
\rho(t,x)&=(\det A(s))^{-1}\tilde{\rho}(s,y), \label{E:TRANSFORMPRE12} \\
\mathbf{u}(t,x)&=(\det A(s))^{\frac{1-3\gamma}{6}} A(s)\, \tilde{\mathbf{u}}(s,y), \label{E:TRANSFORMPRE22}  \\
S(t,x)&=\tilde{S}(s,y), \label{E:TRANSFORMPRE32} \\
p(t,x)&=(\det A(s))^{-\gamma}\tilde{p}(s,y), \\
\frac{ds}{dt}&=(\det A(t))^{\frac{1-3\gamma}{6}}, \\
y&=A(t)^{-1}x.
\end{align}
We seek $A(t)$ such that this transformation solves the vacuum free boundary nonisentropic Euler system (\ref{E:MOM})-(\ref{E:IC}). Introduce
$$ \mu(s) : = \det A(s)^{\frac13}, \qquad  B(s):=-A^{-1}A_s.$$
Then (\ref{E:MOM})-(\ref{E:M}) and (\ref{E:EEOS})-(\ref{E:S}) can be written as
\begin{align}
\tilde{\rho} \left(\partial_s \tilde{\bf u} + \frac{1-3\gamma}{2} \frac{\mu_s}{\mu}\tilde{\bf u} - B \tilde{\bf u} + B (y \cdot \nabla) \mathbf{\tilde{u}} + (\tilde{\mathbf{u}}\cdot\nabla) \tilde{\mathbf{u}}\right)+ \Lambda \nabla(\tilde{p}) &= 0, \label{E:TRANSFORMMOM} \\
\partial_s \tilde{\rho} - 3 \frac{\mu_s}{\mu} \tilde{\rho} + B y \cdot \nabla \tilde{\rho} + \text{div} (\tilde{\rho} \tilde{\mathbf{u}})&=0, \label{E:TRANSFORMM} \\
\partial_s \tilde{S} + \tilde{\bf u}\cdot \nabla \tilde{S} + By \cdot \nabla \tilde{S}&=0, \label{E:TRANSFORMS} \\
\tilde{p}&=\tilde{\rho}^\gamma e^{\tilde{S}}. \label{E:TRANSFORMP}
\end{align}
Next make the important change of variables
\begin{equation}
{\bf U}(s,y)  : = \tilde{\bf u}(s,y) +B(s) y.
\end{equation}
Notice that
\begin{align*}
{\bf U}\cdot \nabla \bf U &= \tilde{\mathbf{u}} \cdot \nabla \tilde{\mathbf{u}} + B y \cdot \nabla \tilde{\mathbf{u}} + B \tilde{\bf u} + B^2 y, \\
\partial_s {\bf U} &= \partial_s \tilde{\mathbf{u}} + B_s y, \\
\text{div} (B y)&=-\text{Tr}(A^{-1}A_s)=-\frac{\partial_s \det A(s)}{\det A(s)}=-3\frac{\mu_s}{\mu}.
\end{align*}
Then (\ref{E:TRANSFORMMOM})-(\ref{E:TRANSFORMS}) can be expressed as
\begin{align}
\partial_s{\bf U} +({\bf U}\cdot\nabla){\bf U}+\left(\frac{1-3\gamma}{2}\frac{\mu_s}{\mu}\,\text{{\bf Id}}-2B\right){\bf U} &-\left(B_s-B^2+B \frac{1-3\gamma}{2} \frac{\mu_s}{\mu} \right) y + \Lambda \frac{1}{\tilde{\rho}} \nabla(\tilde{p})=0, \label{E:CHANGEOFVARIABLEMOM} \\
\partial_s\tilde\rho +  \text{div}\,\left(\tilde\rho{\bf U}\right)  &= 0, \label{E:CHANGEOFVARIABLEM} \\
\partial_s \tilde{S} + \mathbf{U} \cdot \nabla \tilde{S} &=0. \label{E:CHANGEOFVARIABLES}
\end{align}
We now seek a special solution of (\ref{E:CHANGEOFVARIABLEMOM})-(\ref{E:CHANGEOFVARIABLES}) and (\ref{E:TRANSFORMP}) in $\Omega=B_1(\mathbf{0})$ by setting $\mathbf{U}=\mathbf{0}$. Then (\ref{E:CHANGEOFVARIABLEM})-(\ref{E:CHANGEOFVARIABLES}) reduce to $\partial_s \tilde\rho = \partial_s \tilde{S} = 0$ which implies $\tilde{\rho}$ and $\tilde{S}$ are only $y$-dependent. Furthermore by (\ref{E:TRANSFORMP}), $\tilde{p}$ is also only $y$-dependent.

Now for (\ref{E:CHANGEOFVARIABLEMOM}) to hold, we need to solve
\begin{equation}\label{E:NEEDTOSOLVE}
\left( B_s-B^2+\frac{1-3\gamma}{2} \frac{\mu_s}{\mu}B \right) y = \Lambda \frac{1}{\tilde{\rho}}\nabla \tilde{p}.
\end{equation}
Let
\begin{equation}\label{E:PRERHOPODE}
\frac{1}{\tilde{\rho}}\nabla \tilde{p}=-\delta y \quad \text{in} \quad \Omega.
\end{equation}
By a simple argument, this equation implies $\tilde{p}$ and $\tilde{\rho}$ are radial functions. Thus $\tilde{S}$ is also radial. Now (\ref{E:NEEDTOSOLVE}) will be satisfied if $B$ solves the ODE system
\begin{equation}\label{E:BODE}
B_s-B^2+\frac{1-3\gamma}{2} \frac{\mu_s}{\mu}B =-\delta \Lambda.
\end{equation}
Now using the definition of $B$ and $\frac{dt}{ds}=(\det A(t))^{-\tfrac{1-3\gamma}{6}}$, we have
\begin{align*}
B_s &= -\det A^{\frac{3\gamma-1}{6}} \partial_t(A^{-1} A_t\det{A}^{\frac{3\gamma-1}{6}}) \\
& = \det{A}^{\frac{3\gamma-1}3}A^{-1} A_t A^{-1}A_t - \det{A}^{\frac{3\gamma-1}3} A^{-1}A_{tt} -\frac{3\gamma-1}{6} \det{A}^{\frac{3\gamma-1}3} A^{-1} A_{t}
\frac{\partial_t\det{A}}{\det{A}}\\
& = B^2 +\frac{3\gamma-1}{2} \frac{\mu_s}{\mu} B - \det{A}^{\frac{3\gamma-1}3} A^{-1} A_{tt}.
\end{align*}
Thus recalling the definition of $\Lambda$, (\ref{E:NEEDTOSOLVE}) is equivalent to the following ODE system for $A$
\begin{align*}
- \det{A}^{\frac{3\gamma-1}3} A^{-1}A_{tt} = -\delta \det{A}^{\frac23}A^{-1}A^{-\top}.
\end{align*}
which can be written exactly as the affine fundamental system
\begin{equation}\label{E:SIDERISODE}
A_{tt} = \delta \det{A}^{1-\gamma} A^{-\top}.
\end{equation}
Finally since $\tilde{\rho}$ and $\tilde{p}$ are radial, (\ref{E:PRERHOPODE}) gives the fundamental affine ODE
\begin{equation}\label{E:FUNDAMENTALAFFINEODES}
\tilde{p}'(r)=-\delta r \tilde{\rho}(r).
\end{equation}
Then using (\ref{E:TRANSFORMP}) and our vacuum condition, from (\ref{E:FUNDAMENTALAFFINEODES}) we can obtain the explicit affine entropy formula for $e^{\tilde{S}}$ in terms of $\tilde{\rho}$
\begin{equation}\label{E:AFFENTROPYS}  
e^{\tilde{S}}(r) =  \frac{ \delta \int_{r}^1 \ell \tilde{\rho} (\ell) \, d \ell }{(\tilde{\rho}(r))^\gamma}.
\end{equation}

\end{document}